\documentclass[a4paper, 11pt]{amsart}
\usepackage{graphicx, euscript, enumitem, kantlipsum}
\usepackage{marginnote}
\usepackage{amssymb}
\usepackage{amsmath}
\usepackage{amsthm}
\usepackage{amscd, float}
\usepackage{mathbbol}
\usepackage{dsfont}
\usepackage[all,2cell]{xy}

\usepackage{tikz}
\usetikzlibrary{cd}
\usetikzlibrary{calc}
\usetikzlibrary{decorations,decorations.pathreplacing,decorations.markings,decorations.pathmorphing}
\tikzstyle{snake}=[decorate, decoration={snake, segment length=1mm, amplitude=.5mm}]
\newcommand{\tikzmath}[2][]
{\vcenter{\hbox{\scalebox{0.7}{\begin{tikzpicture}[#1]#2\end{tikzpicture}}}}
}
\tikzset{super thick/.style={line width=3pt}}
\tikzstyle{knot}=[preaction={super thick, white, draw}]
\usepackage{geometry}
\usepackage[colorlinks=true,pagebackref,hyperindex]{hyperref}
\hypersetup{colorlinks = true, linkcolor = red, anchorcolor =red, citecolor = blue, filecolor = red, urlcolor = blue, pdfauthor=author}

\UseAllTwocells \SilentMatrices
\newtheorem{thm}{Theorem}[section]

\newtheorem{cor}[thm]{Corollary}
\newtheorem{lem}[thm]{Lemma}
\newtheorem{exm}[thm]{Example}

\newtheorem{thmy}{Theorem}

\newtheorem{prop}[thm]{Proposition}
\theoremstyle{definition}
\newtheorem{defn}[thm]{Definition}
\theoremstyle{remark}
\newtheorem{rem}[thm]{\bf Remark}
\numberwithin{equation}{section}

\newcommand{\Ext}{\mathrm{Ext}}

\newcommand{\Y}{\mathcal{Y}}

\newcommand{\LL}{\mathrm{l}}
\newcommand{\RR}{\mathrm{r}}
\newcommand{\BB}{\mathrm{bi}}

\newcommand{\Hom}{\mathrm{Hom}}

\newcommand{\op}{\mathrm{op}}
\newcommand{\opp}{\mathrm{opp}}

\newcommand{\id}{\mathbf{1}}
\newcommand{\EC}{\EuScript{C}}
\newcommand{\ED}{\EuScript{D}}

\newcommand{\EK}{\EuScript{K}}

\newcommand{\g}{\widetilde{g}}

\makeatletter
\newcommand{\smallotimes}{\mathbin{\mathpalette\make@small\otimes}}

\newcommand{\make@small}[2]{%
  \vcenter{\hbox{%
    $\m@th\ifx#1\displaystyle\scriptstyle\else\ifx#1\textstyle\scriptstyle
     \else\scriptscriptstyle\fi\fi#2$%
  }}%
}
\makeatother

\begin{document}

\title[Tensor products from $B_\infty$-algebras with applications to Hopf algebras]{Monoidal structures arising from $B_\infty$-algebras with applications to Hopf algebras}

\author{Gongxiang Liu}
	\address{School of Mathematics, Nanjing University, Nanjing 210093, Jiangsu, China}
	\email{gxliu@nju.edu.cn}
\author{Zhengfang Wang}
	\address{School of Mathematics, Nanjing University, Nanjing 210093, Jiangsu, China}
	\email{zhengfangw@nju.edu.cn}
\author{Mengdie Zhang}
	\address{School of Mathematics, Nanjing University, Nanjing 210093, Jiangsu, China}
	\email{mengdiezhang@smail.nju.edu.cn}

\subjclass[2010]{16E05, 13D03, 16G20, 16E60}

\keywords{$B_\infty$-algebra, $A_\infty$-module, derived category, monoidal triangulated category, Hopf algebra}

\date{\today}

\begin{abstract}
We give an explicit construction of monoidal structures on derived categories of right $A_\infty$-modules over an $A_\infty$-algebra $A$ equipped with a $B_\infty$-structure. Given such a $B_\infty$-algebra $A$, we construct an induction functor
\[
\iota\colon \ED^{\RR}_\infty(A) \longrightarrow \ED^{\BB}_\infty(A)
\]
from right $A_\infty$-modules to $A_\infty$-bimodules and define
\[
M\boxtimes_A N=M\overset{\infty}{\otimes}_A\iota(N).
\]
We prove that $(\ED^{\RR}_\infty(A),\boxtimes_A,A)$ is a monoidal triangulated category. The proof is entirely $A_\infty$-algebraic: the unit and associativity constraints are induced by explicit quasi-isomorphisms of $A_\infty$-bimodules, including
\[
\iota(A)\simeq A
\qquad\text{and}\qquad
\iota(M)\overset{\infty}{\otimes}_A\iota(N)\simeq \iota(M\boxtimes_A N).
\]
For brace $B_\infty$-algebras, we further prove that every brace $B_\infty$-module is $A_\infty$-quasi-isomorphic to its induced bimodule.

We apply this construction to finite-dimensional Hopf algebras. For a Hopf algebra $H$, the Yoneda dg algebra $\Y(\Bbbk,\Bbbk)$ of the trivial $H$-module carries a natural brace $B_\infty$-structure, and hence its derived category carries the monoidal structure constructed above. We show that the Koszul duality functor
\[
\Hom_H(\Y(H,\Bbbk),-)\colon \EK(\mathrm{Inj}\text{-}H)\longrightarrow \ED(\Y(\Bbbk,\Bbbk))
\]
is triangulated lax monoidal, and that its restriction to the localizing subcategory generated by the injective resolution $\Y(H,\Bbbk)$ is a monoidal triangulated equivalence. If $H$ is local, this localizing subcategory is all of $\EK(\mathrm{Inj}\text{-}H)$. In particular, this gives an alternative, purely algebraic proof of
the monoidal equivalence conjectured by Krause and established by
Benson--Krause through the classifying space $BG$. Finally, our examples recover the usual tensor product for graded commutative algebras and show that the resulting brace $B_\infty$ and monoidal structures can depend essentially on the chosen Hopf structure.
\end{abstract}

\maketitle
\vspace{-3em}
\tableofcontents

\section{Introduction}
Tensor triangular geometry, introduced by Balmer \cite{Bal05, Bal10}, has become an important language in many areas of mathematics, including $K$-theory, algebraic geometry, algebraic topology, and commutative algebra; see, for example, \cite{BIK08, Big, MT, Ste}. Noncommutative analogues were studied in \cite{NVY01, NVY02}.

The basic commutative example is the derived category of a commutative ring $R$, with tensor product given by the derived tensor product $-\otimes_R^{\mathbb L}-$. For a noncommutative algebra, there is no such tensor product on the category of right modules alone: to tensor over $A$, the second factor must also carry a compatible left $A$-action. Thus, a natural problem is to find conditions under which a category of right modules over a noncommutative algebra carries a canonical monoidal triangulated structure.

The purpose of this article is to solve this problem for $A_\infty$-algebras equipped with a $B_\infty$-structure. A $B_\infty$-algebra may be viewed as an $A_\infty$-algebra whose bar construction is a dg bialgebra, or informally as a bialgebra object up to coherent homotopy. Such structures appear naturally in deformation theory, string topology, and the study of Hochschild cochains \cite{Ba, GeVo, GJ, Kel03, LoVa1, Vor}. They also play a central role in the Deligne conjecture for Hochschild cohomology; see, for example, \cite{KoSo, MS, Ta, Vor}.

The first key point of the article is that a $B_\infty$-structure on $A$ produces, for every right $A_\infty$-module $N$, a canonical induced $A_\infty$-bimodule $\iota(N)$. 
Precisely, the antipode of the dg bialgebra $T^c(sA)$ gives a dg coalgebra morphism
\[
T^c(sA)^{\op}\otimes T^c(sA)\longrightarrow T^c(sA),
\]
and hence a dg functor
\[
\iota\colon \EC^{\RR}_\infty(A)\longrightarrow \EC^{\BB}_\infty(A)
\]
from right $A_\infty$-modules to $A_\infty$-bimodules; see Proposition \ref{prop:dualitybinfinity}. Passing to derived categories gives
\[
\iota\colon \ED^{\RR}_\infty(A)\longrightarrow \ED^{\BB}_\infty(A).
\]
We then use this induced bimodule structure on the second factor to define, for right $A_\infty$-modules $M$ and $N$,
\[
M\boxtimes_A N=M\overset{\infty}{\otimes}_A\iota(N).
\]

Our first main result is the following.

\begin{thmy}[= Theorem \ref{thm:monoidalcatofrightmod}] \label{thmA}
Let $A$ be a $B_\infty$-algebra.
Then the derived category $\ED^{\RR}_\infty(A)$ of right $A_\infty$-modules over $A$ is a monoidal triangulated category with unit object $A$.
The corresponding tensor product is
\[
M \boxtimes_A N = M \overset{\infty}{\otimes}_A \iota(N),
\]
where $\iota\colon \ED^{\RR}_\infty(A) \to \ED^{\BB}_\infty(A)$ is the induction functor described above.
\end{thmy}

The proof of Theorem \ref{thmA} is not merely a formal transfer of the monoidal structure on the derived category of $A_\infty$-bimodules. Indeed, the induction functor $\iota$ is faithful but is not full in general, so the monoidal axioms for $\boxtimes_A$ must be checked directly. The main technical work consists of constructing explicit $A_\infty$-bimodule quasi-isomorphisms
\[
\iota(A)\simeq A
\qquad\text{and}\qquad
\iota(M)\overset{\infty}{\otimes}_A\iota(N)\simeq \iota(M\boxtimes_A N),
\]
which provide the unit and associativity constraints. This explicit approach is one of the main features of the article. Although related monoidal structures can often be obtained abstractly using higher-categorical methods \cite{Lu}, our construction is given by concrete $A_\infty$-formulas and verified by graphical calculus.

A second key technical result concerns brace $B_\infty$-algebras. These are $B_\infty$-algebras whose underlying $A_\infty$-algebra is a dg algebra and whose higher operations are governed by brace operations. For such algebras, we show that if $M$ is a brace $B_\infty$-module, then the underlying dg bimodule of $M$ is $A_\infty$-quasi-isomorphic to the induced bimodule $\iota(M)$; see Theorem \ref{thm:bracemodulegeneral}. This comparison theorem is the bridge between the abstract construction of $\boxtimes_A$ and the Hopf-algebra applications below. In particular, when the underlying algebra is a dg algebra, the category $\ED^{\RR}_\infty(A)$ may be identified with the usual derived category $\ED^{\RR}(A)$ of right dg modules; see Remark \ref{rem:dgainfinitymodules}.

The main application of the article is to finite-dimensional Hopf algebras. Let $H$ be a Hopf algebra and let $\Y(\Bbbk,\Bbbk)$ be the Yoneda dg algebra of the trivial $H$-module $\Bbbk$, whose cohomology is $\Ext_H^*(\Bbbk,\Bbbk)$. We prove that $\Y(\Bbbk,\Bbbk)$ carries a natural brace $B_\infty$-structure, and more generally that $\Y(X,Y)$ is a brace $B_\infty$-module over $\Y(\Bbbk,\Bbbk)$ for complexes $X$ and $Y$ of right $H$-modules; see Proposition \ref{prop:bracestructure}. Therefore the derived category $\ED(\Y(\Bbbk,\Bbbk))$ carries the monoidal structure of Theorem \ref{thmA}.

We then compare this monoidal category with the homotopy category of complexes of injective $H$-modules. Let $\EK(\mathrm{Inj}\text{-}H)$ denote the homotopy category of complexes of injective right $H$-modules. The injective resolution $\Y(H,\Bbbk)$ of the trivial module generates the localizing subcategory on which the equivalence will be obtained. The associated Koszul duality functor is
\[
F=\Hom_H(\Y(H,\Bbbk),-)
 \colon \EK(\mathrm{Inj}\text{-}H)\longrightarrow \ED(\Y(\Bbbk,\Bbbk)).
\]
Our second main theorem identifies the monoidal behavior of this functor.

\begin{thmy}[= Theorem \ref{thm:monoidalfunctor} and Corollary \ref{cor:tensor-triangle}] \label{thmB}
Let $H$ be a finite-dimensional Hopf algebra over $\Bbbk$. There exists a triangulated lax monoidal functor
\[
F\colon \EK(\mathrm{Inj}\text{-}H) \longrightarrow \ED(\Y(\Bbbk, \Bbbk)).
\]
The restriction of $F$ to the localizing subcategory of $\EK(\mathrm{Inj}\text{-}H)$ generated by the injective resolution $\Y(H,\Bbbk)$ is a monoidal triangulated equivalence. In particular, if $H$ is local, for example if $H=\Bbbk G$ for a finite $p$-group $G$ and $\operatorname{char}(\Bbbk)=p$, then $F$ is a monoidal triangulated equivalence.
\end{thmy}

Let $G$ be a finite $p$-group, let
$\operatorname{char}(\Bbbk)=p$, and let $\mathbb 1$ be an injective
resolution of the trivial $\Bbbk G$-module. In
\cite[Example 2.5]{Kra}, Krause observed that
\[
\Hom_{\Bbbk G}(\mathbb 1,-)\colon
\EK(\mathrm{Inj}\text{-}\Bbbk G)
\xrightarrow{\sim}
\ED\bigl(\mathcal E\mathrm{nd}_{\Bbbk G}(\mathbb 1)\bigr)
\]
is a triangulated equivalence. He conjectured that this equivalence
identifies the tensor product over $\Bbbk$, equipped with the diagonal
$G$-action, with the product on the target induced by an
$E_\infty$-structure on the derived endomorphism algebra.

Benson--Krause proved this conjecture by comparing the derived
endomorphism algebra with the singular cochain algebra
$C^*(BG;\Bbbk)$ of the classifying space $BG$. More precisely, they constructed an equivalence
\[
\Theta_{\mathrm{BK}}\colon
\EK(\mathrm{Inj}\text{-}\Bbbk G)
\xrightarrow{\sim}
\ED\bigl(C^*(BG;\Bbbk)\bigr)
\]
and showed that the diagonal tensor product corresponds to the
$E_\infty$-tensor product over $C^*(BG;\Bbbk)$; see
\cite[Theorems 4.2 and 7.8]{BeKr}.

Our construction gives a purely algebraic $B_\infty$-model for the
same monoidal category. Denote \(E=\Y(\Bbbk,\Bbbk)\). Then Corollary \ref{cor:tensor-triangle}, together with the
Benson--Krause equivalence, yields a monoidal triangulated equivalence
\[
\bigl(\ED(E),\boxtimes_E\bigr)
\simeq^\otimes
\left(
\ED\bigl(C^*(BG;\Bbbk)\bigr),
\otimes^{E_\infty}_{C^*(BG;\Bbbk)}
\right).
\]
Thus the two products are two realizations of the diagonal tensor
product on $\EK(\mathrm{Inj}\text{-}\Bbbk G)$. The comparison asserted
here is at the level of monoidal derived categories; we do not require
a direct comparison between the chosen brace $B_\infty$-structure on
$E$ and the $E_\infty$-structure on $C^*(BG;\Bbbk)$.

For a general finite group, the cochain model sees the localizing
subcategory generated by the trivial module. This parallels the
restriction in Theorem \ref{thmB}. Our result additionally constructs
a lax monoidal functor on the whole homotopy category of injective
complexes and extends the construction from group algebras to
arbitrary finite-dimensional Hopf algebras.

The final section contains examples that illustrate the scope and limitations of the construction. For graded commutative algebras, the $B_\infty$-structure is trivial and $\boxtimes_A$ recovers the usual derived tensor product. For a non-local Hopf algebra with radical square zero, the Koszul duality functor is shown to be lax monoidal but not monoidal. For elementary $2$-groups, we compute the resulting brace operations explicitly and show that different Hopf structures on the same algebra can lead to essentially different brace $B_\infty$-structures, and hence to different monoidal structures on the same underlying derived category.

\subsubsection*{The article is organized as follows.}
Section \ref{section:ainfinityalgebra} recalls the necessary background on $A_\infty$-algebras, modules, bimodules, and the tensor product of $A_\infty$-bimodules.

Section \ref{section:Binfinity} introduces $B_\infty$-algebras and constructs the induction functor $\iota$ from right $A_\infty$-modules to $A_\infty$-bimodules. The key inputs are Proposition \ref{prop:dualitybinfinity}, which gives the dg coalgebra morphism defining $\iota$, and Theorem \ref{thm:iotabracemodule}, which identifies $\iota(A)$ with $A$ up to $A_\infty$-quasi-isomorphism.

Section \ref{section:monoidaltri} proves Theorem \ref{thmA}. It constructs the tensor product $\boxtimes_A$, proves the associativity and unit constraints, and establishes the comparison theorem for brace $B_\infty$-modules, Theorem \ref{thm:bracemodulegeneral}.

Section \ref{section:koszulduality} applies the theory to Hopf algebras. It proves that $\Y(\Bbbk,\Bbbk)$ has a natural brace $B_\infty$-structure, establishes the lax monoidal structure on the Koszul duality functor, and proves Theorem \ref{thmB}.

Finally, Section \ref{section:example} treats the three families of examples described above.

\subsubsection*{Graphical calculus and sign conventions.}
We use the string-diagram calculus in the symmetric monoidal category
of complexes of $\Bbbk$-vector spaces. For the foundations and
coherence of this calculus, see \cite{JS91,Sel11} and
\cite[Chapter 2]{TV17}. Diagrams are read from bottom to top.
Vertical stacking represents composition, horizontal juxtaposition represents tensor product, and a box labelled by a morphism represents that morphism. 
We suppress associativity and unit constraints in
accordance with monoidal coherence. 

Specifically, in a monoidal category $\EC$, the identity morphism $\id_X$, a morphism $f:X\rightarrow Y$, the composition of two morphisms $f:X\rightarrow Y$ and $g: Y\rightarrow Z$, and the tensor product $f\otimes h$ are represented graphically as follows:
\[\begin{tikzpicture}
\draw (0, 0) --(0, 1.9);
\node at (.3, .2) {$X$};
\node at (-.8, .95) {$\id_X$ =};
\end{tikzpicture}
\quad
\begin{tikzpicture}
\draw (0, 0) --(0, .7);
\node at (.3, .2) {$X$};
\node at (.3, 1.7) {$Y$};

\draw (-.25, .7) rectangle (.25, 1.2);
\node at (0, .95) {$f$};
\node at (-.8, .95) {$f$ =};

\draw (0, 1.2) --(0, 1.9);
\end{tikzpicture}
\qquad
\begin{tikzpicture}
\draw (0, -.3) --(0, .2);
\draw (-.25, .2) rectangle (.25, .7);
\node at (0, .45) {$f$};
\draw (0, .7) --(0, 1.2);
\draw (-.25, 1.2) rectangle (.25, 1.7);
\node at (0, 1.45) {$g$};
\draw (0, 1.7) --(0, 2.2);
\node at (.3, -.1) {$X$};
\node at (.3, .95) {$Y$};
\node at (.3, 1.95) {$Z$};
\node at (-.9, .65) {$g\circ f=\quad$};
\end{tikzpicture}
\;
\begin{tikzpicture}
\draw (-1, 0) --(-1, .7);
\draw (-1.25, .7) rectangle (-.75, 1.2);
\node at (-1, .95) {$f$};
\draw (-1, 1.2) --(-1, 1.9);
\draw (0, 0) --(0, .7);
\node at (-.7, .2) {$X$};
\node at (-.7, 1.7) {$Y$};
\node at (.3, .2) {$U$};
\node at (.3, 1.7) {$V$};
\node at (-3.5, .95) {and};

\node at (-2, .95) {$f\otimes h= \quad$};

\draw (-.25, .7) rectangle (.25, 1.2);
\node at (0, .95) {$h$};
\node at (.8, .2) {.};

\draw (0, 1.2) --(0, 1.9);
\end{tikzpicture}
\]
We read diagrams from bottom to top, following a convention common in parts of the tensor category literature.

A crossing of two strands represents the Koszul symmetry. Thus, for
homogeneous elements $x$ and $y$,
\[
\tau(x\otimes y)=(-1)^{|x||y|}y\otimes x.
\]
For homogeneous morphisms $f$ and $g$, our convention for their tensor
product is
\[
(f\otimes g)(x\otimes y)
   =(-1)^{|g||x|}f(x)\otimes g(y).
\]
All signs arising from crossings or from moving homogeneous
operations past homogeneous inputs are determined by these
conventions.

Throughout this article, we assume that $\Bbbk$ is a field. When we write $\Hom$ and $\otimes$ without subscripts, we mean $\Hom_\Bbbk$ and $\otimes_\Bbbk$, respectively.

\section{\texorpdfstring{$A_{\infty}$}{}-algebras and modules}\label{section:ainfinityalgebra}

This section fixes the $A_\infty$-conventions used throughout the article. We work with the shifted convention, since it makes the signs in the tensor-product and induction constructions more transparent. We recall $A_\infty$-algebras, modules, bimodules, the $A_\infty$-tensor product of bimodules, and a general induction procedure from right modules to bimodules. For background, see \cite[Section 5]{GJ} and \cite[Section 2]{Vor}.

\subsection{\texorpdfstring{$A_{\infty}$}{}-algebras and \texorpdfstring{$A_{\infty}$}{}-morphisms}\label{subsection:Ainfinityalgebra}
Let $\Bbbk$ be a field. For a graded vector space $A=\bigoplus_{i\in \mathbb Z} A^i$, we write $sA$ for the $1$-shift of $A$, so that $(sA)^i=A^{i+1}$ for all $i\in \mathbb Z$.
\begin{defn}
An {\it $A_{\infty}$-algebra over $\Bbbk$} is a $\mathbb Z$-graded vector space $A$ together with graded maps
\[
m^n \colon (sA)^{\otimes n} \to sA,
\qquad n\geq 1,
\]
of degree $1$, such that, for each $n\geq 1$,
\begin{align}\label{align:A-infinity3}
 \sum_{j=0}^{n-1}\sum_{t=1}^{n-j} m^{n-t+1}\circ (\mathbf 1_{s A} ^{\otimes j} \otimes m^t \otimes \mathbf 1_{sA} ^{\otimes n-j-t}) = 0.
\end{align}
Here $\mathbf 1_{sA}\colon sA\to sA$ denotes the identity map. We often omit subscripts from identity maps when no confusion can arise.

In particular, $m^1\circ m^1=0$, so $(sA,m^1)$ is a cochain complex.
\end{defn}

\begin{rem}
The shifted operations $m^n$ correspond to the usual unshifted operations
\[
\nu^n \colon A^{\otimes n} \to A
\]
of degree $2-n$. They are determined by the following commutative diagram, equivalently by $s\circ \nu^n=m^n\circ s^{\otimes n}$:
\begin{align}\label{align:mnbn}
\xymatrix{
A^{\otimes n} \ar[r]^-{\nu^n} \ar[d]^-{s^{\otimes n}} & A \ar[d]^-{s}\\
(sA)^{\otimes n} \ar[r]^-{m^n} & sA
}
\end{align}
where $s\colon A\to sA$, $a\mapsto sa$, is the natural map of degree $-1$. Thus $m^n$ and $\nu^n$ are related by
\[
m^n(sa_1 \otimes \dotsb \otimes sa_n) = (-1)^{|a_1|(n-1) + |a_2|(n-2) + \dotsb + |a_{n-1}|} s \nu^n (a_1\otimes \dotsb \otimes a_n)
\]
for any $a_1,\dotsc, a_n \in A$. In particular, $m^1(sa_1) = s\nu^1(a_1)$.

Under this translation, identity \eqref{align:A-infinity3} is equivalent to
\begin{align}
\label{Ainfinity2}
\sum_{j=0}^{n-1}\sum_{t=1}^{n-j} (-1)^{j+t(n-j-t)} \nu^{n-t+1}\circ (\mathbf 1 ^{\otimes j} \otimes \nu^t \otimes \mathbf 1 ^{\otimes n-j-t}) = 0,
\end{align}
which is the usual unshifted form of the $A_\infty$-identities; see, for example, \cite[Section 3]{Kel01}. The extra signs $(-1)^{j+t(n-j-t)}$ in \eqref{Ainfinity2} are one reason for using the shifted convention throughout the article.
\end{rem}

\begin{exm}\label{exm:shiftedandunshifted}
Every dg $\Bbbk$-algebra $(A,d,\cdot)$, in the usual unshifted sense, determines an $A_\infty$-algebra $(A,m^n)$ by
 $$ m^1(sa_1)= s d(a_1), \qquad m^2(sa_1 \otimes sa_2)= (-1)^{|a_1|} s (a_1\cdot a_2)$$
 and $m^k =0$ for $k \geq 3.$
\end{exm}

\begin{rem}\label{rem:cofreereformation}
It is often useful to encode the $A_\infty$-operations coalgebraically. Consider the cofree tensor coalgebra
\[
\bigg(T^c(sA)=\bigoplus_{i\geq 0}(sA)^{\otimes i},\Delta\bigg),
\]
where $(sA)^{\otimes 0}=\Bbbk$ and the coproduct $\Delta$ is given by
\begin{align*}
\Delta(sa_1\otimes \dotsb \otimes sa_n) =& 1_\Bbbk \otimes (sa_{1,n})+ (sa_{1,n})\otimes 1_\Bbbk + \sum_{i=1}^{n-1} (sa_{1,i}) \otimes (sa_{i+1,n}).
\end{align*}
Throughout the article we use the shorthand
\[
sa_{i,j}=sa_i\otimes sa_{i+1}\otimes\dotsb\otimes sa_j
\]
for $i\leq j$; if $i>j$, the corresponding empty tensor is understood as $1_\Bbbk$.

Any collection of degree-$1$ maps $\{m^n\colon (sA)^{\otimes n}\to sA\}_{n\geq 1}$ has a unique extension
\begin{align}
\widehat m \colon T^c (sA) \to T^c (sA)
\end{align}
of degree $1$ such that $\widehat m(1)=0$ and, for $n\geq 1$,
\begin{align}\label{align:widehatm}
\widehat m (sa_1 \otimes \dotsb \otimes sa_n) = \sum_{t=1}^n\sum_{i=0}^{n-t} (-1)^{\epsilon_i} sa_{1,i} \otimes m^t(sa_{i+1, i+t}) \otimes sa_{i+t+1, n},
\end{align}
where $\epsilon_i=|a_1|+\dotsb+|a_i|-i$. The map $\widehat m$ is automatically a coderivation of degree $1$; that is,
\[
\Delta\circ \widehat m=(\widehat m\otimes \mathbf 1_{T^c(sA)}+\mathbf 1_{T^c(sA)}\otimes \widehat m)\circ \Delta.
\]
Moreover, identity \eqref{align:A-infinity3} is equivalent to $\widehat m\circ \widehat m=0$.
In this case, $(T^c(sA),\Delta,\widehat m)$ is a differential graded coalgebra, or dg coalgebra for short. Conversely, any dg coalgebra structure $(T^c(sA),\Delta,\widehat m)$ such that the image of $\widehat m$ lies in $T^{\geq 1}(sA)=\bigoplus_{i\geq 1}(sA)^{\otimes i}$ and $\widehat m(1)=0$ induces an $A_\infty$-algebra structure $(A,m^n)_{n\geq 1}$, where $m^n\colon (sA)^{\otimes n}\to sA$ is the composite
\[
m^n \colon (sA)^{\otimes n} \hookrightarrow T^c(sA) \xrightarrow{\widehat m} T^c(sA) \xrightarrow{\mathrm{proj}} sA.
\]

Thus the assignment $(A,m^n)\mapsto (T^c(sA),\Delta,\widehat m)$ embeds $A_\infty$-algebra structures on $A$ into dg coalgebra structures on $(T^c(sA),\Delta)$:
\[
\{\text{$A_\infty$ structures on $A$}\} \hookrightarrow \{\text{dg structures on $(T^c(sA), \Delta)$}\}.
\]

\end{rem}

\begin{defn}
An $A_\infty$-algebra $(A,m^n)$ is {\it strictly unital} if there is an element $e_A\in A$ of degree $0$ such that
\begin{enumerate}
\item $m^1(se_A) = 0,$
\item $m^2(se_A \otimes sa)=sa=(-1)^{|a|}m^2(sa\otimes se_A)$ for any $a\in A$,
\item $m^k(sa_{1,i} \otimes se_A \otimes sa_{i+1, k-1}) = 0$ for $k \geq 3$ and $a_1,\dotsc, a_{k-1} \in A$.
\end{enumerate}

\end{defn}
Throughout this article, all $A_\infty$-algebras are assumed to be strictly unital.

\begin{defn}\label{defn:ainfinitymorphism}
An {\it $A_{\infty}$-morphism} $\widehat f\colon A\to B$ between $A_\infty$-algebras $(A,m_A^n)$ and $(B,m_B^n)$ is a collection of degree-$0$ graded maps
\[
f_n\colon (sA)^{\otimes n}\to sB,
\qquad n\geq 1,
\]
such that, for each $n\geq 1$,
\begin{equation}\label{equationforAmorph2}
\begin{split}
\sum_{j=0}^{n-1}\sum_{t=1}^{n-j} {} & f_{n-t+1}\circ (\mathbf{1}^{\smallotimes
j}\smallotimes m_A^t\smallotimes \mathbf{1}^{\smallotimes n-t-j})=\sum_{\substack{r\geq 1\\ i_1+\dotsb + i_r=n}} m_{B}^{r} \circ (f_{i_1}\smallotimes \cdots\smallotimes f_{i_r}).
\end{split}
\end{equation}
In particular, $f_1\colon (sA,m_A^1)\to (sB,m_B^1)$ is a cochain map.
\end{defn}

\begin{rem}\label{rem:a-infinitytwo1}
Equivalently, a collection of graded maps $\{f_n\}_{n\geq 1}$ defines an $A_\infty$-morphism from $A$ to $B$ if and only if the extension
\[
\widehat f\colon (T^c(sA),\Delta,\widehat m_A)\to (T^c(sB),\Delta,\widehat m_B)
\]
given by
 \begin{equation*}
 \begin{split}
 \widehat{f} (sa_{1,n}) = & \sum_{\substack{r\geq 1\\ i_1+\dotsb + i_r=n}} f_{i_1}(sa_{1, i_1}) \otimes f_{i_2}(sa_{i_1+1, i_1+i_2}) \otimes \dotsb \otimes f_{i_r}(sa_{i_1+\dotsb+i_{r-1}+1,i_1+\cdots+i_r})
 \end{split}
 \end{equation*}
is a degree-$0$ homomorphism of dg coalgebras.
The composition $\widehat f\circ_\infty \widehat g$ of two $A_\infty$-morphisms is the usual composition of the corresponding dg coalgebra homomorphisms.
\end{rem}

\subsection{Left \texorpdfstring{$A_{\infty}$}{}-modules}
We next recall left modules, right modules, and $A_\infty$-bimodules. These objects form dg categories whose morphism complexes and differentials can be written explicitly; see \cite[Section 3]{Kel01a}.

\begin{defn}
Let $(A, m^n)$ be an $A_\infty$-algebra. A {\it left $A_\infty$-module} $M$ over $A$ is a $\mathbb Z$-graded vector space $M$ equipped with maps
\[
\lambda_M^n \colon (sA)^{\otimes n-1} \otimes sM \to sM, \ n\geq 1
\]
of degree $1$ satisfying, for each $n\geq 1$,
\begin{align}\label{align:leftmodule}
\begin{aligned}
\sum_{t=1}^{n}& \lambda_M^{n-t+1}\circ (\mathbf 1_{sA} ^{\otimes n-t} \otimes \lambda_M^t ) + \sum_{j=1}^{n-1}\sum_{t=1}^{n-j} \lambda_M^{n-t+1}\circ (\mathbf 1_{sA} ^{\otimes n-j-t } \otimes m_A^t \otimes \mathbf 1_{sA} ^{\otimes j}) = 0,
\end{aligned}
\end{align}
In particular, $\lambda_M^1\circ \lambda_M^1=0$, so $(sM,\lambda_M^1)$ is a complex. We write $(M,\lambda_M^n)$, or simply $(M,\lambda_M)$, when the superscript is not needed.
\end{defn}

\begin{exm}
Let $(A, m_A^n)$ be an $A_\infty$-algebra. Then $A$, endowed with the maps
\[
\lambda^n_A =m^n_A\colon (sA)^{\otimes n-1} \otimes sA \to sA, \quad \text{for $n\geq 1$}
\]
is a left $A_\infty$-module over itself, as follows by comparing \eqref{align:A-infinity3} with \eqref{align:leftmodule}.
\end{exm}

\begin{rem}
The left-module identities also have a coalgebraic form. For a graded vector space $M$, consider the left cofree comodule $(T^c(sA)\otimes sM,\Delta_M^{\LL})$ over $(T^c(sA),\Delta)$, where
\[
\Delta_M^{\LL}\colon T^c(sA)\otimes sM\to T^c(sA)\otimes T^c(sA)\otimes sM
\]
is defined by $\Delta_M^{\LL}=\Delta\otimes \mathbf 1_{sM}$.

As in Remark \ref{rem:cofreereformation}, the maps $\lambda_M^n$ can be packaged into a dg comodule structure on $(T^c(sA)\otimes sM,\Delta_M^{\LL})$. More precisely, define the extension
\begin{align}
\widehat \lambda_M \colon T^c(sA) \otimes sM \to T^c(sA) \otimes sM
\end{align}
of degree $1$ given by
\begin{align*}
\widehat \lambda_M (sa_{1,n-1} \otimes sx) ={} & \sum_{i=0}^{n-1} (-1)^{\epsilon_i} sa_{1,i} \otimes \lambda_M^{n-i}(sa_{i+1, n-1}\otimes sx)+ \\
& \sum_{t=1}^{n-1}\sum_{i=0}^{n-1-t} (-1)^{\epsilon_i} sa_{1,i}\otimes m_A^t(sa_{i+1, i+t}) \otimes sa_{i+t+1, n-1}\otimes sx
\end{align*}
for each $n\geq 1$, where $\epsilon_i=|a_1|+\dotsb+|a_i|-i$.
Equivalently,
\begin{align}\label{algin:leftwidehat}
\widehat \lambda_M = \widehat m \otimes \id_{sM} + (\id_{T} \otimes \lambda_M) \circ \Delta^{\LL}_M.
\end{align}

Identity \eqref{algin:leftwidehat}, together with the coassociativity of $\Delta$, yields $$(\widehat{m} \otimes \mathbf 1_T \otimes \mathbf 1_{sM} + \mathbf 1_T \otimes \widehat \lambda_M)\circ \Delta^{\LL}_M = \Delta^{\LL}_M \circ \widehat \lambda_M,$$ where $T=T^c(sA)$ and $\Delta^{\LL}_M=\Delta \otimes \id_{sM}$. Thus $\widehat\lambda_M$ is a coderivation of the comodule $(T^c(sA)\otimes sM,\Delta_M^{\LL})$ over the dg coalgebra $(T^c(sA),\Delta,\widehat m)$.

One checks that identity \eqref{align:leftmodule} is equivalent to $\widehat\lambda_M\circ \widehat\lambda_M=0$. Equivalently, $(T^c(sA)\otimes sM,\Delta_M^{\LL},\widehat\lambda_M)$ is a dg comodule over $(T^c(sA),\Delta,\widehat m)$; see \cite[Section 3]{Kel01a}.
\end{rem}

\begin{defn}\label{defn:unitalleft}
Let $(A, m_A^n)$ be a strictly unital $A_{\infty}$-algebra with unit $e_A$. A left $A_{\infty}$-module $(M,\lambda_M^n)$ over $A$ is {\it strictly unital} if, for every $x\in M$,
\begin{enumerate}
\item $\lambda_M^2(se_A \otimes sx)=sx$,
\item $\lambda_M^n(sa_{1, n-1} \otimes sx) = 0$ for $n > 2$ with $a_i=e_A$ for some $1\leq i <n$.
\end{enumerate}
\end{defn}

\begin{defn}
Let $(A,m^n)$ be an $A_\infty$-algebra. We denote by $\EC^{\LL}_\infty(A)$ the dg category of strictly unital left $A_\infty$-modules over $A$. Its objects are strictly unital left modules. For two such modules $(M,\lambda_M^n)$ and $(N,\lambda_N^n)$, the pre-morphism complex is
\[
(\Hom(T^c(sA)\otimes sM,sN),\delta),
\]
where
\[
\delta(\varphi)=\lambda_N\circ \widehat\varphi-(-1)^{|\varphi|}\varphi\circ \widehat\lambda_M.
\]
Here $\widehat\varphi\colon T^c(sA)\otimes sM\to T^c(sA)\otimes sN$ is the standard extension obtained by applying $\varphi\in \Hom(T^c(sA)\otimes sM,sN)$ after an arbitrary initial segment:
$$\widehat \varphi(sa_{1,k}\otimes sx) = \sum_{t=0}^k (-1)^{|\varphi|(|sa_1|+\dots+|sa_t|)} sa_{1, t}\otimes \varphi(sa_{t+1,k} \otimes sx).$$
Equivalently, $\widehat{\varphi}=(\id_T\otimes \varphi)\circ (\Delta\otimes \id_{sM})$.
The composition of two pre-morphisms $\varphi \in \Hom(T^c(sA) \otimes sM , sN)$ and $\psi \in \Hom(T^c(sA) \otimes sL , sM)$ is defined by
\[
 \varphi\overset{\infty}{\circ} \psi := \varphi \circ \widehat \psi .
\]

The dg category $\EC^{\LL}_\infty(A)$ is pretriangulated \cite{Kel01a}. We define the {\it derived category $\ED^{\LL}_{\infty}(A)$ of left $A_\infty$-modules} over $A$ to be its homotopy category; that is, \begin{align}\label{align:definderived}
\ED^{\LL}_{\infty}(A) = \mathrm H^0(\EC^{\LL}_\infty(A)).
\end{align}
\end{defn}

\begin{rem}
By a {\it morphism} $\varphi\colon (M,\lambda_M^n)\to (N,\lambda_N^n)$ of $A_\infty$-modules, we mean a pre-morphism satisfying $\delta(\varphi)=0$. Equivalently, $\varphi$ is a collection of degree-$0$ maps
\[
\varphi^n\colon (sA)^{\otimes n-1}\otimes sM\to sN,
\qquad n\geq 1,
\]
such that, for each $n\geq 1$,
\begin{align}\label{align:modulemorphism}
\begin{aligned}
&\sum_{t=1}^{n} \varphi^{n-t+1}\circ (\mathbf{1}^{\smallotimes
n-t}\smallotimes \lambda_M^t ) + \sum_{j=1}^{n-1}\sum_{t=1}^{n-j} \varphi^{n-t+1}\circ (\mathbf{1}^{\smallotimes
n-t-j}\smallotimes m^t\smallotimes \mathbf{1}^{\smallotimes j})\\
={} & \sum_{t=1}^{n} \lambda_N^{n-t+1} \circ (\mathbf 1^{\otimes n-t} \smallotimes \varphi^{t} ).
\end{aligned}
\end{align}
In particular, $\varphi^1\circ \lambda_M^1=\lambda_N^1\circ \varphi^1$, so $\varphi^1\colon sM\to sN$ is a morphism of complexes.
The morphism $\varphi$ is a {\it quasi-isomorphism} if $\varphi^1$ is a quasi-isomorphism. By \cite[Subsection 4.2]{Kel01}, the category $\ED^{\LL}_{\infty}(A)$ defined in \eqref{align:definderived} agrees with the usual derived category obtained by inverting quasi-isomorphisms; indeed, a morphism is a quasi-isomorphism if and only if it is a homotopy equivalence \cite[Theorem 4.2.1]{Kel01}.

For strictly unital modules, we require pre-morphisms to be compatible with the unit: namely,
\begin{align}
\varphi^n( sa_{1,n-1}\otimes sx)=0, \; n\geq 2
\end{align}
with $a_i=e_A$ for some $1\leq i <n$.
\end{rem}

\subsection{Right \texorpdfstring{$A_{\infty}$}{}-modules}
Right $A_\infty$-modules are defined analogously. We recall the convention only briefly and refer to \cite[Section 4]{Kel01} for further details.

A right $A_\infty$-module over $A$ is a $\mathbb Z$-graded vector space $M$ equipped with degree-$1$ maps
\[
\rho_M^n \colon sM \otimes (sA)^{\otimes n-1} \to sM, \quad n \geq 1
\]
satisfying the right-handed analogue of \eqref{align:leftmodule}. Equivalently, $(M,\rho_M^n)$ determines a right dg comodule $(sM\otimes T^c(sA),\Delta_M^{\RR},\widehat\rho_M)$ over the dg coalgebra $(T^c(sA),\Delta,\widehat m)$.

\begin{rem}\label{rem:unitalright}
 As in Definition \ref{defn:unitalleft}, a right $A_\infty$-module $(M,\rho_M)$ is {\it strictly unital} if
\[
\rho_M^2(sx\otimes se_A)=(-1)^{|x|}sx
\]
and $\rho_M^n(sx\otimes\dotsb\otimes se_A\otimes\dotsb)=0$ for $n>2$.
\end{rem}

We denote by $\EC^{\RR}_\infty(A)$ the dg category of strictly unital right $A_\infty$-modules over $A$, and we define the {\it derived category $\ED^{\RR}_{\infty}(A)$ of strictly unital right $A_\infty$-modules} to be $\mathrm H^0(\EC^{\RR}_\infty(A))$.

\begin{rem}
Let $(A,m_A^n)$ be an $A_\infty$-algebra. Recall from \cite[Remark 5.8]{CLW} that the {\it opposite $A_\infty$-algebra} of $A$ is $(A,m^n_{A^{\op}})$, where $m^n_{A^{\op}}\colon (sA)^{\otimes n}\to sA$ is given by
\[
m^n_{A^{\op}}(sa_1\otimes \dotsb \otimes sa_n) = (-1)^{\epsilon} m^n_A(sa_n \otimes \dotsb \otimes sa_1) , \quad n \geq 1
\]
where $\epsilon=n-1+\sum_{j=1}^{n-1}|sa_j|(|sa_{j+1}|+\dotsb+|sa_n|)$. We write this opposite algebra simply as $A^{\op}$.

A right $A_\infty$-module $(M,\rho_M^n)$ over $A$ can then be viewed as a left $A_\infty$-module over $A^{\op}$. Its structure maps
\[
\lambda^n_{M^{\op}} \colon (sA)^{\otimes n-1} \otimes sM \to sM,\quad n\geq 1
\]
are determined by
\begin{align*}
 \lambda^n_{M^{\op}}(sa_1\otimes \dotsb \otimes sa_{n-1} \otimes sx) = (-1)^{\epsilon}
 {\rho}_M^n(sx \otimes sa_{n-1} \otimes \dotsb \otimes sa_1)
\end{align*}
where
\[
\epsilon=(|x|-1)\left(\sum^{n-1}_{i=1}(|a_i|-1)\right)+\sum^{n-1}_{j=1}(|a_j|-1)\left(\sum^{j-1}_{k=1}(|a_k|-1)\right).
\]
This construction gives an isomorphism of dg categories
\begin{align}\label{align:leftrightopposite}
 \EC^{\RR}_\infty(A) \xrightarrow{\simeq} \EC^{\LL}_\infty(A^{\op}), \quad (M, \rho_M^n) \mapsto (M, \lambda_{M^{\op}}^n).
\end{align}
\end{rem}

\subsection{\texorpdfstring{$A_{\infty}$}{}-bimodules and the tensor product}
We now recall $A_\infty$-bimodules and the $A_\infty$-tensor product. This tensor product gives the standard monoidal triangulated structure on the derived category of $A_\infty$-bimodules.

\subsubsection{\textbf{\textup{$A_\infty$-bimodules}}}
\begin{defn}\label{defn:Abimodule}
Let $(A, m^n)$ be an $A_\infty$-algebra. An $A_\infty$-bimodule over $A$ is a $\mathbb Z$-graded vector space $M$ equipped with maps
\begin{align*}
\beta_M^{p,q} \colon (sA)^{\otimes p} \otimes sM \otimes (sA)^{\otimes q} \to sM, \quad p, q \geq 0
\end{align*}
of degree $1$ satisfying the following identity:
\begin{align*}
0={}& \sum_{i=0}^{p-1}\sum_{t=1}^{p-i} (-1)^{\epsilon_i } \beta_M^{p-t+1, q} \left(sa_{1,i }\smallotimes m^t(sa_{i+1, i+t}) \smallotimes sa_{i+t+1, p } \smallotimes sx \smallotimes sb_{1,q}\right)\\
&+ \sum_{u=0}^p\sum_{v=0}^q (-1)^{\epsilon_{p-u}} \beta_M^{p-u, q-v} \left(sa_{1, p-u} \smallotimes \beta_M^{u, v}(sa_{p-u+1, p} \smallotimes sx \smallotimes sb_{1, v}) \smallotimes sb_{v+1, q}\right)\\
&+ \sum_{i=0}^{q-1}\sum_{t=1}^{q-i} (-1)^{\epsilon_p + |sx|+ \eta_i} \beta_M^{p, q-t+1} \left(sa_{1, p} \smallotimes sx \smallotimes sb_{1, i} \smallotimes m^t (sb_{i+1, i+t}) \smallotimes sb_{i+t+1,q}\right)
\end{align*}
where $\epsilon_i=|a_1|+\dotsb+|a_i|-i$ and $\eta_i=|b_1|+\dotsb+|b_i|-i$. In particular, $(sM,\beta_M^{0,0})$ is a complex.
\end{defn}
Strict unitality for $A_\infty$-bimodules is defined in the analogous left-and-right sense.

\begin{rem}
Let $(A,m^n)$ be an $A_\infty$-algebra. A collection of graded maps
\[
\beta_M^{p,q}\colon (sA)^{\otimes p}\otimes sM\otimes (sA)^{\otimes q}\to sM,
\qquad p,q\geq 0,
\]
determines a map
\begin{align*}
\widehat \beta_M \colon T^c(sA) \otimes sM \otimes T^c(sA) \to T^c(sA) \otimes sM \otimes T^c(sA)
\end{align*}
of degree $1$ by
\begin{align}\label{align:betahat}
\widehat \beta_M = \widehat m \otimes \id_{sM} \otimes \id_{T} + (\id_T \otimes \beta_M \otimes \id_T) \circ (\Delta \otimes \id_{sM} \otimes \Delta) + \id_T \otimes \id_{sM} \otimes \widehat m.
\end{align}

The map $\widehat\beta_M$ satisfies the following compatibility diagram:
\begin{align*}
 \xymatrix@C=7em@R=5em{
 T\otimes sM \otimes T \ar[d]^{\widehat{\beta}_M}
 \ar[r]^{\begin{pmatrix}
 \Delta\smallotimes \id \smallotimes \id \\
 \id \smallotimes \id \smallotimes \Delta
 \end{pmatrix}\qquad\qquad\qquad} &(T^{\otimes 2} \otimes sM\otimes T) \oplus (T\otimes sM\otimes T^{\otimes 2}) \ar@<-5ex>[d]^{\begin{pmatrix}
 \id \smallotimes \widehat{\beta}_M + \widehat m\smallotimes \id^{\smallotimes 3}& 0 \\
 0 & \widehat{\beta}_M \smallotimes \id + \id^{\smallotimes 3}\smallotimes \widehat m
 \end{pmatrix}}\\
 T\otimes sM \otimes T \ar[r]_{\begin{pmatrix}
 \Delta\smallotimes \id \smallotimes \id \\
 \id \smallotimes \id \smallotimes \Delta
 \end{pmatrix}\qquad\qquad\qquad} &(T^{\otimes 2}\otimes sM\otimes T) \oplus (T\otimes sM\otimes T^{\otimes 2})
 }
\end{align*}
where $T=T^c(sA)$. 
This diagram shows that $\widehat\beta_M$ is a bicomodule coderivation with respect to both the left and right coactions.
The maps $(\beta_M^{p,q})$ define an $A_\infty$-bimodule structure if and only if $\widehat\beta_M\circ \widehat\beta_M=0$, equivalently $\beta_M\circ \widehat\beta_M=0$. Thus $(T^c(sA)\otimes sM\otimes T^c(sA),\Delta_M^{\LL}, \Delta_M^{\RR}, \widehat\beta_M)$ is a dg bicomodule over $(T^c(sA),\Delta,\widehat m)$; see \cite[Definition 2.6]{Tra}.
\end{rem}

This reformulation gives the dg category of $A_\infty$-bimodules.

\begin{defn}
Let $(A,m^n)$ be an $A_\infty$-algebra. Denote by $\EC^{\BB}_\infty(A)$ the dg category of strictly unital $A_\infty$-bimodules over $A$. For two objects $(M,\beta_M^{p,q})$ and $(N,\beta_N^{p,q})$, the pre-morphism complex is
\[
(\Hom(T^c(sA)\otimes sM\otimes T^c(sA),sN),\delta),
\]
with differential
\[
\delta(\varphi)=\beta_N\circ \widehat\varphi-(-1)^{|\varphi|}\varphi\circ \widehat\beta_M.
\]
The composition of pre-morphisms $\varphi\in \Hom(T^c(sA)\otimes sM\otimes T^c(sA),sN)$ and $\psi\in \Hom(T^c(sA)\otimes sL\otimes T^c(sA),sM)$ is
\begin{align}\label{align:compositionmorphisms}
 \varphi\overset{\infty}{\circ} \psi := \varphi \circ \widehat \psi
\end{align}
where $\widehat\psi\colon T^c(sA)\otimes sL\otimes T^c(sA)\to T^c(sA)\otimes sM\otimes T^c(sA)$ is defined by
\begin{align*}
\widehat \psi(sa_{1,p} \otimes sx \otimes sb_{1,q})= \sum_{u=0}^p\sum_{v=0}^q (-1)^{\epsilon_{p-u}} sa_{1,p-u} \otimes \psi^{u, v}(sa_{p-u+1, p} \otimes sx \otimes sb_{1,v}) \otimes sb_{v+1,q} ,
\end{align*}
where $\epsilon_{p-u} = |\psi| ( |sa_1|+\dotsb+|sa_{p-u}|)$ and $\psi^{u, v}$ is the restriction of $\psi$ to the component $(sA)^{\otimes u} \otimes s L \otimes (sA)^{\otimes v}.$
For any $A_\infty$-bimodule $M$, the {\it identity morphism} is the ordinary identity $\id_{sM}\colon sM\to sM$, with no higher components. In particular,
\[
 \varphi \overset{\infty}{\circ} \id_{sM}=\varphi=\id_{sN} \overset{\infty}{\circ} \varphi
\]

The dg category $\EC^{\BB}_\infty(A)$ is pretriangulated \cite{Kel01a}. Its homotopy category
\[
\ED^{\BB}_{\infty}(A)=\mathrm H^0(\EC^{\BB}_\infty(A))
\]
is the {\it derived category of $A_\infty$-bimodules} over $A$.
\end{defn}

\begin{rem}\label{rem:Ainfinitydgmodules}
Let $(A,m^n)$ be a dg algebra, so $m^n=0$ for $n>2$. Then every $A_\infty$-bimodule is $A_\infty$-quasi-isomorphic to a dg bimodule over $A$. Hence the canonical functor
\[
\ED^{\BB}(A)\to \ED^{\BB}_{\infty}(A)
\]
is a triangulated equivalence, where $\ED^{\BB}(A)$ is the classical derived category of dg bimodules; see \cite[Subsection 4.3]{Kel01} and \cite{Kel94}.

Thus, for a dg algebra, one may compute morphisms either classically, by inverting quasi-isomorphisms of dg bimodules, or $A_\infty$-categorically, by allowing $A_\infty$-morphisms and then passing to the homotopy category. The same observation applies to left and right $A_\infty$-modules.
\end{rem}

\begin{rem}\label{rem:forgetfulfunctor}
An $A_\infty$-bimodule $(M,\beta_M^{p,q})$ naturally induces a left $A_\infty$-module and a right $A_\infty$-module by restricting to one side; explicitly, $\lambda_M^n=\beta_M^{n-1,0}$ and $\rho_M^n=\beta_M^{0,n-1}$ for $n\geq 1$. Thus we have forgetful functors
\begin{align*}
F^{\LL} &\colon \EC^{\BB}_\infty(A) \to \EC^{\LL}_\infty(A), \quad (M, \beta_M^{p, q}) \mapsto (M, \{\beta_M^{n-1,0}\}_{n\geq 1}),\\
F^{\RR} &\colon \EC^{\BB}_\infty(A) \to \EC^{\RR}_\infty(A), \quad (M, \beta_M^{p, q}) \mapsto (M, \{\beta_M^{0,n-1}\}_{n\geq 1}).
\end{align*}
\end{rem}

\subsubsection{\textbf{\textup{Tensor product of $A_\infty$-bimodules and internal Hom}}}
We now recall the tensor product of $A_\infty$-bimodules; see \cite[Definition 10]{Fer12} for details.

Let $T=T^c(sA)$. We first introduce the isomorphism of degree $-1$
\begin{align*}
\zeta\colon T \otimes sM \otimes T \otimes N \otimes T &\longrightarrow T \otimes sM \otimes T \otimes sN \otimes T \\
sa_{1, k} \otimes sx \otimes sb_{1, l} \otimes y \otimes sc_{1, r} &\mapsto (-1)^{|sa_{1,k}|+|x|+|sb_{1,l}|} sa_{1, k}
\otimes sx \otimes sb_{1,l} \otimes sy \otimes sc_{1, r}.
\end{align*}

\begin{defn}
Let $(A, m^n)$ be an $A_\infty$-algebra. Let $(M, \beta_M^{p, q})$ and $(N, \beta_N^{p, q})$ be two $A_\infty$-bimodules.
The {\it tensor product} $M\overset{\infty}\otimes_A N$ is the following $A_\infty$-bimodule over $A$. Its underlying graded vector space is
\[
M\otimes T\otimes N.
\]
The structure maps
\[
\beta_{M \overset{\infty}\otimes_A N}^{p, q} \colon (sA)^{\otimes p}\otimes sM \otimes T \otimes N \otimes (sA)^{\otimes q} \to sM \otimes T \otimes N
\]
are defined by requiring that the extension $\widehat\beta_{M\overset{\infty}\otimes_A N}$ be the composite
\[
 \zeta^{-1} \circ \widetilde \beta_{M \overset{\infty}\otimes_A N} \circ \zeta \colon T \otimes sM \otimes T \otimes N \otimes T \to T \otimes sM \otimes T \otimes N \otimes T,
\]
where $\zeta$ is given as above and $\widetilde \beta_{M \overset{\infty}\otimes_A N} \colon T \otimes sM \otimes T \otimes sN \otimes T \to T \otimes sM \otimes T \otimes sN \otimes T$ is defined by
\[
 \widetilde \beta_{M \overset{\infty}\otimes_A N} = -\widehat \beta_M \otimes \mathbf 1_{sN \otimes T} - \mathbf 1_{T\otimes sM}\otimes \widehat \beta_N + \mathbf 1_{T\otimes sM} \otimes \widehat m_A \otimes \mathbf 1_{sN\otimes T} .
\]
\end{defn}

\begin{rem}
On components, the maps $\beta_{M\overset{\infty}\otimes_A N}^{p,q}$ are as follows. Let $x\in M$ and $y\in N$.
\begin{enumerate}
\item If $p=q=0$, then
\begin{align*}
&\beta_{M \overset{\infty}\smallotimes_A N}^{0, 0} \left( sx\smallotimes sb_{1,l}\smallotimes y\right)\\
= & \sum_{i=0}^l \beta_M^{i+1} (sx\smallotimes sb_{1,i}) \smallotimes sb_{i+1,l} \smallotimes y + \sum_{i=0}^l (-1)^{|x|+|sb_{1,l}|} sx\smallotimes sb_{1,i} \smallotimes s^{-1}\beta_N^{l-i+1}(sb_{i+1,l} \smallotimes sy)\\
& +\sum_{t=1}^l \sum_{i=0}^{l-t} (-1)^{|x|+|sb_{1,i}|+1} sx\smallotimes sb_{1,i} \smallotimes m^t(sb_{i+1, i+t}) \smallotimes sb_{i+t+1, l} \smallotimes y.
\end{align*}

\item If $q=0$ and $p>0$, then
\begin{align*}
\beta_{M \overset{\infty}\smallotimes _A N}^{p, 0} (sa_{1,p} \smallotimes sx\smallotimes sb_{1,l} \smallotimes y) = \sum_{i=0}^l \beta_M^{p, i} (sa_{1,p} \smallotimes sx\smallotimes sb_{1,i}) \smallotimes sb_{i+1, l} \smallotimes y.
\end{align*}

\item If $p=0$ and $q>0$, then the first and second summands vanish, and hence
\begin{align*}
\beta_{M \overset{\infty}\smallotimes _A N}^{0, q}(sx\smallotimes sb_{1,l} \smallotimes y\smallotimes sc_{1,q}) = \sum_{i=0}^l (-1)^{|x|+ |sb_{1,l}|} sx\smallotimes sb_{1,i} \smallotimes s^{-1}\beta^{l-i,q}_N(sb_{i+1,l} \smallotimes sy\smallotimes sc_{1,q}).
\end{align*}

\item If $p>0$ and $q>0$, then
\begin{align*}
\beta^{p,q}_{M \overset{\infty}\smallotimes _A N}(sa_{1,p} \smallotimes sx\smallotimes sb_{1,l} \smallotimes y\smallotimes sc_{1,q}) =0.
\end{align*}
\end{enumerate}
\end{rem}

\begin{rem}\label{rem:tensormorphism}
The tensor product also extends to pre-morphisms. Consider two pre-morphisms in $\EC^{\BB}_\infty(A)$:
\[
f\in \Hom(T \otimes sM \otimes T, sN), \quad f' \in \Hom(T \otimes sM' \otimes T, sN').
 \]
Then $f\overset{\infty}\otimes_A f'\in \Hom(T\otimes sM\otimes T\otimes M'\otimes T,sN\otimes T\otimes N')$ is defined by
 \begin{align}
 \begin{aligned}
& (f \overset{\infty}\otimes_A f') (sa_{1,p}\otimes sx\otimes sb_{1,q} \otimes y \otimes sc_{1,r})\\
 ={} & \sum_{i=0}^q\sum_{j=0}^{q-i} (-1)^{\epsilon} f(sa_{1,p} \otimes sx \otimes sb_{1,i}) \otimes sb_{i+1, i+j} \otimes s^{-1} f'(sb_{i+j+1, q} \otimes sy \otimes sc_{1,r}),
 \end{aligned}
 \end{align}
where $\epsilon=|f'|(|sx|+|sa_{1,p}|+|sb_{1,i+j}|)+|sb_{i+j+1,q}|$.

For $g\in \Hom(T\otimes sN\otimes T,sL)$ and $g'\in \Hom(T\otimes sN'\otimes T,sL')$, one checks that
\begin{align}\label{align:tensorcomposition}
 (f \overset{\infty}{\circ}g) \overset{\infty} \otimes_A (f' \overset{\infty}{\circ} g') = (f \overset{\infty} \otimes_A f') \overset{\infty}\circ (g \overset{\infty} \otimes_A g'),
\end{align}
where $\overset{\infty} \circ$ is the composition of morphisms in $\EC^{\BB}_\infty(A)$; see \eqref{align:compositionmorphisms}.

Consequently, $-\overset{\infty}\otimes_A-$ is a bifunctor on $\EC^{\BB}_\infty(A)$, and the induced tensor product makes $\ED^{\BB}_\infty(A)$ a monoidal category with unit object $A$; see, for example, \cite[Proposition 2]{Fer12}. Moreover, this monoidal triangulated category admits an internal Hom: for any $A_\infty$-bimodules $M$ and $N$, there is an $A_\infty$-bimodule $\overset{\infty}{\Hom}_A(M,N)$, with underlying graded space $\Hom(T^c(sA)\otimes sM,sN)$ and an induced bimodule structure, such that there is a natural isomorphism
\begin{align*}
\Hom_{\EC^{\BB}_\infty(A)}(M \overset{\infty}{\otimes}_A N, K) \xrightarrow{\simeq} \Hom_{\EC^{\BB}_\infty(A)}(M, \overset{\infty}{\Hom}_A(N, K))
\end{align*}
for any $A_\infty$-bimodules $M, N$ and $K$. It follows that $-\overset{\infty}{\otimes}_A-$ commutes with coproducts in $\ED^{\BB}_\infty(A)$. See \cite[Proposition 2.10]{Gan} for a similar statement.
\end{rem}

\begin{rem}\label{rem:tensorfunctor}
Let $(M,\beta_M^{p,q})$ be an $A_\infty$-bimodule over $A$. Tensoring with $M$ defines a dg functor
\[
- \overset{\infty}\otimes_A M \colon \EC^{\BB}_\infty(A) \to \EC^{\BB}_\infty(A)
\]
which sends an object $(N,\beta_N^{p,q})$ to $(N\overset{\infty}\otimes_A M,\beta_{N\overset{\infty}\otimes_A M}^{p,q})$.
On morphisms, a pre-morphism $f\in \Hom(T\otimes sN\otimes T,sK)$ is sent to
\[
f\overset{\infty}\otimes_A M=f\overset{\infty}\otimes_A \mathbf 1_{sM},
\]
where $\mathbf 1_{sM}\colon sM\to sM$ is viewed as a morphism in $\Hom(T\otimes sM\otimes T,sM)$.
The same construction gives a functor on right $A_\infty$-modules:
\[
 - \overset{\infty}\otimes_A M \colon \EC^{\RR}_\infty(A) \to \EC^{\RR}_\infty(A) , \quad N \mapsto N \overset{\infty}\otimes_A M
\]
where the underlying graded space of $N\overset{\infty}\otimes_A M$ is $N\otimes T^c(sA)\otimes M$ and the right $A_\infty$-module structure is induced in the standard way. See \cite[Subsection 6.3]{Kel01}.
\end{rem}

\subsection{An induction from right $A_\infty$-modules to $A_\infty$-bimodules}
Let $(A,m^n)$ be an $A_\infty$-algebra, and let $(T^c(sA),\Delta,\widehat m)$ be its associated dg coalgebra. We also use the opposite dg coalgebra $(T^c(sA),\Delta^{\op},\widehat m)$, where $\Delta^{\op}=\tau\circ\Delta$ and $\tau$ is the graded swap
\[
\tau\colon T^c(sA)\otimes T^c(sA)\xrightarrow{\simeq} T^c(sA)\otimes T^c(sA),
\qquad
x\otimes y\mapsto (-1)^{|x||y|}y\otimes x.
\]
The following lemma is the abstract mechanism behind the induction functor used later.

\begin{lem}\label{lem:bi-rightduality}
Let $(A, m^n)$ be an $A_{\infty}$-algebra.
Then any homomorphism $\varphi\colon T^c(sA)^{\op}\otimes T^c(sA)\to T^c(sA)$ of dg coalgebras induces a dg functor
\[
 e_{\varphi}\colon \EC^{\RR}_\infty(A) \to \EC^{\BB}_\infty(A).
\]
\end{lem}

\begin{proof}
Let $(M,\rho_M^n)$ be a right $A_\infty$-module over $A$. We define $e_\varphi(M)=(M,\beta_M^{p,q})$, where the structure map $\beta_M^{p,q}$ is represented by
\[
\beta_M^{p, q}=\tikzmath{
\draw (-.8, .8) --(-.8, 1);
\draw (0, .8) --(0, 1);
\draw (.8, .8) --(.8, 2.5);

\node at (2, .5) {.};
\node at (-1, .5) {$(sA)^{\otimes p}$};
\node at (0, .5) {$sM$};
\node at (1, .5) {$(sA)^{\otimes q}$};

\draw [knot] (0, 1) to[out=90,in=-90](-.8, 2.5);
\draw [knot] (-.8, 1) to[out=90,in=-90](0, 2.5);

\draw (-.25, 2.5) rectangle (.95, 3);
\node at (0.35, 2.75) {$\varphi$};

\draw (-.8, 2.5) -- (-.8, 4);
\draw (.8, 3) -- (.8, 4);

\draw (-.95, 4) rectangle (.95, 4.5);
\node at (0, 4.25) {$\rho_M$};
}
\]
To see that $(M,\beta_M^{p,q})$ is a bimodule, it is enough to verify $\beta_M\circ \widehat\beta_M=0$. By \eqref{align:betahat},
\[
\widehat\beta_M=\widehat m\otimes \id_{sM}\otimes \id_T+(\id_T\otimes \beta_M\otimes \id_T)\circ(\Delta\otimes \id_{sM}\otimes \Delta)+\id_T\otimes \id_{sM}\otimes \widehat m.
\]
For the summands $\beta_M\circ(\widehat m\otimes \id_{sM}\otimes \id_T+\id_T\otimes \id_{sM}\otimes \widehat m)$ in $\beta_M\circ\widehat\beta_M$, we have
\\\\
$\tikzmath{
\draw (-1.2, 0) -- (-1.2, 1);
\draw (0, 0) -- (0, 1.5);
\draw (1.2, 0) -- (1.2, 3);
\node at (-1.2, -.5) {$(sA)^{\otimes p}$};
\node at (0, -.5) {$sM$};
\node at (1.2, -.5) {$(sA)^{\otimes q}$};

\draw (-1.45, 1) rectangle (-.95, 1.5);
\node at (-1.2, 1.25) {$\widehat{m}$};
\draw [knot] (0, 1.5) to[out=90,in=-90](-1.2, 3);
\draw [knot] (-1.2, 1.5) to[out=90,in=-90](0, 3);

\draw (1.45, 3) rectangle (-.25, 3.5);
\node at (.6, 3.25) {$\varphi$};
\draw (1.2, 3.5) -- (1.2, 4.3);
\draw (-1.2, 3) -- (-1.2, 4.3);

\draw (1.45, 4.3) rectangle (-1.45, 4.8);
\node at (0, 4.55) {$\rho_M$};
}$\hspace{1 mm} + \hspace{1 mm}
$\tikzmath{
\draw (-1.2, 0) -- (-1.2, 1.5);
\draw (0, 0) -- (0, 1.5);
\draw (1.2, 0) -- (1.2, 1);
\draw (1.2, 3) -- (1.2, 1.5);
\node at (-1.2, -.5) {$(sA)^{\otimes p}$};
\node at (0, -.5) {$sM$};
\node at (1.2, -.5) {$(sA)^{\otimes q}$};

\draw (1.45, 1) rectangle (.95, 1.5);
\node at (1.2, 1.25) {$\widehat{m}$};
\draw [knot] (0, 1.5) to[out=90,in=-90](-1.2, 3);
\draw [knot] (-1.2, 1.5) to[out=90,in=-90](0, 3);

\draw (1.45, 3) rectangle (-.25, 3.5);
\node at (.6, 3.25) {$\varphi$};
\draw (1.2, 3.5) -- (1.2, 4.3);
\draw (-1.2, 3) -- (-1.2, 4.3);

\draw (1.45, 4.3) rectangle (-1.45, 4.8);
\node at (0, 4.55) {$\rho_M$};
}$\hspace{1 mm} = \hspace{1 mm}
$\tikzmath{
\draw (-1.2, 1.5) -- (-1.2, 4.3);
\draw (0, 2) -- (0, 3);
\draw (1.2, 0) -- (1.2, 3);

\draw (-.25, 1.5) rectangle (.25, 2);
\node at (0, 1.75) {$\widehat{m}$};
\draw [knot] (0, 0) to[out=90,in=-90](-1.2, 1.5);
\draw [knot] (-1.2, 0) to[out=90,in=-90](0, 1.5);

\draw (1.45, 3) rectangle (-.25, 3.5);
\node at (.6, 3.25) {$\varphi$};
\draw (1.2, 3.5) -- (1.2, 4.3);

\draw (1.45, 4.3) rectangle (-1.45, 4.8);
\node at (0, 4.55) {$\rho_M$};
}$\hspace{1 mm} + \hspace{1 mm}
$\tikzmath{
\draw (-1.2, 0) -- (-1.2, 1.5);
\draw (0, 0) -- (0, 1.5);
\draw (1.2, 0) -- (1.2, 1.5);
\draw (1.2, 3) -- (1.2, 2);

\draw (1.45, 2) rectangle (.95, 1.5);
\node at (1.2, 1.75) {$\widehat{m}$};
\draw [knot] (0, 1.5) to[out=90,in=-90](-1.2, 3);
\draw [knot] (-1.2, 1.5) to[out=90,in=-90](0, 3);

\draw (1.45, 3) rectangle (-.25, 3.5);
\node at (.6, 3.25) {$\varphi$};
\draw (1.2, 3.5) -- (1.2, 4.3);
\draw (-1.2, 3) -- (-1.2, 4.3);

\draw (1.45, 4.3) rectangle (-1.45, 4.8);
\node at (0, 4.55) {$\rho_M$};
}$\hspace{1 mm} = \hspace{1 mm}
$\tikzmath{
\draw (-1.2, 1.5) -- (-1.2, 4.3);
\draw (1, 0) -- (1, 1.5);

\draw [knot] (0, 0) to[out=90,in=-90](-1.2, 1.5);
\draw [knot] (-1.2, 0) to[out=90,in=-90](0, 1.5);

\draw (1.25, 1.5) rectangle (-.25, 2);
\node at (.5, 1.75) {$\varphi$};
\draw (.85, 3) rectangle (.35, 3.5);
\node at (.6, 3.25) {$\widehat{m}$};
\draw (.6, 2) -- (.6, 3);
\draw (.6, 3.5) -- (.6, 4.3);

\draw (1.05, 4.3) rectangle (-1.45, 4.8);
\node at (-.3, 4.55) {$\rho_M$};
\node at (2, 0) {,};
}$
\\
\\
where the second equality uses that $\varphi$ is compatible with the differentials. The remaining summand
\[
\beta_M\circ(\id_T\otimes \beta_M\otimes \id_T)\circ(\Delta\otimes \id_{sM}\otimes \Delta)
\]
of $\beta_M\circ\widehat\beta_M$ is represented by\\
\\
$\tikzmath{
\draw (-2,-0.3) -- (-2,3);
\draw (0.0,-1.3) -- (0., -.3);
\draw (1.,-0.3) -- (1., 1.2);
\draw (2.,-0.3) -- (2.,4.5);
\node at (-1.5, -1.8) {$(sA)^{\otimes p}$};
\node at (0, -1.8) {$sM$};
\node at (1.5, -1.8) {$(sA)^{\otimes q}$};

\draw (-2.25, -0.3) rectangle (-.75, -.8);
\node at (-1.5, -.55) {$\Delta$};
\draw (.75, -0.3) rectangle (2.25, -.8);
\node at (1.5, -.55) {$\Delta$};
\draw (-1.5, -1.3) -- (-1.5, -.8);
\draw (1.5, -1.3) -- (1.5, -.8);

\draw [knot] (0, -.3) to[out=90,in=-90](-1, 1.2);
\draw [knot] (-1, -.3) to[out=90,in=-90](0., 1.2);

\draw (-.25, 1.2) rectangle (1.25, 1.7);
\node at (0.5, 1.45) {$\varphi$};

\draw (-1.25, 2.5) rectangle (.25, 3.0);
\node at (-0.55, 2.75) {$\rho_M$};
\draw (-1.,1.2) -- (-1.,2.5);
\draw (-0.,1.7) -- (-0.,2.5);

\draw [knot] (-1., 3) to[out=90,in=-90](-2., 4.5);
\draw [knot] (-2., 3) to[out=90,in=-90](-1., 4.5);

\draw (-1.25, 4.5) rectangle (2.25, 5);
\node at (0.5, 4.75) {$\varphi$};

\draw (-2.25, 5.8) rectangle (2.25, 6.3);
\node at (0, 6.05) {$\rho_M$};
\draw (-2, 4.5) -- (-2, 5.8);
\draw (2,5) -- (2, 5.8);
}$ \hspace{1 mm} = \hspace{1 mm}
$\tikzmath{
\draw (-2, -1.3) -- (-2, -1.5);
\draw (-.5, -1.3) -- (-.5, -1.5);
\draw (1.5, -1.5) -- (1.5, .7);

\draw [knot] (-.5, -1.3) to[out=90,in=-90](-2, .7);
\draw [knot] (-2, -1.3) to[out=90,in=-90](-.5, .7);
\draw (-2, .7) -- (-2, 4.3);

\draw (-1.25, .7) rectangle (.25, 1.2);
\node at (-.5, .95) {$\Delta$};
\draw (2.25, .7) rectangle (.75, 1.2);
\node at (1.5, .95) {$\Delta$};

\draw (-1, 1.2) -- (-1, 2.5);
\draw (0.0, 1.2) -- (0., 2);
\draw (1., 1.2) -- (1., 2);
\draw (2., 1.2) -- (2., 4.3);

\draw (-.25, 2) rectangle (1.25, 2.5);
\node at (.5, 2.25) {$\varphi$};

\draw [knot] (-1, 2.5) to[out=90,in=-90](.5, 4.3);
\draw [knot] (.5, 2.5) to[out=90,in=-90](-1, 4.3);

\draw (.25, 4.3) rectangle (2.25, 4.8);
\node at (1.25, 4.55) {$\varphi$};
\draw (2, 4.8) -- (2, 5.6);

\draw (-.75, 4.3) rectangle (-2.25, 4.8);
\node at (-1.5, 4.55) {$\rho_M$};
\draw (-2, 4.8) -- (-2, 5.6);

\draw (-2.25, 5.6) rectangle (2.25, 6.1);
\node at (0, 5.85) {$\rho_M$};
}$\hspace{1 mm} = \hspace{1 mm}
$\tikzmath{
\draw (-2, 0) -- (-2, -.1);
\draw (-.5, 0) -- (-.5, -.1);
\draw (1.5, 0) -- (1.5, -.1);
\draw [knot] (-2, 0) to[out=90,in=-90](-.5, 1.5);
\draw [knot] (-.5, 0) to[out=90,in=-90](-2, 1.5);
\draw (1.5, 0) -- (1.5, 1.5);

\draw (-1.25, 1.5) rectangle (.25, 2);
\node at (-.5, 1.75) {$\Delta$};
\draw (2.25, 1.5) rectangle (.75, 2);
\node at (1.5, 1.75) {$\Delta$};

\draw [knot] (-1, 2) to[out=90,in=-90](0, 3.5);
\draw [knot] (0, 2) to[out=90,in=-90](-1, 3.5);
\draw (1, 2) -- (1, 3.5);
\draw (2, 2) -- (2, 5);

\draw [knot] (0, 3.5) to[out=90,in=-90](1, 5);
\draw [knot] (1, 3.5) to[out=90,in=-90](0, 5);
\draw (-1, 3.5) -- (-1, 5);

\draw (-1.25, 5) rectangle (.25, 5.5);
\node at (-.5, 5.25) {$\varphi$};
\draw (2.25, 5) rectangle (.75, 5.5);
\node at (1.5, 5.25) {$\varphi$};
\draw (-2, 1.5) -- (-2, 6);
\draw (-.5, 5.5) -- (-.5, 6);
\draw (2, 5.5) -- (2, 7);

\draw (-2.25, 6) rectangle (-.25, 6.5);
\node at (-1.25, 6.25) {$\rho_M$};
\draw (-2, 6.5) -- (-2, 7);

\draw (-2.25, 7) rectangle (2.25, 7.5);
\node at (0, 7.25) {$\rho_M$};
}$\hspace{1 mm} = \hspace{1 mm}
$\tikzmath{
\draw (-1.2, -.5) -- (-1.2, -.2);
\draw (0, -.5) -- (0, -.2);
\draw (1.2, -.5) -- (1.2, 1.5);

\draw [knot] (-1.2, -.2) to[out=90,in=-90](0, 1.5);
\draw [knot] (0, -.2) to[out=90,in=-90](-1.2, 1.5);
\draw (-1.2, 1.5) -- (-1.2, 4.1);
\draw (-.25, 2) rectangle (1.45, 1.5);
\node at (.6, 1.75) {$\varphi$};

\draw (.6, 2) -- (.6, 2.8);
\draw (-.25, 2.8) rectangle (1.45, 3.3);
\node at (.6, 3.05) {$\Delta$};

\draw (1.2, 3.3) -- (1.2, 5.4);
\draw (0, 3.3) -- (0, 4.1);

\draw (.25, 4.1) rectangle (-1.45, 4.6);
\node at (-.6, 4.35) {$\rho_M$};
\draw (-1.2, 4.6) -- (-1.2, 5.4);
\draw (1.45, 5.4) rectangle (-1.45, 5.9);
\node at (0, 5.65) {$\rho_M$};
\node at (2, -.5) {,};
}$
\\
\\
where the third equality follows because $\varphi$ is a coalgebra morphism. Applying $\rho_M\circ \widehat\rho_M=0$ to the resulting diagrams gives $\beta_M\circ \widehat\beta_M=0$.

Now let $N$ be another right $A_{\infty}$-module over $A$, and let $f\in \Hom(sM\otimes T^c(sA),sN)$. Define $e_\varphi(f)\in \Hom(T^c(sA)\otimes sM\otimes T^c(sA),sN)$ by
$$e_\varphi(f)=\tikzmath{
\draw (-.8, .8) --(-.8, 1);
\draw (0, .8) --(0, 1);
\draw (.8, .8) --(.8, 2.5);

\node at (-1, .5) {$(sA)^{\otimes p}$};
\node at (0, .5) {$sM$};
\node at (1, .5) {$(sA)^{\otimes q}$};

\draw [knot] (0, 1) to[out=90,in=-90](-.8, 2.5);
\draw [knot] (-.8, 1) to[out=90,in=-90](0, 2.5);

\draw (-.25, 2.5) rectangle (.95, 3);
\node at (0.35, 2.75) {$\varphi$};

\draw (-.8, 2.5) -- (-.8, 4);
\draw (.8, 3) -- (.8, 4);

\draw (-.95, 4) rectangle (.95, 4.5);
\node at (0, 4.25) {$f$};
\node at (2, .5) {.};
}$$
Clearly $e_\varphi(\id_M)=\id_M$.
We next check that, for $g\in \Hom(sN\otimes T^c(sA),sK)$,
\[
e_\varphi( g\overset{\infty}{\circ} f) = e_\varphi(g)\overset{\infty}{\circ} e_\varphi(f).\]
Using the definition of $e_{\varphi}$, the left-hand side is represented by
$$\tikzmath{
\draw (-1, .8) --(-1, 1);
\draw (0, .8) --(0, 1);
\draw (1, .8) --(1, 3.4);

\node at (-1, .5) {$(sA)^{\otimes p}$};
\node at (0, .5) {$sM$};
\node at (1, .5) {$(sA)^{\otimes q}$};
\draw [knot] (0, 1) to[out=90,in=-90](-1, 2.3);
\draw [knot] (-1, 1) to[out=90,in=-90](0, 2.3);
\draw (0, 2.3) --(0, 3.4);
\draw (-.25, 3.4) rectangle (1.25, 3.9);
\node at (0.5, 3.65) {$\varphi$};
\draw (-1, 2.3) --(-1, 5);
\draw (.5, 3.9) --(.5, 5);
\draw (-1.25, 5) rectangle (1.25, 5.7);
\node at (0, 5.35) {$g\overset{\infty}{\circ} f$};
}\hspace{1 mm} = \hspace{1 mm}\tikzmath{
\draw (-1, .8) --(-1, 1);
\draw (0, .8) --(0, 1);
\draw (1, .8) --(1, 2.3);
\node at (2, .8) {,};

\draw [knot] (0, 1) to[out=90,in=-90](-1, 2.3);
\draw [knot] (-1, 1) to[out=90,in=-90](0, 2.3);
\draw (-.25, 2.3) rectangle (1.25, 2.8);
\node at (0.5, 2.55) {$\varphi$};
\draw (.5, 2.8) --(.5, 3.5);
\draw (-1, 2.3) --(-1, 4.8);
\draw (-.25, 3.5) rectangle (1.25, 4);
\node at (0.5, 3.75) {$\Delta$};
\draw (0, 4) --(0, 4.8);
\draw (1, 4) --(1, 5.9);
\draw (.25, 4.8) rectangle (-1.25, 5.3);
\node at (-0.5, 5.05) {$f$};
\draw (-.5, 5.3) --(-.5, 5.9);
\draw (-.75, 5.9) rectangle (1.25, 6.4);
\node at (0.25, 6.15) {$g$};
}$$
while $e_\varphi(g)\overset{\infty}{\circ} e_\varphi(f)$ equals
$$\tikzmath{
\draw (-1.5, -1.3) -- (-1.5, -.8);
\draw (1.5, -1.3) -- (1.5, -.8);
\node at (-1.5, -1.8) {$(sA)^{\otimes p}$};
\node at (0, -1.8) {$sM$};
\node at (1.5, -1.8) {$(sA)^{\otimes q}$};

\draw (-2.25, -0.3) rectangle (-.75, -.8);
\node at (-1.5, -.55) {$\Delta$};
\draw (.75, -0.3) rectangle (2.25, -.8);
\node at (1.5, -.55) {$\Delta$};
\draw (-2,-0.3) -- (-2, 2.2);
\draw (0.0, -1.3) -- (0., .7);
\draw (-1.,-0.3) -- (-1., .7);
\draw (1.,-0.3) -- (1., .7);
\draw (2.,-0.3) -- (2., 2.2);
\draw (-1.25, .7) rectangle (1.25, 1.2);
\node at (0, .95) {$e_\varphi(f)$};
\draw (0.0, 1.2) -- (0., 2.2);
\draw (-2.25, 2.7) rectangle (2.25, 2.2);
\node at (0, 2.45) {$e_\varphi(g)$};

}\hspace{1 mm} = \hspace{1 mm}\tikzmath{
\draw (-1.5, -1.3) -- (-1.5, -.8);
\draw (1.5, -1.3) -- (1.5, -.8);

\draw (-2.25, -0.3) rectangle (-.75, -.8);
\node at (-1.5, -.55) {$\Delta$};
\draw (.75, -0.3) rectangle (2.25, -.8);
\node at (1.5, -.55) {$\Delta$};
\draw (-2,-0.3) -- (-2, 2.8);
\draw (0.0, -1.3) -- (0., -.3);
\draw [knot] (0, -.3) to[out=90,in=-90](-1, 1);
\draw [knot] (-1, -.3) to[out=90,in=-90](0, 1);
\draw (1., -0.3) -- (1., 1);

\draw (-.25, 1) rectangle (1.25, 1.5);
\node at (0.5, 1.25) {$\varphi$};
\draw (-1., 1) -- (-1., 2.3);
\draw (.5, 1.5) -- (.5, 2.3);
\draw (.75, 2.3) rectangle (-1.25, 2.8);
\node at (-.25, 2.55) {$f$};
\draw [knot] (-.5, 2.8) to[out=90,in=-90](-2, 4.3);
\draw [knot] (-2, 2.8) to[out=90,in=-90](-.5, 4.3);
\draw (2.,-0.3) -- (2., 4.3);

\draw (-.75, 4.3) rectangle (2.25, 4.8);
\node at (0.7, 4.55) {$\varphi$};
\draw (-2, 4.3) -- (-2, 5.6);
\draw (2, 4.8) -- (2, 5.6);
\draw (-2.25, 5.6) rectangle (2.25, 6.1);
\node at (0, 5.85) {$g$};
}\hspace{1 mm} = \hspace{1 mm}\tikzmath{
\draw (-1, .8) --(-1, 1);
\draw (0, .8) --(0, 1);
\draw (1, .8) --(1, 2.3);
\node at (2, .8) {,};

\draw [knot] (0, 1) to[out=90,in=-90](-1, 2.3);
\draw [knot] (-1, 1) to[out=90,in=-90](0, 2.3);
\draw (-.25, 2.3) rectangle (1.25, 2.8);
\node at (0.5, 2.55) {$\varphi$};
\draw (.5, 2.8) --(.5, 3.5);
\draw (-1, 2.3) --(-1, 4.8);
\draw (-.25, 3.5) rectangle (1.25, 4);
\node at (0.5, 3.75) {$\Delta$};
\draw (0, 4) --(0, 4.8);
\draw (1, 4) --(1, 5.9);
\draw (.25, 4.8) rectangle (-1.25, 5.3);
\node at (-0.5, 5.05) {$f$};
\draw (-.5, 5.3) --(-.5, 5.9);
\draw (-.75, 5.9) rectangle (1.25, 6.4);
\node at (0.25, 6.15) {$g$};
}$$
where the second equality follows because $\varphi$ is a coalgebra morphism. 
Therefore
\[
e_\varphi(g\overset{\infty}{\circ} f)=e_\varphi(g)\overset{\infty}{\circ}e_\varphi(f).
\]

It remains to show that $e_\varphi$ commutes with the differentials, i.e. $e_\varphi(\delta(f))=\delta(e_\varphi(f))$.
On the one hand, by the definitions of $e_{\varphi}$ and $\delta(f)$, the expression $e_\varphi(\delta(f))$ is represented by
$$\tikzmath{
\draw (-1, .8) --(-1, 1);
\draw (0, .8) --(0, 1);
\draw (1, .8) --(1, 3.4);

\node at (-1, .5) {$(sA)^{\otimes p}$};
\node at (0, .5) {$sM$};
\node at (1, .5) {$(sA)^{\otimes q}$};
\draw [knot] (0, 1) to[out=90,in=-90](-1, 2.3);
\draw [knot] (-1, 1) to[out=90,in=-90](0, 2.3);
\draw (0, 2.3) --(0, 3.4);
\draw (-.25, 3.4) rectangle (1.25, 3.9);
\node at (0.5, 3.65) {$\varphi$};
\draw (-1, 2.3) --(-1, 5);
\draw (.5, 3.9) --(.5, 5);
\draw (-1.25, 5) rectangle (1.25, 5.5);
\node at (0, 5.25) {$\delta(f)$};
}\hspace{1 mm} = \hspace{1 mm}\tikzmath{
\draw (-1, .8) --(-1, 1);
\draw (0, .8) --(0, 1);
\draw (1, .8) --(1, 2.3);

\draw [knot] (0, 1) to[out=90,in=-90](-1, 2.3);
\draw [knot] (-1, 1) to[out=90,in=-90](0, 2.3);
\draw (-.25, 2.3) rectangle (1.25, 2.8);
\node at (0.5, 2.55) {$\varphi$};
\draw (1, 2.8) --(1, 4);
\draw (-1, 2.3) --(-1, 4);
\draw (-1.25, 4) rectangle (1.25, 4.7);
\node at (0, 4.35) {$\widehat{f}$};
\draw (1, 4.7) --(1, 5.9);
\draw (-1, 4.7) --(-1, 5.9);
\draw (-1.25, 5.9) rectangle (1.25, 6.4);
\node at (0., 6.15) {$\rho_N$};
}\hspace{1 mm} - (-1)^{|f|} \hspace{1 mm}\tikzmath{
\draw (-1, .8) --(-1, 1);
\draw (0, .8) --(0, 1);
\draw (1, .8) --(1, 2.3);

\draw [knot] (0, 1) to[out=90,in=-90](-1, 2.3);
\draw [knot] (-1, 1) to[out=90,in=-90](0, 2.3);
\draw (-.25, 2.3) rectangle (1.25, 2.8);
\node at (0.5, 2.55) {$\varphi$};
\draw (1, 2.8) --(1, 4);
\draw (-1, 2.3) --(-1, 4);
\draw (-1.25, 4) rectangle (1.25, 4.7);
\node at (0, 4.35) {$\widehat{\rho}_M$};
\draw (1, 4.7) --(1, 5.9);
\draw (-1, 4.7) --(-1, 5.9);
\draw (-1.25, 5.9) rectangle (1.25, 6.4);
\node at (0., 6.15) {$f$};
}\qquad\qquad\qquad\qquad$$
$$\qquad\qquad\qquad\qquad
\hspace{1 mm} = \hspace{1 mm}\tikzmath{
\draw (-1, .8) --(-1, 1);
\draw (0, .8) --(0, 1);
\draw (1, .8) --(1, 2.3);

\draw [knot] (0, 1) to[out=90,in=-90](-1, 2.3);
\draw [knot] (-1, 1) to[out=90,in=-90](0, 2.3);
\draw (-.25, 2.3) rectangle (1.25, 2.8);
\node at (0.5, 2.55) {$\varphi$};
\draw (.5, 2.8) --(.5, 3.5);
\draw (-1, 2.3) --(-1, 4.8);
\draw (-.25, 3.5) rectangle (1.25, 4);
\node at (0.5, 3.75) {$\Delta$};
\draw (0, 4) --(0, 4.8);
\draw (1, 4) --(1, 5.9);
\draw (.25, 4.8) rectangle (-1.25, 5.3);
\node at (-0.5, 5.05) {$f$};
\draw (-.5, 5.3) --(-.5, 5.9);
\draw (-.75, 5.9) rectangle (1.25, 6.4);
\node at (0.25, 6.15) {$\rho_N$};
}\hspace{1 mm} - (-1)^{|f|} \hspace{1 mm}\tikzmath{
\draw (-1, .8) --(-1, 1);
\draw (0, .8) --(0, 1);
\draw (1, .8) --(1, 2.3);

\draw [knot] (0, 1) to[out=90,in=-90](-1, 2.3);
\draw [knot] (-1, 1) to[out=90,in=-90](0, 2.3);
\draw (-.25, 2.3) rectangle (1.25, 2.8);
\node at (0.5, 2.55) {$\varphi$};
\draw (1, 2.8) --(1, 4);
\draw (-1, 2.3) --(-1, 4.7);
\draw (.75, 4) rectangle (1.25, 4.7);
\node at (1, 4.35) {$\widehat{m}$};
\draw (1, 4.7) --(1, 5.9);
\draw (-1, 4.7) --(-1, 5.9);
\draw (-1.25, 5.9) rectangle (1.25, 6.4);
\node at (0., 6.15) {$f$};
}\hspace{1 mm} - (-1)^{|f|} \hspace{1 mm}\tikzmath{
\draw (-1, .8) --(-1, 1);
\draw (0, .8) --(0, 1);
\draw (1, .8) --(1, 2.3);
\node at (2, .8) {,};

\draw [knot] (0, 1) to[out=90,in=-90](-1, 2.3);
\draw [knot] (-1, 1) to[out=90,in=-90](0, 2.3);
\draw (-.25, 2.3) rectangle (1.25, 2.8);
\node at (0.5, 2.55) {$\varphi$};
\draw (.5, 2.8) --(.5, 3.5);
\draw (-1, 2.3) --(-1, 4.8);
\draw (-.25, 3.5) rectangle (1.25, 4);
\node at (0.5, 3.75) {$\Delta$};
\draw (0, 4) --(0, 4.8);
\draw (1, 4) --(1, 5.9);
\draw (.25, 4.8) rectangle (-1.25, 5.3);
\node at (-0.5, 5.05) {$\rho_M$};
\draw (-.5, 5.3) --(-.5, 5.9);
\draw (-.75, 5.9) rectangle (1.25, 6.4);
\node at (0.25, 6.15) {$f$};
}$$
where the second equality follows from the definition of the extension $\widehat{f}$ as a right comodule morphism and the extension $\widehat{\rho_M}$ as a right comodule coderivation.
On the other hand, the expression $\delta(e_\varphi(f))$ is represented by
$$\tikzmath{
\draw (-1, .8) --(-1, 2.3);
\draw (0, .8) --(0, 2.3);
\draw (1, .8) --(1, 2.3);

\node at (-1, .5) {$(sA)^{\otimes p}$};
\node at (0, .5) {$sM$};
\node at (1, .5) {$(sA)^{\otimes q}$};
\draw (-1.25, 2.3) rectangle (1.25, 3);
\node at (0, 2.65) {$\widehat{e_\varphi(f)}$};
\draw (., 3) --(., 3.9);
\draw (-1.25, 3.9) rectangle (1.25, 4.4);
\node at (0, 4.15) {$\beta_N$};
}\hspace{1 mm} - (-1)^{|f|} \hspace{1 mm}\tikzmath{
\draw (-1, .8) --(-1, 2.3);
\draw (0, .8) --(0, 2.3);
\draw (1, .8) --(1, 2.3);

\node at (-1, .5) {$(sA)^{\otimes p}$};
\node at (0, .5) {$sM$};
\node at (1, .5) {$(sA)^{\otimes q}$};
\draw (-1.25, 2.3) rectangle (1.25, 3);
\node at (0, 2.65) {$\widehat{\beta}_M$};
\draw (., 3) --(., 3.9);
\draw (-1.25, 3.9) rectangle (1.25, 4.4);
\node at (0, 4.15) {$e_\varphi(f)$};
}\hspace{1 mm} = \hspace{1 mm}\tikzmath{
\draw (-2,-0.3) -- (-2,3);
\draw (0.0,-1.3) -- (0., -.3);
\draw (1.,-0.3) -- (1., 1.2);
\draw (2.,-0.3) -- (2.,4.5);
\draw (-1.5, -1.3) -- (-1.5, -.8);
\draw (1.5, -1.3) -- (1.5, -.8);

\draw (-2.25, -0.3) rectangle (-.75, -.8);
\node at (-1.5, -.55) {$\Delta$};
\draw (.75, -0.3) rectangle (2.25, -.8);
\node at (1.5, -.55) {$\Delta$};
\draw [knot] (0, -.3) to[out=90,in=-90](-1, 1.2);
\draw [knot] (-1, -.3) to[out=90,in=-90](0., 1.2);
\draw (-.25, 1.2) rectangle (1.25, 1.7);
\node at (.5, 1.45) {$\varphi$};
\draw (-1., 2.5) -- (-1., 1.2);
\draw (0.5, 2.5) -- (0.5, 1.7);
\draw (.75, 2.5) rectangle (-1.25, 3);
\node at (-.25, 2.75) {$f$};
\draw [knot] (-.3, 3) to[out=90,in=-90](-2, 4.5);
\draw [knot] (-2, 3) to[out=90,in=-90](-.3, 4.5);
\draw (-.55, 4.5) rectangle (2.25, 5);
\node at (.85, 4.75) {$\varphi$};
\draw (-2, 4.5) -- (-2, 5.8);
\draw (.85, 5) -- (.85, 5.8);
\draw (-2.25, 5.8) rectangle (2.25, 6.3);
\node at (., 6.05) {$\rho_N$};
}$$
$$\hspace{1 mm} - (-1)^{|f|} \hspace{1 mm}\tikzmath{
\draw (-1, .8) --(-1, 1);
\draw (0, .8) --(0, 1);
\draw (1, .8) --(1, 4.2);

\draw [knot] (0, 1) to[out=90,in=-90](-1, 2.3);
\draw [knot] (-1, 1) to[out=90,in=-90](0, 2.3);
\draw (-.25, 2.3) rectangle (.25, 3);
\node at (0, 2.65) {$\widehat{m}$};
\draw (-.25, 4.2) rectangle (1.25, 4.7);
\node at (0.5, 4.45) {$\varphi$};
\draw (0, 3) --(0, 4.2);
\draw (-1, 2.3) --(-1, 5.9);
\draw (1, 4.7) --(1, 5.9);
\draw (-1.25, 5.9) rectangle (1.25, 6.4);
\node at (0., 6.15) {$f$};
}\hspace{1 mm} - (-1)^{|f|} \hspace{1 mm}\tikzmath{
\draw (-2,-0.3) -- (-2,3);
\draw (0.0,-1.3) -- (0., -.3);
\draw (1.,-0.3) -- (1., 1.2);
\draw (2.,-0.3) -- (2.,4.5);
\draw (-1.5, -1.3) -- (-1.5, -.8);
\draw (1.5, -1.3) -- (1.5, -.8);

\draw (-2.25, -0.3) rectangle (-.75, -.8);
\node at (-1.5, -.55) {$\Delta$};
\draw (.75, -0.3) rectangle (2.25, -.8);
\node at (1.5, -.55) {$\Delta$};
\draw [knot] (0, -.3) to[out=90,in=-90](-1, 1.2);
\draw [knot] (-1, -.3) to[out=90,in=-90](0., 1.2);
\draw (-.25, 1.2) rectangle (1.25, 1.7);
\node at (.5, 1.45) {$\varphi$};
\draw (-1., 2.5) -- (-1., 1.2);
\draw (0.5, 2.5) -- (0.5, 1.7);
\draw (.75, 2.5) rectangle (-1.25, 3);
\node at (-.25, 2.75) {$\rho_M$};
\draw [knot] (-.3, 3) to[out=90,in=-90](-2, 4.5);
\draw [knot] (-2, 3) to[out=90,in=-90](-.3, 4.5);
\draw (-.55, 4.5) rectangle (2.25, 5);
\node at (.85, 4.75) {$\varphi$};
\draw (-2, 4.5) -- (-2, 5.8);
\draw (.85, 5) -- (.85, 5.8);
\draw (-2.25, 5.8) rectangle (2.25, 6.3);
\node at (., 6.05) {$f$};
}\hspace{1 mm} - (-1)^{|f|} \hspace{1 mm}\tikzmath{
\draw (-1, .8) --(-1, 1);
\draw (0, .8) --(0, 1);
\draw (1, .8) --(1, 2.3);
\node at (2, .8) {,};

\draw [knot] (0, 1) to[out=90,in=-90](-1, 2.3);
\draw [knot] (-1, 1) to[out=90,in=-90](0, 2.3);
\draw (1.25, 2.3) rectangle (.75, 3);
\node at (1, 2.65) {$\widehat{m}$};
\draw (-.25, 4.2) rectangle (1.25, 4.7);
\node at (0.5, 4.45) {$\varphi$};
\draw (0, 2.3) --(0, 4.2);
\draw (-1, 2.3) --(-1, 5.9);
\draw (1, 3) --(1, 4.2);
\draw (1, 4.7) --(1, 5.9);
\draw (-1.25, 5.9) rectangle (1.25, 6.4);
\node at (0., 6.15) {$f$};
}
$$
where the second equality follows from the defining formulas of $e_{\varphi(f)}$ and $\beta_M$ associated with the right $A_{\infty}$-module $M$.
Since $\varphi$ is a morphism of dg coalgebras, the two expressions agree. Hence $e_\varphi(\delta(f))=\delta(e_\varphi(f))$, and $e_\varphi$ is a dg functor.
\end{proof}

\begin{rem}\label{rem:question}
It is natural to ask when an $A_\infty$-algebra admits a dg coalgebra homomorphism
\[
T^c(sA)^{\op}\otimes T^c(sA)\longrightarrow T^c(sA).
\]
In the next section, we show that every $B_\infty$-algebra has such a homomorphism; see Proposition \ref{prop:dualitybinfinity}.
 \end{rem}

\section{\texorpdfstring{$B_{\infty}$}{}-algebras and an induction functor}\label{section:Binfinity}

The purpose of this section is to extract from a $B_\infty$-structure on $A$ a concrete induction functor from right $A_\infty$-modules to $A_\infty$-bimodules. We begin by recalling the cofree dg Hopf algebra attached to a $B_\infty$-algebra. We then isolate two elementary identities, Lemma \ref{lem:cofreehopfidentity} and Proposition \ref{prop:concatenationidentity}. 


\begin{defn}[{\cite[Definition 5.2]{GJ}}]
A {\it $B_{\infty}$-algebra structure} on a graded vector space $A$ is a dg bialgebra structure $(T^c(sA), \Delta, \widehat{m}, \widehat{\mu})$ such that $1_\Bbbk \in \Bbbk = (sA)^{\otimes 0}$ is the unit for $\widehat{\mu}$ and $\widehat{m}(1_\Bbbk)=0$.
\end{defn}

In particular, a $B_{\infty}$-algebra has an underlying $A_{\infty}$-algebra structure induced by the dg coalgebra $(T^c(sA),\Delta,\widehat m)$.
Since the tensor coalgebra is cofree, the multiplication $\widehat{\mu}$ is uniquely determined by a collection $(\mu^{p,q})_{p,q\geq 0}$ of degree-zero maps
\begin{align*}
\mu^{p, q} & \colon (sA)^{\otimes p} \bigotimes (sA)^{\otimes q} \to sA
\end{align*}
satisfying the usual compatibility relations; see \cite[Subsection 2.2]{Vor}.
In particular, $\mu^{1,0}=\mathbf{1}_{sA}=\mu^{0,1}$ and $\mu^{0,k}=0=\mu^{k,0}$ for $k\neq 1$.
Throughout this section we assume that the underlying $A_\infty$-algebra of $A$ is strictly unital, with unit $e_A$, and that for $p,q\geq 1$
\[
\mu^{p,q}(sa_{1,p}\bigotimes sa_{p+1,p+q})=0
\]
whenever $a_i=e_A$ for some $1\leq i\leq p+q$. We write $(A,m^n;\mu^{p,q})$ for such a $B_\infty$-algebra. The symbol $\bigotimes$ is used only to separate the two tensor coalgebra factors in the domain of $\mu^{p,q}$.

\begin{rem}\label{rem:widehatmu}
The coalgebra map $\widehat\mu$ is recovered from the maps $\mu^{p,q}$ by the formula
\begin{align*}
\widehat \mu(sa_{1, p} \bigotimes sb_{1,q}) = {} & \sum (-1)^{\epsilon} \mu^{i_1, j_1}(sa_{1, i_1} \bigotimes sb_{1, j_1}) \otimes \mu^{i_2, j_2}(sa_{i_1+1,i_1+i_2} \bigotimes sb_{j_1+1,j_1+j_2}) \\
 & \otimes \dotsb \otimes \mu^{i_k, j_k} (sa_{i_1+\dotsb+i_{k-1}+1, i_1+\cdots+i_k}\bigotimes sb_{j_1+\dotsb+j_{k-1}+1, j_1+\cdots+j_k}),
\end{align*}
where the sum is taken over all sequences of nonnegative integers $\{i_1,i_2,\dotsc,i_k\}$ and $\{j_1,j_2,\dotsc,j_k\}$ such that $i_1+\dotsb+i_k=p$ and $j_1+\dotsb+j_k=q$. This is a finite sum because $\mu^{0,0}=0$.

This formula gives the following recursive expression of $\widehat{\mu}$, which will be used later:
\begin{align}\label{align:widehatmurecursive}
\begin{aligned}
\widehat \mu (sa_{1, p} \bigotimes sb_{1,q})
= {} \sum_{l=0}^p \sum_{j_{l+1}=1}^q & (-1)^\epsilon sa_{1,l} \otimes \mu^{i_{l+1}, j_{l+1}}(sa_{l+1, l+i_{l+1}} \bigotimes sb_{1, j_{l+1}}) \otimes \\
& \widehat \mu(sa_{l+i_{l+1}+1, p} \bigotimes sb_{j_{l+1}+1, q}).
\end{aligned}
\end{align}
To deduce this identity from the preceding formula, single out the first index $l$ such that $j_{l+1}\neq 0$. This forces $j_1=j_2=\dotsb=j_l=0$, and hence $i_1=i_2=\dotsb=i_l=1$, since $\mu^{k,0}=0$ for $k>1$.

The signs in these identities are determined by the usual Koszul rule. More precisely, we use the following block permutation of variables; see \cite[Section 5.2]{CLW}. Let $r\geq 1$ and $l,n\geq 0$.
For any two sequences of nonnegative integers $(l_1, l_2, \dots, l_r)$ and $(n_1, n_2, \dots, n_r)$ satisfying $l=l_1+\dots+l_r$ and $n=n_1+\dots+n_r$, we define a $\Bbbk$-linear map
\begin{align*}
\tau_{(l_1,\dots, l_r;n_1, \dots, n_r)}\colon (sA)^{\otimes l}\bigotimes (sA)^{\otimes n} \rightarrow & ((sA)^{\otimes l_1}\bigotimes (sA)^{\otimes n_1})\otimes\cdots\otimes ((sA)^{\otimes l_r}\bigotimes (sA)^{\otimes n_r})\\
(sa_{1,l})\bigotimes (sb_{1,n}) \mapsto & (-1)^{\epsilon} (sa_{1,l_1}\bigotimes sb_{1,n_1}) \otimes (sa_{l_1+1, l_2} \bigotimes sb_{n_1+1, n_2})\\
& \otimes \dotsb\otimes (sa_{l_1+\cdots + l_{r-1}+1, l}\bigotimes sb_{n_1+\cdots + n_{r-1}+1, n}),
\end{align*}
where $\epsilon=\sum_{i=0}^{r-2} (|sb_{n_1+\cdots+n_i+1}|+\cdots+|sb_{n_1+\cdots+n_{i+1}}|)(|sa_{l_1+\cdots+l_{i+1}+1}|+\cdots+|sa_l|)$.

\end{rem}

The dg bialgebra $(T^c(sA),\Delta,\widehat m,\widehat\mu)$ is connected with respect to the tensor-length filtration, and hence it admits an antipode
$S\colon T^c(sA)\to T^c(sA)$; see, for example, \cite[Theorem 5.10]{CLW}. We refer to \cite{GRT} for a recent discussion of the relationship between dg Hopf algebras and $B_\infty$-algebras.
Thus $(T^c(sA),\Delta,\widehat m,\widehat\mu,S)$ is a dg Hopf algebra whose underlying coalgebra is cofree. We shall use the following standard Hopf identities:
\begin{align}
\Delta\circ (\Delta\otimes \id) & =\Delta \circ (\id\otimes \Delta) \, ,\label{coassociativity}\\
\widehat{\mu}\circ (\widehat{\mu}\otimes \id) & =\widehat{\mu}\circ (\id\otimes \widehat{\mu}) \, ,\label{associativity}\\
\qquad\Delta \circ \widehat{\mu}& = (\widehat{\mu}\otimes \widehat{\mu})\circ(\id\otimes \tau\otimes \id)\circ(\Delta\otimes \Delta) \, , \label{Delta:algebramorphism} \\
 \Delta^{\op} \circ S &= (S\otimes S)\circ \Delta \, , \label{S:anti-coalg} \\
S\circ \widehat{\mu} & =\widehat{\mu}^{\opp} \circ (S\otimes S) \, , \label{S:anti-alg}\\
S\circ \widehat{m} & =\widehat{m}\circ S\, , \label{S:differential}\\
 \widehat{\mu}(S\otimes \id) \Delta (sa_{1,n}) &= \varepsilon(sa_{1,n})=0 \quad \text{for $ sa_{1,n}\in (sA)^{\otimes n}$ with $n\geq$ 1}.\label{S:antipodeaxiom}
\end{align}
Similarly, $\widehat{\mu}(\id\otimes S)\Delta(sa_{1,n})=0$ for $n\geq 1$.
For simplicity, set
\[
\Delta^{(2)}:=\Delta\circ(\Delta\otimes\id)=\Delta\circ(\id\otimes\Delta).
\]
We will use graphical presentations of these identities below.

The antipode is determined inductively by $S(1_\Bbbk)=1_\Bbbk$ and
\[S(sa_{1,n})=\begin{cases}
-sa_1 & \text{for $n=1,$}\\
-sa_{1,n} -\sum\limits_{i=1}^{n-1} \widehat{\mu} (sa_{1,i} \bigotimes S(sa_{i+1,n})) & \text{for $n>1.$}
\end{cases}
\]

We denote the $n$-th projection of $S$ onto $sA$ by
\[
S_n \colon (sA)^{\otimes n} \hookrightarrow T^c(sA) \xrightarrow{S} T^c(sA) \xrightarrow{\mathrm{proj}} sA.
\]
In particular, $S_0=0$. Since $T^c(sA)$ is cofree as a coalgebra, the antipode is reconstructed from its components by
\begin{align*}
S(sa_{1,n}) ={} \sum_{\substack{r\geq 1\\ i_1+\dotsb + i_r=n}} (-1)^{\epsilon} S_{i_r}( sa_{i_1+\dotsb+i_{r-1}+1, i_1+\dotsb+i_{r}}) \smallotimes \dotsb\smallotimes S_{i_2}(sa_{i_1+1, i_1+i_2}) \smallotimes S_{i_1} (sa_{1,i_1})
\end{align*}
for any $n>0$. Here the sign $\epsilon$ is given by the usual Koszul rule.

We now record two elementary Hopf-algebra identities needed later. They do not depend on the special brace operations defining a $B_\infty$-algebra: the same statements hold for any dg Hopf algebra $T^c(V)$ whose underlying coalgebra is the tensor coalgebra, with deconcatenation coproduct and antipode $S$. We state them in the notation $V=sA$.

\begin{lem}\label{lem:cofreehopfidentity}
Let $(T^c(sA),\Delta,\widehat m,\widehat\mu,S)$ be any cofree dg Hopf algebra. For any fixed $sa_{1,i}, sb_{1,j}, sc_{1,k}\in T^c(sA)$, we have 
\begin{align*}
& \sum_{t=0}^i (-1)^{\epsilon} (\id\otimes\widehat{\mu})(\id\otimes \id \otimes  S) (\Delta\otimes \id)(\widehat{\mu}(S(sa_{t+1,i})\bigotimes sb_{1,j}) \otimes \widehat{\mu}(S(sa_{1,t})\bigotimes sc_{1,k})) \\
= {} & \sum_{l=0}^j \widehat{\mu}(S(sa_{1,i})\bigotimes sb_{1,l})\otimes \widehat{\mu}(S(sb_{l+1, j})\bigotimes sc_{1,k}),
\end{align*}
where $\epsilon= |sa_{1,t}|(|sa_{t+1,i}|+|sb_{1,j}|)$.
Equivalently, we have the following graphical presentation
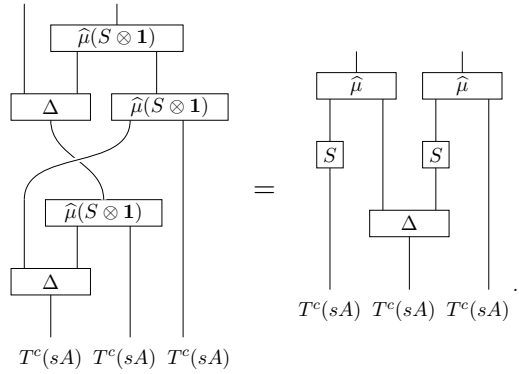
\begin{figure}[H]
\centering$$\tikzmath{
\draw (-1.5, .4) --(-1.5, 1.2);
\draw (0, 0.4) --(0, 2.5);
\draw (1, 0.4) --(1, 4.5);
\node at (-1.5, 0) {$T^c(sA)$};
\node at (-.1, 0) {$T^c(sA)$};
\node at (1.3, 0) {$T^c(sA)$};

\draw (-2.25, 1.2) rectangle (-.75, 1.7);
\node at (-1.5, 1.45) {$\Delta$};
\draw (-1, 1.7) --(-1, 2.5);
\draw (-2, 1.7) --(-2, 3);
\draw (-1.6, 2.5) rectangle (.6, 3);
\node at (-.5, 2.75) {$\widehat{\mu}(S\otimes \id)$};

\draw [knot] (-.5, 3) to[out=90,in=-90](-1.5, 4.5);
\draw [knot] (-2, 3) to[out=90,in=-90](0, 4.5);

\draw (-2.25, 4.5) rectangle (-.75, 5);
\node at (-1.5, 4.75) {$\Delta$};
\draw (-.35, 4.5) rectangle (1.85, 5);
\node at (.75, 4.75) {$\widehat{\mu}(S\otimes \id)$};
\draw (.5, 5) --(.5, 5.8);
\draw (-.25, 6.3) --(-.25, 6.7);
\draw (-1, 5) --(-1, 5.8);
\draw (-2, 5) --(-2, 6.7);
\draw (1., 5.8) rectangle (-1.5, 6.3);
\node at (-.25, 6.05) {$\widehat{\mu}(S\otimes \id)$};
}\hspace{1 mm} = \hspace{1 mm}\tikzmath{
\draw (-1.5, 0) --(-1.5, 2.3);
\draw (1.5, 0) --(1.5, 3.6);
\draw (0, 0) --(0, 1);
\node at (-1.5, -.4) {$T^c(sA)$};
\node at (-.1, -.4) {$T^c(sA)$};
\node at (1.3, -.4) {$T^c(sA)$};
\node at (2, 0) {.};

\draw (-1.25, 2.3) rectangle (-1.75, 2.8);
\node at (-1.5, 2.55) {$S$};
\draw (-1.5, 2.8) --(-1.5, 3.6);
\draw (-.75, 1.5) rectangle (.75, 1);
\node at (0, 1.25) {$\Delta$};
\draw (-.5, 1.5) --(-.5, 3.6);
\draw (.5, 2.3) --(.5, 1.5);

\draw (-1.75, 3.6) rectangle (-.25, 4.1);
\node at (-1, 3.85) {$\widehat{\mu}$};
\draw (.25, 2.3) rectangle (.75, 2.8);
\node at (.5, 2.55) {$S$};
\draw (-1, 4.1) --(-1, 4.5);
\draw (.5, 2.8) --(.5, 3.6);
\draw (1.75, 3.6) rectangle (.25, 4.1);
\node at (1, 3.85) {$\widehat{\mu}$};
\draw (1, 4.1) --(1, 4.5);
}$$
\end{figure}
\end{lem}

\begin{proof}
This is a formal consequence of the Hopf identities \eqref{Delta:algebramorphism}, \eqref{S:anti-coalg}, \eqref{S:anti-alg}, \eqref{associativity}, \eqref{coassociativity}, and \eqref{S:antipodeaxiom}. In graphical form, the calculation is 
$$\tikzmath{
\draw (-1.5, 0) --(-1.5, 1.2);
\draw (0, 0) --(0, 3.8);
\draw (1, 0) --(1, 7.1);
\node at (-1.5, -.4) {$T^c(sA)$};
\node at (-.1, -.4) {$T^c(sA)$};
\node at (1.3, -.4) {$T^c(sA)$};

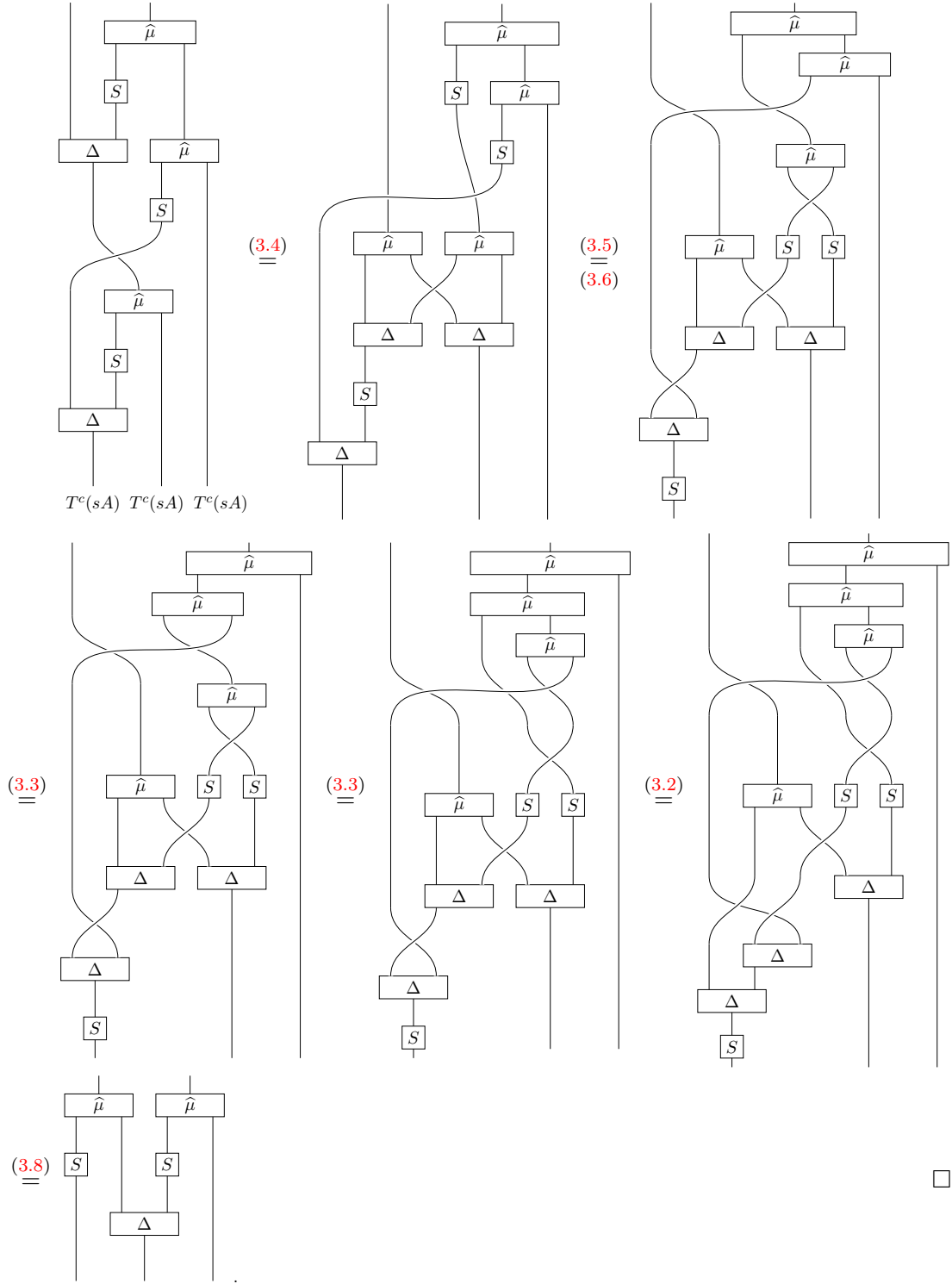
\draw (-2.25, 1.2) rectangle (-.75, 1.7);
\node at (-1.5, 1.45) {$\Delta$};
\draw (-1, 1.7) --(-1, 2.5);
\draw (-2, 1.7) --(-2, 4.3);
\draw (-1.25, 2.5) rectangle (-.75, 3);
\node at (-1, 2.75) {$S$};
\draw (-1, 3) --(-1, 3.8);
\draw (-1.25, 3.8) rectangle (.25, 4.3);
\node at (-.5, 4.05) {$\widehat{\mu}$};
\draw [knot] (-.5, 4.3) to[out=90,in=-90](-1.5, 5.8);
\draw [knot] (-2, 4.3) to[out=90,in=-90](0, 5.8);
\draw (-1.5, 5.8) --(-1.5, 7.1);
\draw (-2.25, 7.1) rectangle (-.75, 7.6);
\node at (-1.5, 7.35) {$\Delta$};
\draw (-2, 7.6) --(-2, 10.6);

\draw (.25, 5.8) rectangle (-.25, 6.3);
\node at (0, 6.05) {$S$};
\draw (0, 6.3) --(0, 7.1);
\draw (-.25, 7.1) rectangle (1.25, 7.6);
\node at (.5, 7.35) {$\widehat{\mu}$};

\draw (-1, 7.6) --(-1, 8.4);
\draw (-1.25, 8.4) rectangle (-.75, 8.9);
\node at (-1, 8.65) {$S$};
\draw (-1, 8.9) --(-1, 9.7);
\draw (.75, 9.7) rectangle (-1.25, 10.2);
\node at (-.25, 9.95) {$\widehat{\mu}$};
\draw (.5, 7.6) --(.5, 9.7);
\draw (-.25, 10.2) --(-.25, 10.6);
}\hspace{-2 mm} \overset{\eqref{Delta:algebramorphism}}{=}\hspace{2 mm}\tikzmath{
\draw (-1.5, 0) --(-1.5, 1.2);
\draw (1.5, 0) --(1.5, 3.8);
\draw (3, 0) --(3, 9.1);

\draw (-2.25, 1.2) rectangle (-.75, 1.7);
\node at (-1.5, 1.45) {$\Delta$};
\draw (-2, 1.7) --(-2, 6.3);
\draw (-1, 1.7) --(-1, 2.5);
\draw (-1.25, 2.5) rectangle (-.75, 3);
\node at (-1, 2.75) {$S$};
\draw (-1, 3.8) --(-1, 3);
\draw (-1.25, 3.8) rectangle (.25, 4.3);
\node at (-.5, 4.05) {$\Delta$};
\draw (2.25, 3.8) rectangle (.75, 4.3);
\node at (1.5, 4.05) {$\Delta$};
\draw (-1, 4.3) --(-1, 5.8);
\draw (2, 4.3) --(2, 5.8);
\draw [knot] (1, 4.3) to[out=90,in=-90](0, 5.8);
\draw [knot] (0, 4.3) to[out=90,in=-90](1, 5.8);
\draw (.25, 5.8) rectangle (-1.25, 6.3);
\node at (-.5, 6.05) {$\widehat{\mu}$};
\draw (.75, 5.8) rectangle (2.25, 6.3);
\node at (1.5, 6.05) {$\widehat{\mu}$};

\draw [knot] (-0.5, 6.3) to[out=90,in=-90](-.5, 11.3);
\draw [knot] (1.5, 6.3) to[out=90,in=-90](1, 9.1);
\draw [knot] (-2, 6.3) to[out=90,in=-90](2, 7.8);
\draw (2.25, 7.8) rectangle (1.75, 8.3);
\node at (2, 8.05) {$S$};
\draw (2, 8.3) --(2, 9.1);
\draw (3.25, 9.1) rectangle (1.75, 9.6);
\node at (2.5, 9.35) {$\widehat{\mu}$};
\draw (1.25, 9.1) rectangle (.75, 9.6);
\node at (1, 9.35) {$S$};
\draw (1, 9.6) --(1, 10.4);
\draw (2.5, 9.6) --(2.5, 10.4);
\draw (3.25, 10.4) rectangle (.75, 10.9);
\node at (2, 10.65) {$\widehat{\mu}$};
\draw (2, 10.9) --(2, 11.3);
}\hspace{2 mm}\overset{\eqref{S:anti-coalg}}{\underset{\eqref{S:anti-alg}}{=}}
\hspace{2 mm}\tikzmath{
\draw (-1.5, 0.8) --(-1.5, 1.2);
\draw (1.5, 0.8) --(1.5, 4.5);
\draw (3, 0.8) --(3, 10.5);

\draw (-1.25, 1.2) rectangle (-1.75, 1.7);
\node at (-1.5, 1.45) {$S$};
\draw (-1.5, 1.7) --(-1.5, 2.5);
\draw (-2.25, 2.5) rectangle (-.75, 3);
\node at (-1.5, 2.75) {$\Delta$};
\draw (-2, 4.5) --(-2, 9);
\draw [knot] (-1, 3) to[out=90,in=-90](-2, 4.5);
\draw [knot] (-2, 3) to[out=90,in=-90](-1, 4.5);

\draw (-1.25, 4.5) rectangle (.25, 5);
\node at (-.5, 4.75) {$\Delta$};
\draw (2.25, 4.5) rectangle (.75, 5);
\node at (1.5, 4.75) {$\Delta$};
\draw (-1, 5) --(-1, 6.5);
\draw (2, 5) --(2, 6.5);
\draw [knot] (1, 5) to[out=90,in=-90](0, 6.5);
\draw [knot] (0, 5) to[out=90,in=-90](1, 6.5);
\draw (.25, 6.5) rectangle (-1.25, 7);
\node at (-.5, 6.75) {$\widehat{\mu}$};
\draw (-.5, 7) --(-.5, 9);
\draw (1.25, 6.5) rectangle (.75, 7);
\node at (1, 6.75) {$S$};
\draw (2.25, 6.5) rectangle (1.75, 7);
\node at (2, 6.75) {$S$};
\draw [knot] (2, 7) to[out=90,in=-90](1, 8.5);
\draw [knot] (1, 7) to[out=90,in=-90](2, 8.5);
\draw (.75, 8.5) rectangle (2.25, 9);
\node at (1.5, 8.75) {$\widehat{\mu}$};

\draw [knot] (-.5, 9) to[out=90,in=-90](-2, 10.5);
\draw [knot] (1.5, 9) to[out=90,in=-90](0, 10.5);
\draw [knot] (-2, 9) to[out=90,in=-90](1.5, 10.5);
\draw (3.25, 10.5) rectangle (1.25, 11);
\node at (2.25, 10.75) {$\widehat{\mu}$};
\draw (0, 10.5) --(0, 11.4);
\draw (2.25, 11) --(2.25, 11.4);
\draw (-.25, 11.4) rectangle (2.5, 11.9);
\node at (1.15, 11.65) {$\widehat{\mu}$};
\draw (-2, 10.5) --(-2, 12.1);
\draw (1.15, 11.9) --(1.15, 12.1);
}$$
$$\hspace{1 mm} \overset{\eqref{associativity}}{=} \hspace{1 mm}\tikzmath{
\draw (-1.5, 0.8) --(-1.5, 1.2);
\draw (1.5, 0.8) --(1.5, 4.5);
\draw (3, 0.8) --(3, 11.4);

\draw (-1.25, 1.2) rectangle (-1.75, 1.7);
\node at (-1.5, 1.45) {$S$};
\draw (-1.5, 1.7) --(-1.5, 2.5);
\draw (-2.25, 2.5) rectangle (-.75, 3);
\node at (-1.5, 2.75) {$\Delta$};
\draw (-2, 4.5) --(-2, 9);
\draw [knot] (-1, 3) to[out=90,in=-90](-2, 4.5);
\draw [knot] (-2, 3) to[out=90,in=-90](-1, 4.5);

\draw (-1.25, 4.5) rectangle (.25, 5);
\node at (-.5, 4.75) {$\Delta$};
\draw (2.25, 4.5) rectangle (.75, 5);
\node at (1.5, 4.75) {$\Delta$};
\draw (-1, 5) --(-1, 6.5);
\draw (2, 5) --(2, 6.5);
\draw [knot] (1, 5) to[out=90,in=-90](0, 6.5);
\draw [knot] (0, 5) to[out=90,in=-90](1, 6.5);
\draw (.25, 6.5) rectangle (-1.25, 7);
\node at (-.5, 6.75) {$\widehat{\mu}$};
\draw (-.5, 7) --(-.5, 9);
\draw (1.25, 6.5) rectangle (.75, 7);
\node at (1, 6.75) {$S$};
\draw (2.25, 6.5) rectangle (1.75, 7);
\node at (2, 6.75) {$S$};
\draw [knot] (2, 7) to[out=90,in=-90](1, 8.5);
\draw [knot] (1, 7) to[out=90,in=-90](2, 8.5);
\draw (.75, 8.5) rectangle (2.25, 9);
\node at (1.5, 8.75) {$\widehat{\mu}$};

\draw [knot] (-.5, 9) to[out=90,in=-90](-2, 10.5);
\draw [knot] (1.5, 9) to[out=90,in=-90](0, 10.5);
\draw [knot] (-2, 9) to[out=90,in=-90](1.5, 10.5);
\draw (1.75, 10.5) rectangle (-.25, 11);
\node at (.75, 10.75) {$\widehat{\mu}$};
\draw (.75, 11) --(.75, 11.4);
\draw (3.25, 11.4) rectangle (.5, 11.9);
\node at (1.875, 11.65) {$\widehat{\mu}$};
\draw (-2, 10.5) --(-2, 12.1);
\draw (1.875, 11.9) --(1.875, 12.1);
}\hspace{1 mm} \overset{\eqref{associativity}}{=} \hspace{1 mm}\tikzmath{
\draw (-1.5, 1.4) --(-1.5, 1.2);
\draw (1.5, 1.4) --(1.5, 4.5);
\draw (3, 1.4) --(3, 11.8);

\draw (-1.25, 1.4) rectangle (-1.75, 1.9);
\node at (-1.5, 1.65) {$S$};
\draw (-1.5, 1.9) --(-1.5, 2.5);
\draw (-2.25, 2.5) rectangle (-.75, 3);
\node at (-1.5, 2.75) {$\Delta$};
\draw (-2, 4.5) --(-2, 8.5);
\draw [knot] (-1, 3) to[out=90,in=-90](-2, 4.5);
\draw [knot] (-2, 3) to[out=90,in=-90](-1, 4.5);

\draw (-1.25, 4.5) rectangle (.25, 5);
\node at (-.5, 4.75) {$\Delta$};
\draw (2.25, 4.5) rectangle (.75, 5);
\node at (1.5, 4.75) {$\Delta$};
\draw (-1, 5) --(-1, 6.5);
\draw (2, 5) --(2, 6.5);
\draw [knot] (1, 5) to[out=90,in=-90](0, 6.5);
\draw [knot] (0, 5) to[out=90,in=-90](1, 6.5);
\draw (.25, 6.5) rectangle (-1.25, 7);
\node at (-.5, 6.75) {$\widehat{\mu}$};
\draw (-.5, 7) --(-.5, 8.5);
\draw (1.25, 6.5) rectangle (.75, 7);
\node at (1, 6.75) {$S$};
\draw (2.25, 6.5) rectangle (1.75, 7);
\node at (2, 6.75) {$S$};
\draw [knot] (2, 7) to[out=90,in=-90](1, 8.5);
\draw [knot] (1, 7) to[out=90,in=-90](2, 8.5);

\draw [knot] (-.5, 8.5) to[out=90,in=-90](-2, 10);
\draw [knot] (1, 8.5) to[out=90,in=-90](0, 10);
\draw [knot] (2, 8.5) to[out=90,in=-90](1, 10);
\draw [knot] (-2, 8.5) to[out=90,in=-90](2, 10);
\draw (-2, 12.5) --(-2, 10);
\draw (0, 10.9) --(0, 10);
\draw (.75, 10.5) rectangle (2.25, 10);
\node at (1.5, 10.25) {$\widehat{\mu}$};
\draw (1.5, 10.5) --(1.5, 10.9);
\draw (-.25, 10.9) rectangle (2.25, 11.4);
\node at (1, 11.15) {$\widehat{\mu}$};
\draw (1, 11.4) --(1, 11.8);
\draw (3.25, 11.8) rectangle (-.25, 12.3);
\node at (1.5, 12.05) {$\widehat{\mu}$};
\draw (1.5, 12.3) --(1.5, 12.5);
}\hspace{1 mm} \overset{(\ref{coassociativity})}{=} \hspace{1 mm}\tikzmath{
\draw (-1.5, .8) --(-1.5, 1);
\draw (1.5, .8) --(1.5, 4.5);
\draw (3, .8) --(3, 11.8);

\draw (-1.25, 1) rectangle (-1.75, 1.5);
\node at (-1.5, 1.25) {$S$};
\draw (-1.5, 1.5) --(-1.5, 2);

\draw (-2.25, 2.5) rectangle (-.75, 2);
\node at (-1.5, 2.25) {$\Delta$};
\draw (-2, 2.5) --(-2, 3.5);
\draw (-1, 2.5) --(-1, 3);
\draw (-1.25, 3.5) rectangle (.25, 3);
\node at (-.5, 3.25) {$\Delta$};
\draw [knot] (0, 3.5) to[out=90,in=-90](-2, 5);
\draw [knot] (-1, 3.5) to[out=90,in=-90](0, 5);
\draw [knot] (-2, 3.5) to[out=90,in=-90](-1, 5);
\draw (-2, 5) --(-2, 8.5);
\draw (2.25, 4.5) rectangle (.75, 5);
\node at (1.5, 4.75) {$\Delta$};
\draw (-1, 5) --(-1, 6.5);
\draw (2, 5) --(2, 6.5);
\draw [knot] (1, 5) to[out=90,in=-90](0, 6.5);
\draw [knot] (0, 5) to[out=90,in=-90](1, 6.5);
\draw (.25, 6.5) rectangle (-1.25, 7);
\node at (-.5, 6.75) {$\widehat{\mu}$};
\draw (-.5, 7) --(-.5, 8.5);
\draw (1.25, 6.5) rectangle (.75, 7);
\node at (1, 6.75) {$S$};
\draw (2.25, 6.5) rectangle (1.75, 7);
\node at (2, 6.75) {$S$};
\draw [knot] (2, 7) to[out=90,in=-90](1, 8.5);
\draw [knot] (1, 7) to[out=90,in=-90](2, 8.5);

\draw [knot] (-.5, 8.5) to[out=90,in=-90](-2, 10);
\draw [knot] (1, 8.5) to[out=90,in=-90](0, 10);
\draw [knot] (2, 8.5) to[out=90,in=-90](1, 10);
\draw [knot] (-2, 8.5) to[out=90,in=-90](2, 10);
\draw (-2, 12.5) --(-2, 10);
\draw (0, 10.9) --(0, 10);
\draw (.75, 10.5) rectangle (2.25, 10);
\node at (1.5, 10.25) {$\widehat{\mu}$};
\draw (1.5, 10.5) --(1.5, 10.9);
\draw (-.25, 10.9) rectangle (2.25, 11.4);
\node at (1, 11.15) {$\widehat{\mu}$};
\draw (1, 11.4) --(1, 11.8);
\draw (3.25, 11.8) rectangle (-.25, 12.3);
\node at (1.5, 12.05) {$\widehat{\mu}$};
\draw (1.5, 12.3) --(1.5, 12.5);
}$$
$\hspace{2 mm} \overset{\eqref{S:antipodeaxiom}}{=} \hspace{1 mm}\tikzmath{
\draw (-1.5, 0) --(-1.5, 2.3);
\draw (1.5, 0) --(1.5, 3.6);
\draw (0, 0) --(0, 1);

\draw (-1.25, 2.3) rectangle (-1.75, 2.8);
\node at (-1.5, 2.55) {$S$};
\draw (-1.5, 2.8) --(-1.5, 3.6);
\draw (-.75, 1.5) rectangle (.75, 1);
\node at (0, 1.25) {$\Delta$};
\draw (-.5, 1.5) --(-.5, 3.6);
\draw (.5, 2.3) --(.5, 1.5);

\draw (-1.75, 3.6) rectangle (-.25, 4.1);
\node at (-1, 3.85) {$\widehat{\mu}$};
\draw (.25, 2.3) rectangle (.75, 2.8);
\node at (.5, 2.55) {$S$};
\draw (-1, 4.1) --(-1, 4.5);
\draw (.5, 2.8) --(.5, 3.6);
\draw (1.75, 3.6) rectangle (.25, 4.1);
\node at (1, 3.85) {$\widehat{\mu}$};
\draw (1, 4.1) --(1, 4.5);
\node at (2, 0) {.};
}$
\end{proof}

For the tensor coalgebra, denote by
\[
c\colon T^c(sA)\otimes T^c(sA)\longrightarrow T^c(sA)
\]
the concatenation map, namely
\[
c(sa_{1,p}\bigotimes sb_{1,r})=sa_{1,p}\otimes sb_{1,r}.
\]
Thus $c(1_\Bbbk\bigotimes sb_{1,r})=sb_{1,r}=c(sb_{1,r}\bigotimes 1_\Bbbk)$. We also write
\[
c^{(2)}=c\circ(c\otimes\id)=c\circ(\id\otimes c).
\]

The following identities express elementary compatibilities between this concatenation map, the differential $\widehat m$, the Hopf multiplication $\widehat\mu$, and the antipode. Again, they are formal for any cofree dg Hopf algebra and will be used in the proof of Lemma \ref{lemm:Hwithunit}.
\begin{prop}\label{prop:concatenationidentity}
Let $(A,m^n;\mu^{p,q})$ be a $B_\infty$-algebra, and let $S$ be the antipode of the associated cofree dg Hopf algebra $(T^c(sA),\Delta,\widehat m,\widehat\mu)$. Then the following identities hold:
\begin{equation}\label{equ:mwithc}
\tikzmath{
\draw (0, -.5) --(0, 1.2);
\draw (1, -.5) --(1, 1.2);
\node at (., -1) {$sA$};
\node at (1, -1) {$T^c(sA)$};
\draw (1.25, 1.2) rectangle (-.25, 1.7);
\node at (.5, 1.45) {$c$};
\draw (.5, 1.7) -- (.5, 2.5);
\draw (1.25, 2.5) rectangle (-.25, 3);
\node at (.5, 2.75) {$\widehat{m}$};
}\hspace{1 mm} = \hspace{1 mm} \tikzmath{
\draw (0, -.5) --(0, 2.5);
\draw (1, -.5) --(1, 1.2);
\node at (., -1) {$sA$};
\node at (1, -1) {$T^c(sA)$};
\draw (1.25, 1.2) rectangle (.75, 1.7);
\node at (1, 1.45) {$\widehat{m}$};
\draw (1, 1.7) -- (1, 2.5);
\draw (1.25, 2.5) rectangle (-.25, 3);
\node at (.5, 2.75) {$c$};
}\hspace{1 mm} + \hspace{1 mm}
\tikzmath{
\draw (-2, 0) --(-2, 2.5);
\draw (-.5, 0) --(-.5, 1.2);
\node at (-2, -.5) {$sA$};
\node at (-.5, -.5) {$T^c(sA)$};
\draw (-1.25, 1.2) rectangle (.25, 1.7);
\node at (-.5, 1.45) {$\Delta$};
\draw (-1, 1.7) --(-1, 2.5);
\draw (0, 1.7) --(0, 3.8);
\draw (-2.25, 2.5) rectangle (-.75, 3);
\node at (-1.5, 2.75) {$m\circ c$};
\draw (-1.5, 3) -- (-1.5, 3.8);
\draw (-2.25, 3.8) rectangle (.25, 4.3);
\node at (-1, 4.05) {$c$};
\node at (.5, 0) {,};
}
\end{equation}
\begin{equation}\label{equ:muwithc}
\tikzmath{
\draw (-1, -.5) --(-1, 2.5);
\draw (0, -.5) --(0, 1.2);
\draw (1, -.5) --(1, 1.2);
\node at (-1, -1) {$T^c(sA)$};
\node at (., -1) {$sA$};
\node at (1, -1) {$T^c(sA)$};

\draw (1.25, 1.2) rectangle (-.25, 1.7);
\node at (.5, 1.45) {$c$};
\draw (.5, 1.7) -- (.5, 2.5);
\draw (-1.25, 2.5) rectangle (1.25, 3);
\node at (0, 2.75) {$\widehat{\mu}$};
}
\hspace{1 mm} = \hspace{1 mm}
\tikzmath{
\draw (-2, 0) --(-2, 1.2);
\draw (0, 0) --(0, 2.5);
\draw (1.5, 0) --(1.5, 1.2);
\node at (-2, -.5) {$T^c(sA)$};
\node at (-.2, -.5) {$sA$};
\node at (1, -.5) {$T^c(sA)$};

\draw (-3.25, 1.2) rectangle (-.75, 1.7);
\node at (-2, 1.45) {$\Delta^{(2)}$};
\draw (2.25, 1.2) rectangle (.75, 1.7);
\node at (1.5, 1.45) {$\Delta$};
\draw (2, 1.7) -- (2, 4.5);
\draw (1, 1.7) -- (1, 2.5);
\draw (-1, 1.7) -- (-1, 3);
\draw (-2, 1.7) -- (-2, 4.5);
\draw (-3, 1.7) -- (-3, 5.8);
\draw (-.25, 2.5) rectangle (1.25, 3);
\node at (.5, 2.75) {$c$};
\draw [knot] (1, 3) to[out=90,in=-90](-1, 4.5);
\draw [knot] (-1, 3) to[out=90,in=-90](1, 4.5);
\draw (.75, 4.5) rectangle (2.25, 5);
\node at (1.5, 4.75) {$\widehat{\mu}$};
\draw (-.75, 4.5) rectangle (-2.25, 5);
\node at (-1.5, 4.75) {$\mu$};
\draw (-1.5, 5) -- (-1.5, 5.8);
\draw (1.5, 5) -- (1.5, 5.8);
\draw (-3.25, 5.8) rectangle (2.25, 6.3);
\node at (-.5, 6.05) {$c^{(2)}$};
\node at (2.5, 0) {,};
}
\end{equation}
and
\begin{equation}\label{equ:mpq}
\tikzmath{
\draw (-1.5, 0) --(-1.5, 1.2);
\draw (0, 0) --(0, 1.7);
\draw (1, 0) --(1, 4.5);

\node at (-1.5, -.5) {$(sA)^{\otimes p}$};
\node at (-.2, -.5) {$sA$};
\node at (1, -.5) {$(sA)^{\otimes m}$};

\draw (-2.25, 1.2) rectangle (-.75, 1.7);
\node at (-1.5, 1.45) {$\Delta$};
\draw (-2, 5.6) -- (-2, 1.7);
\draw [knot] (0, 1.7) to[out=90,in=-90](-1, 3.2);
\draw [knot] (-1, 1.7) to[out=90,in=-90](0, 3.2);
\draw (-1, 5.6) -- (-1, 3.2);
\draw (0, 4.5) -- (0, 3.7);
\draw (-.25, 3.2) rectangle (.25, 3.7);
\node at (0, 3.45) {$S$};
\draw (1.25, 4.5) rectangle (-.25, 5);
\node at (.5, 4.75) {$\widehat{\mu}$};

\draw (-1.25, 5.6) rectangle (.75, 6.1);
\node at (-.25, 5.85) {$c$};
\draw (.5, 5.6) -- (.5, 5);
\draw (-.25, 6.1) -- (-.25, 6.7);
\draw (-2, 5.6) -- (-2, 6.7);
\draw (-2.25, 6.7) rectangle (0, 7.2);
\node at (-1.125, 6.95) {$\widehat{\mu}$};
}\hspace{1 mm} = \hspace{1 mm} \tikzmath{
\draw (-3, 1.7) --(-3, 7.7);
\draw (-2, .8) --(-2, 1.2);
\draw (0, .8) --(0, 1.7);
\draw (1.5, .8) --(1.5, 1.2);

\node at (-2, .3) {$(sA)^{\otimes p}$};
\node at (-.2, .3) {$sA$};
\node at (1.5, .3) {$(sA)^{\otimes m}$};

\draw (-3.25, 1.2) rectangle (-.75, 1.7);
\node at (-2, 1.45) {$\Delta^{(2)}$};
\draw (2.25, 1.2) rectangle (.75, 1.7);
\node at (1.5, 1.45) {$\Delta$};
\draw (1, 4.5) -- (1, 1.7);
\draw (-2, 5.6) -- (-2, 1.7);
\draw [knot] (0, 1.7) to[out=90,in=-90](-1, 3.2);
\draw [knot] (-1, 1.7) to[out=90,in=-90](0, 3.2);
\draw (-1, 5.6) -- (-1, 3.2);
\draw (0, 4.5) -- (0, 3.7);
\draw (-.25, 3.2) rectangle (.25, 3.7);
\node at (0, 3.45) {$S$};
\draw (1.25, 4.5) rectangle (-.25, 5);
\node at (.5, 4.75) {$\widehat{\mu}$};

\draw (-1.25, 5.6) rectangle (.75, 6.1);
\node at (-.25, 5.85) {$c$};
\draw (.5, 5.6) -- (.5, 5);
\draw (-.25, 6.1) -- (-.25, 6.7);
\draw (-2, 5.6) -- (-2, 6.7);
\draw (-2.25, 6.7) rectangle (0, 7.2);
\node at (-1.125, 6.95) {$\mu$};
\draw (2, 1.7) -- (2, 7.7);
\draw (-1.125, 7.2) -- (-1.125, 7.7);
\draw (-3.25, 7.7) rectangle (2.25, 8.2);
\node at (-.5, 7.95) {$c^{(2)}$};
\node at (2.5, 0.8) {.};
}
\end{equation}
\end{prop}
\begin{proof}
Formula \eqref{equ:mwithc} follows from the coderivation formula \eqref{align:widehatm}:
 $$\widehat m (sa \otimes sb_{1, r}) = sa \otimes \widehat m (sb_{1,r}) + \sum_{k=0}^r m^{k+1} (sa \otimes sb_{1, k}) \otimes sb_{k+1, r}.$$
Formula \eqref{equ:muwithc} follows from the graphical presentation of the recursive formula \eqref{align:widehatmurecursive}.

It remains to prove \eqref{equ:mpq}. Using only the bialgebra identity, the anti-coalgebra and anti-algebra properties of the antipode, associativity, coassociativity, and the antipode axiom, we obtain
$$
\tikzmath{
\draw (-1.5, 0) --(-1.5, 1.2);
\draw (0, 0) --(0, 1.7);
\draw (1, 0) --(1, 4.5);

\node at (-1.5, -.5) {$(sA)^{\otimes p}$};
\node at (-.2, -.5) {$sA$};
\node at (1, -.5) {$(sA)^{\otimes m}$};

\draw (-2.25, 1.2) rectangle (-.75, 1.7);
\node at (-1.5, 1.45) {$\Delta$};
\draw (-2, 5.6) -- (-2, 1.7);
\draw [knot] (0, 1.7) to[out=90,in=-90](-1, 3.2);
\draw [knot] (-1, 1.7) to[out=90,in=-90](0, 3.2);
\draw (-1, 5.6) -- (-1, 3.2);
\draw (0, 4.5) -- (0, 3.7);
\draw (-.25, 3.2) rectangle (.25, 3.7);
\node at (0, 3.45) {$S$};
\draw (1.25, 4.5) rectangle (-.25, 5);
\node at (.5, 4.75) {$\widehat{\mu}$};

\draw (-1.25, 5.6) rectangle (.75, 6.1);
\node at (-.25, 5.85) {$c $};
\draw (.5, 5.6) -- (.5, 5);
\draw (-.25, 6.1) -- (-.25, 6.7);
\draw (-2, 5.6) -- (-2, 6.7);
\draw (-2.25, 6.7) rectangle (0, 7.2);
\node at (-1.125, 6.95) {$\widehat{\mu}$};
}\hspace{-1 mm} \overset{(\ref{equ:muwithc})}{=} \hspace{1 mm} \tikzmath{
\draw (-1.75, 0) --(-1.75, 1.2);
\draw (0, 0) --(0, 1.7);
\draw (1, 0) --(1, 4.5);


\draw (-2.75, 1.2) rectangle (-.75, 1.7);
\node at (-1.75, 1.45) {$\Delta$};
\draw (-2.5, 5.6) -- (-2.5, 1.7);
\draw [knot] (0, 1.7) to[out=90,in=-90](-1, 3.2);
\draw [knot] (-1, 1.7) to[out=90,in=-90](0, 3.2);
\draw (-1, 6.9) -- (-1, 3.2);
\draw (0, 4.5) -- (0, 3.7);
\draw (-.25, 3.2) rectangle (.25, 3.7);
\node at (0, 3.45) {$S$};
\draw (1.25, 4.5) rectangle (-.25, 5);
\node at (.5, 4.75) {$\widehat{\mu}$};
\draw (.5, 5)-- (.5, 5.6);

\draw (-4.25, 5.6) rectangle (-1.75, 6.1);
\node at (-3, 5.85) {$\Delta^{(2)}$};
\draw (1.25, 5.6) rectangle (-.25, 6.1);
\node at (.5, 5.85) {$\Delta$};
\draw (1, 6.1) -- (1, 8.9);
\draw (0, 6.1) -- (0, 6.9);
\draw (-2, 6.1) -- (-2, 7.4);
\draw (-3, 6.1) -- (-3, 8.9);
\draw (-4, 6.1) -- (-4, 10);
\draw (-1.25, 6.9) rectangle (.25, 7.4);
\node at (-.5, 7.15) {$c$};
\draw [knot] (0, 7.4) to[out=90,in=-90](-2, 8.9);
\draw [knot] (-2, 7.4) to[out=90,in=-90](0, 8.9);
\draw (-.25, 8.9) rectangle (1.25, 9.4);
\node at (.5, 9.15) {$\widehat{\mu}$};
\draw (-1.75, 8.9) rectangle (-3.25, 9.4);
\node at (-2.5, 9.15) {$\mu$};
\draw (-2.5, 9.4) -- (-2.5, 10);
\draw (.5, 9.4) -- (.5, 10);
\draw (-4.25, 10) rectangle (1.25, 10.5);
\node at (-1.5, 10.25) {$c^{(2)}$};
}\hspace{2 mm} \overset{\eqref{Delta:algebramorphism}}{=}\hspace{2 mm}\tikzmath{
\draw (-1.75, 0) --(-1.75, .6);
\draw (0, 0) --(0, 1.1);
\draw (2.5, 0) --(2.5, 3.9);

\draw (-2.75, .6) rectangle (-.75, 1.1);
\node at (-1.75, .85) {$\Delta$};
\draw (-2.5, 5.6) -- (-2.5, 1.1);
\draw [knot] (0, 1.1) to[out=90,in=-90](-1, 2.6);
\draw [knot] (-1, 1.1) to[out=90,in=-90](0, 2.6);
\draw (-1, 6.9) -- (-1, 2.6);
\draw (0, 3.9) -- (0, 3.1);
\draw (-.25, 2.6) rectangle (.25, 3.1);
\node at (0, 2.85) {$S$};
\draw (1.25, 3.9) rectangle (-.25, 4.4);
\node at (.5, 4.15) {$\Delta$};
\draw (1.75, 3.9) rectangle (3.25, 4.4);
\node at (2.5, 4.15) {$\Delta$};
\draw [knot] (2, 4.4) to[out=90,in=-90](1, 5.6);
\draw [knot] (1, 4.4) to[out=90,in=-90](2, 5.6);
\draw (1.25, 5.6) rectangle (-.25, 6.1);
\node at (.5, 5.85) {$\widehat{\mu}$};
\draw (1.75, 5.6) rectangle (3.25, 6.1);
\node at (2.5, 5.85) {$\widehat{\mu}$};
\draw (0, 5.6) -- (0, 4.4);
\draw (3, 5.6) -- (3, 4.4);

\draw (-4.25, 5.6) rectangle (-1.75, 6.1);
\node at (-3, 5.85) {$\Delta^{(2)}$};
\draw (2, 6.1) -- (2, 8.9);
\draw (0, 6.1) -- (0, 6.9);
\draw (-2, 6.1) -- (-2, 7.4);
\draw (-3, 6.1) -- (-3, 8.9);
\draw (-4, 6.1) -- (-4, 10);
\draw (-1.25, 6.9) rectangle (.25, 7.4);
\node at (-.5, 7.15) {$c$};
\draw [knot] (0, 7.4) to[out=90,in=-90](-2, 8.9);
\draw [knot] (-2, 7.4) to[out=90,in=-90](0, 8.9);
\draw (-.25, 8.9) rectangle (2.25, 9.4);
\node at (1, 9.15) {$\widehat{\mu}$};
\draw (-1.75, 8.9) rectangle (-3.25, 9.4);
\node at (-2.5, 9.15) {$\mu$};
\draw (-2.5, 9.4) -- (-2.5, 10);
\draw (1, 9.4) -- (1, 10);
\draw (-4.25, 10) rectangle (2.25, 10.5);
\node at (-1, 10.25) {$c^{(2)}$};
}
$$
$$\hspace{2 mm}\overset{\eqref{S:anti-coalg}}{=}\hspace{2 mm}\tikzmath{
 \draw (-1.75, -.4) --(-1.75, .2);
\draw (0, -.4) --(0, .7);
\draw (2.5, -.4) --(2.5, 3.9);

\draw (-2.75, .2) rectangle (-.75, .7);
\node at (-1.75, .45) {$\Delta$};
\draw (-2.5, 5.6) -- (-2.5, .7);
\draw [knot] (0, .7) to[out=90,in=-90](-1, 2.2);
\draw [knot] (-1, .7) to[out=90,in=-90](0, 2.2);
\draw (-1, 6.9) -- (-1, 2.2);
\draw (-.25, 2.2) rectangle (1.25, 2.7);
\node at (.5, 2.45) {$\Delta$};
\draw [knot] (1, 2.7) to[out=90,in=-90](0, 3.9);
\draw [knot] (0, 2.7) to[out=90,in=-90](1, 3.9);
\draw (1.25, 3.9) rectangle (.75, 4.4);
\node at (1, 4.15) {$S$};
\draw (.25, 3.9) rectangle (-.25, 4.4);
\node at (0, 4.15) {$S$};
\draw (1.75, 3.9) rectangle (3.25, 4.4);
\node at (2.5, 4.15) {$\Delta$};
\draw [knot] (2, 4.4) to[out=90,in=-90](1, 5.6);
\draw [knot] (1, 4.4) to[out=90,in=-90](2, 5.6);
\draw (1.25, 5.6) rectangle (-.25, 6.1);
\node at (.5, 5.85) {$\widehat{\mu}$};
\draw (1.75, 5.6) rectangle (3.25, 6.1);
\node at (2.5, 5.85) {$\widehat{\mu}$};
\draw (0, 5.6) -- (0, 4.4);
\draw (3, 5.6) -- (3, 4.4);

\draw (-4.25, 5.6) rectangle (-1.75, 6.1);
\node at (-3, 5.85) {$\Delta^{(2)}$};
\draw (2, 6.1) -- (2, 8.9);
\draw (0, 6.1) -- (0, 6.9);
\draw (-2, 6.1) -- (-2, 7.4);
\draw (-3, 6.1) -- (-3, 8.9);
\draw (-4, 6.1) -- (-4, 10);
\draw (-1.25, 6.9) rectangle (.25, 7.4);
\node at (-.5, 7.15) {$c$};
\draw [knot] (0, 7.4) to[out=90,in=-90](-2, 8.9);
\draw [knot] (-2, 7.4) to[out=90,in=-90](0, 8.9);
\draw (-.25, 8.9) rectangle (2.25, 9.4);
\node at (1, 9.15) {$\widehat{\mu}$};
\draw (-1.75, 8.9) rectangle (-3.25, 9.4);
\node at (-2.5, 9.15) {$\mu$};
\draw (-2.5, 9.4) -- (-2.5, 10);
\draw (1, 9.4) -- (1, 10);
\draw (-4.25, 10) rectangle (2.25, 10.5);
\node at (-1, 10.25) {$c^{(2)}$};
}\hspace{2mm} \overset{\eqref{associativity}}{=} \hspace{2mm}\tikzmath{
\draw (-1.75, -.4) --(-1.75, .2);
\draw (0, -.4) --(0, .7);
\draw (2.5, -.4) --(2.5, 3.9);

\draw (-2.75, .2) rectangle (-.75, .7);
\node at (-1.75, .45) {$\Delta$};
\draw (-2.5, 5.6) -- (-2.5, .7);
\draw [knot] (0, .7) to[out=90,in=-90](-1, 2.2);
\draw [knot] (-1, .7) to[out=90,in=-90](0, 2.2);
\draw (-1, 6.7) -- (-1, 2.2);
\draw (-.25, 2.2) rectangle (1.25, 2.7);
\node at (.5, 2.45) {$\Delta$};
\draw [knot] (1, 2.7) to[out=90,in=-90](0, 3.9);
\draw [knot] (0, 2.7) to[out=90,in=-90](1, 3.9);
\draw (1.25, 3.9) rectangle (.75, 4.4);
\node at (1, 4.15) {$S$};
\draw (.25, 3.9) rectangle (-.25, 4.4);
\node at (0, 4.15) {$S$};
\draw (1.75, 3.9) rectangle (3.25, 4.4);
\node at (2.5, 4.15) {$\Delta$};
\draw [knot] (2, 4.4) to[out=90,in=-90](1, 5.6);
\draw [knot] (1, 4.4) to[out=90,in=-90](2, 5.6);
\draw (1.25, 5.6) rectangle (-.25, 6.1);
\node at (.5, 5.85) {$\widehat{\mu}$};
\draw (0, 5.6) -- (0, 4.4);
\draw (3, 8.9) -- (3, 4.4);

\draw (-4.25, 5.6) rectangle (-1.75, 6.1);
\node at (-3, 5.85) {$\Delta^{(2)}$};
\draw (0, 6.1) -- (0, 6.7);
\draw (-2, 6.1) -- (-2, 7.2);
\draw (-3, 6.1) -- (-3, 8.9);
\draw (-4, 6.1) -- (-4, 10);
\draw (-1.25, 6.7) rectangle (.25, 7.2);
\node at (-.5, 6.95) {$c$};
\draw [knot] (-1, 7.2) to[out=90,in=-90](-2, 8.9);
\draw [knot] (-2, 7.2) to[out=90,in=-90](.5, 7.9);
\draw (.25, 7.9) rectangle (2.25, 8.4);
\node at (1.25, 8.15) {$\widehat{\mu}$};
\draw (2, 8.4) -- (2, 8.9);
\draw (2, 5.6) -- (2, 7.9);
\draw (1.75, 8.9) rectangle (3.25, 9.4);
\node at (2.5, 9.15) {$\widehat{\mu}$};
\draw (-1.75, 8.9) rectangle (-3.25, 9.4);
\node at (-2.5, 9.15) {$\mu$};
\draw (-2.5, 9.4) -- (-2.5, 10);
\draw (2.5, 9.4) -- (2.5, 10);
\draw (-4.25, 10) rectangle (3.25, 10.5);
\node at (-.5, 10.25) {$c^{(2)}$};
}
$$
$\hspace{8 mm} \overset{(\ref{coassociativity})}{=} \hspace{1 mm} \tikzmath{
\draw (2.5, -2) --(2.5, 3.9);
\draw (1, 1.2) --(1, -2);
\draw (-2.5, -1.7) --(-2.5, -2);

\draw (-4.25, -1.2) rectangle (-.75, -1.7);
\node at (-2.5, -1.45) {$\Delta^{(2)}$};
\draw (-3, -1.2) -- (-3, 8.9);
\draw (-4, -1.2) -- (-4, 10);
\draw (-1, -1.2) --(-1, -.5);

\draw (-1.25, 0) rectangle (.25, -.5);
\node at (-.5, -.25) {$\Delta$};
\draw (-1, 0) --(-1, .7);
\draw (0, 0) --(0, 1.2);
\draw (-2.25, .7) rectangle (-.75, 1.2);
\node at (-1.5, .95) {$\Delta$};
\draw (-2, 1.2) --(-2, 7.2);
\draw [knot] (1, 1.2) to[out=90,in=-90](-1, 2.7);
\draw [knot] (-1, 1.2) to[out=90,in=-90](0, 2.7);
\draw [knot] (0, 1.2) to[out=90,in=-90](1, 2.7);
\draw (-1, 2.7) --(-1, 6.7);

\draw [knot] (1, 2.7) to[out=90,in=-90](0, 3.9);
\draw [knot] (0, 2.7) to[out=90,in=-90](1, 3.9);
\draw (1.25, 3.9) rectangle (.75, 4.4);
\node at (1, 4.15) {$S$};
\draw (.25, 3.9) rectangle (-.25, 4.4);
\node at (0, 4.15) {$S$};
\draw (1.75, 3.9) rectangle (3.25, 4.4);
\node at (2.5, 4.15) {$\Delta$};
\draw [knot] (2, 4.4) to[out=90,in=-90](1, 5.6);
\draw [knot] (1, 4.4) to[out=90,in=-90](2, 5.6);
\draw (1.25, 5.6) rectangle (-.25, 6.1);
\node at (.5, 5.85) {$\widehat{\mu}$};
\draw (0, 5.6) -- (0, 4.4);
\draw (3, 8.9) -- (3, 4.4);

\draw (0, 6.1) -- (0, 6.7);
\draw (-2, 6.1) -- (-2, 7.2);

\draw (-1.25, 6.7) rectangle (.25, 7.2);
\node at (-.5, 6.95) {$c$};
\draw [knot] (-1, 7.2) to[out=90,in=-90](-2, 8.9);
\draw [knot] (-2, 7.2) to[out=90,in=-90](.5, 7.9);
\draw (.25, 7.9) rectangle (2.25, 8.4);
\node at (1.25, 8.15) {$\widehat{\mu}$};
\draw (2, 8.4) -- (2, 8.9);
\draw (2, 5.6) -- (2, 7.9);
\draw (1.75, 8.9) rectangle (3.25, 9.4);
\node at (2.5, 9.15) {$\widehat{\mu}$};
\draw (-1.75, 8.9) rectangle (-3.25, 9.4);
\node at (-2.5, 9.15) {$\mu$};
\draw (-2.5, 9.4) -- (-2.5, 10);
\draw (2.5, 9.4) -- (2.5, 10);
\draw (-4.25, 10) rectangle (3.25, 10.5);
\node at (-.5, 10.25) {$c^{(2)}$};
}
\hspace{2 mm} \overset{(\ref{S:antipodeaxiom})}{=} \hspace{2 mm} \tikzmath{
\draw (-3, 1.7) --(-3, 7.7);
\draw (-2, .8) --(-2, 1.2);
\draw (0, .8) --(0, 1.7);
\draw (1.5, .8) --(1.5, 1.2);

\node at (-2, .3) {$(sA)^{\otimes p}$};
\node at (-.2, .3) {$sA$};
\node at (1.5, .3) {$(sA)^{\otimes m}$};

\draw (-3.25, 1.2) rectangle (-.75, 1.7);
\node at (-2, 1.45) {$\Delta^{(2)}$};
\draw (2.25, 1.2) rectangle (.75, 1.7);
\node at (1.5, 1.45) {$\Delta$};
\draw (1, 4.5) -- (1, 1.7);
\draw (-2, 5.6) -- (-2, 1.7);
\draw [knot] (0, 1.7) to[out=90,in=-90](-1, 3.2);
\draw [knot] (-1, 1.7) to[out=90,in=-90](0, 3.2);
\draw (-1, 5.6) -- (-1, 3.2);
\draw (0, 4.5) -- (0, 3.7);
\draw (-.25, 3.2) rectangle (.25, 3.7);
\node at (0, 3.45) {$S$};
\draw (1.25, 4.5) rectangle (-.25, 5);
\node at (.5, 4.75) {$\widehat{\mu}$};

\draw (-1.25, 5.6) rectangle (.75, 6.1);
\node at (-.25, 5.85) {$c$};
\draw (.5, 5.6) -- (.5, 5);
\draw (-.25, 6.1) -- (-.25, 6.7);
\draw (-2, 5.6) -- (-2, 6.7);
\draw (-2.25, 6.7) rectangle (0, 7.2);
\node at (-1.125, 6.95) {$\mu$};
\draw (2, 1.7) -- (2, 7.7);
\draw (-1.125, 7.2) -- (-1.125, 7.7);
\draw (-3.25, 7.7) rectangle (2.25, 8.2);
\node at (-.5, 7.95) {$c^{(2)}$};
\node at (2.5, 0.8) {.};
}$
\end{proof}


We now return to the $B_\infty$-algebra $A$ and use the preceding Hopf-algebra identities to construct the induction functor. The next result answers the question raised in Remark \ref{rem:question}.

\begin{prop}\label{prop:dualitybinfinity}
Let $(A,m^n;\mu^{p,q})$ be a $B_\infty$-algebra, and let $S$ be the antipode of $(T^c(sA),\Delta,\widehat m,\widehat\mu)$. Then there exists a homomorphism of dg coalgebras
\[
\varphi\colon T^c(sA)^{\op}\otimes T^c(sA)\longrightarrow T^c(sA)
\]
given by
$$\varphi = \tikzmath{
\draw (1, .5) --(1, 3);

\node at (-.2, 0) {$T^c(sA)$};
\node at (1.2, 0) {$T^c(sA)$};
\node at (1.5, .5) {.};
\draw (0, .5) -- (0,1.6);

\draw (-.25, 1.6) rectangle (.25, 2.1);
\node at (0, 1.85) {$S$};

\draw (0, 2.1) -- (0,3);

\draw (-.25, 3) rectangle (1.25, 3.5);
\node at (0.5, 3.25) {$\widehat{\mu}$};
}$$

Consequently, we obtain a dg functor $\iota\colon \EC^{\RR}_\infty(A)\to \EC^{\BB}_\infty(A)$, and hence a triangulated functor $\iota\colon \ED^{\RR}_\infty(A)\to \ED^{\BB}_\infty(A)$.
\end{prop}

\begin{proof}
First we show that $\varphi$ is compatible with the differentials:
$$\tikzmath{
\draw (1, .6) --(1, 3);
\draw (0, .6) -- (0,1.8);

\draw (-.25, 1.8) rectangle (.25, 2.3);
\node at (0, 2.05) {$S$};

\draw (0, 2.3) -- (0,3);

\draw (-.25, 3) rectangle (1.25, 3.5);
\node at (0.5, 3.25) {$\widehat{\mu}$};

\draw (.5, 3.5) -- (.5, 4.1);

\draw (-.25, 4.1) rectangle (1.25, 4.6);
\node at (0.5, 4.35) {$\widehat{m}$};
}\hspace{1 mm} = \hspace{1 mm}
\tikzmath{
\draw (1, .0) --(1, 3.5);
\draw (0, .) -- (0, .9);

\draw (-.25, 0.9) rectangle (.25, 1.4);
\node at (0, 1.15) {$S$};

\draw (0, 1.4) -- (0, 2.2);
\draw (-.25, 2.2) rectangle (.25, 2.7);
\node at (0, 2.45) {$\widehat{m}$};

\draw (0, 2.7) -- (0, 3.5);

\draw (-.25, 3.5) rectangle (1.25, 4);
\node at (0.5, 3.75) {$\widehat{\mu}$};
}\hspace{1 mm} + \hspace{1 mm}
\tikzmath{
\draw (1, .0) --(1, 1.7);
\draw (0, .) -- (0, 1.7);

\draw (-.25, 1.7) rectangle (.25, 2.2);
\node at (0, 1.95) {$S$};
\draw (.75, 1.7) rectangle (1.25, 2.2);
\node at (1, 1.95) {$\widehat{m}$};

\draw (0, 2.2) -- (0, 3.5);
\draw (1, 2.2) -- (1, 3.5);

\draw (-.25, 3.5) rectangle (1.25, 4);
\node at (0.5, 3.75) {$\widehat{\mu}$};
}\hspace{1 mm} = \hspace{1 mm}
\tikzmath{
\draw (1, .0) --(1, 3.5);
\draw (0, .) -- (0, .9);

\draw (-.25, 0.9) rectangle (.25, 1.4);
\node at (0, 1.15) {$\widehat{m}$};

\draw (0, 1.4) -- (0, 2.2);
\draw (-.25, 2.2) rectangle (.25, 2.7);
\node at (0, 2.45) {$S$};

\draw (0, 2.7) -- (0, 3.5);

\draw (-.25, 3.5) rectangle (1.25, 4);
\node at (0.5, 3.75) {$\widehat{\mu}$};
}\hspace{1 mm} + \hspace{1 mm}
\tikzmath{
\draw (1, .0) --(1, 1.7);
\draw (0, .) -- (0, 1.7);
\node at (1.5, 0) {,};

\draw (-.25, 1.7) rectangle (.25, 2.2);
\node at (0, 1.95) {$S$};
\draw (.75, 1.7) rectangle (1.25, 2.2);
\node at (1, 1.95) {$\widehat{m}$};

\draw (0, 2.2) -- (0, 3.5);
\draw (1, 2.2) -- (1, 3.5);

\draw (-.25, 3.5) rectangle (1.25, 4);
\node at (0.5, 3.75) {$\widehat{\mu}$};
}$$
Here the first equality follows from
$\widehat{m}\circ\widehat{\mu}=\widehat{\mu}\circ(\widehat{m}\otimes\id+\id\otimes\widehat{m})$, and the second follows from $S\circ\widehat{m}=\widehat{m}\circ S$.

Next we prove compatibility with the coproducts, namely
\[
\Delta\circ\varphi=(\varphi\otimes\varphi)\circ(\id\otimes\tau\otimes\id)(\Delta^{\mathrm{op}}\otimes\Delta).
\]
Indeed,
$$\tikzmath{
\draw (1, 1) --(1, 3);
\draw (0, 1) -- (0, 1.8);

\draw (-.25, 1.8) rectangle (.25, 2.3);
\node at (0, 2.05) {$S$};

\draw (0, 2.3) -- (0,3);

\draw (-.25, 3) rectangle (1.25, 3.5);
\node at (0.5, 3.25) {$\widehat{\mu}$};

\draw (.5, 3.5) -- (.5, 4.1);

\draw (-.25, 4.1) rectangle (1.25, 4.6);
\node at (0.5, 4.35) {$\Delta$};
}\hspace{1 mm} = \hspace{1 mm}
\tikzmath{
\draw (2, 0.4) --(2, 2.5);
\draw (0, 0.4) -- (0, 1.1);

\draw (-.25, 1.1) rectangle (.25, 1.6);
\node at (0, 1.35) {$S$};
\draw (0, 1.6) -- (0, 2.5);

\draw (-.75, 2.5) rectangle (.75, 3);
\node at (0, 2.75) {$\Delta$};
\draw (1.25, 2.5) rectangle (2.75, 3);
\node at (2, 2.75) {$\Delta$};

\draw (-.5, 3) -- (-.5, 4.5);
\draw [knot] (1.5, 3) to[out=90,in=-90](.5, 4.5);
\draw [knot] (.5, 3) to[out=90,in=-90](1.5, 4.5);
\draw (2.5, 3) -- (2.5, 4.5);

\draw (-.75, 4.5) rectangle (.75, 5);
\node at (0, 4.75) {$\widehat{\mu}$};
\draw (1.25, 4.5) rectangle (2.75, 5);
\node at (2, 4.75) {$\widehat{\mu}$};
}\hspace{1 mm} = \hspace{1 mm}
\tikzmath{
\draw (2, 0.4) --(2, 1.2);
\draw (0, 0.4) -- (0, 1.2);

\draw (-.75, 1.2) rectangle (.75, 1.7);
\node at (0, 1.45) {$\Delta^{\op}$};
\draw (1.25, 1.2) rectangle (2.75, 1.7);
\node at (2, 1.45) {$\Delta$};

\draw (.5, 1.7) -- (.5, 2.7);
\draw (-.5, 1.7) -- (-.5, 2.7);
\draw (1.5, 1.7) -- (1.5, 3.2);
\draw (2.5, 1.7) -- (2.5, 4.5);
\draw (-.75, 2.7) rectangle (-.25, 3.2);
\node at (-.5, 2.95) {$S$};
\draw (.75, 2.7) rectangle (.25, 3.2);
\node at (.5, 2.95) {$S$};
\draw [knot] (1.5, 3.2) to[out=90,in=-90](.5, 4.5);
\draw [knot] (.5, 3.2) to[out=90,in=-90](1.5, 4.5);
\draw (-.5, 3.2) -- (-.5, 4.5);

\draw (-.75, 4.5) rectangle (.75, 5);
\node at (0, 4.75) {$\widehat{\mu}$};
\draw (1.25, 4.5) rectangle (2.75, 5);
\node at (2, 4.75) {$\widehat{\mu}$};
}\hspace{1 mm} = \hspace{1 mm}
\tikzmath{
\draw (2, 0.5) --(2, 1.2);
\draw (0, 0.5) -- (0, 1.2);
\node at (3, 0.5) {,};

\draw (-.75, 1.2) rectangle (.75, 1.7);
\node at (0, 1.45) {$\Delta^{\op}$};
\draw (1.25, 1.2) rectangle (2.75, 1.7);
\node at (2, 1.45) {$\Delta$};

\draw (-.5, 1.7) -- (-.5, 3);
\draw (2.5, 1.7) -- (2.5, 4.5);
\draw (-.75, 3) rectangle (-.25, 3.5);
\node at (-.5, 3.25) {$S$};
\draw (1.75, 3.5) rectangle (1.25, 3);
\node at (1.5, 3.25) {$S$};
\draw [knot] (1.5, 1.7) to[out=90,in=-90](.5, 3);
\draw [knot] (.5, 1.7) to[out=90,in=-90](1.5, 3);
\draw (-.5, 3.5) -- (-.5, 4.5);
\draw (.5, 3) -- (.5, 4.5);
\draw (1.5, 3.5) -- (1.5, 4.5);

\draw (-.75, 4.5) rectangle (.75, 5);
\node at (0, 4.75) {$\widehat{\mu}$};
\draw (1.25, 4.5) rectangle (2.75, 5);
\node at (2, 4.75) {$\widehat{\mu}$};
}
$$
For the first equality we use
$\Delta\circ\widehat{\mu}=(\widehat{\mu}\otimes\widehat{\mu})\circ(\id\otimes\tau\otimes\id)\circ(\Delta\otimes\Delta)$; for the second, we use $\Delta\circ S=(S\otimes S)\circ\Delta^{\op}$.

The second assertion now follows from Lemma \ref{lem:bi-rightduality}; in other words, $\iota=e_\varphi$. Explicitly, a right $A_\infty$-module $(M,\rho_M^n)$ is sent to the $A_\infty$-bimodule
\begin{align}\label{align:iotabimodule}
\iota((M, \rho_M^n))=(M, \beta_M^{p,q})\,,
\end{align}
where the graded maps
$\beta_M^{p, q} \colon (sA)^{\otimes p} \otimes sM \otimes (sA)^{\otimes q} \to sM,\ p, q \geq 0$
are given by
\\
\[\tikzmath{
\draw (-.8, .8) --(-.8, 1);
\draw (0, .8) --(0, 1);
\draw (.8, .8) --(.8, 3);

\node at (2, .5) {,};
\node at (-1, .5) {$(sA)^{\otimes p}$};
\node at (0, .5) {$sM$};
\node at (1, .5) {$(sA)^{\otimes q}$};
\draw [knot] (0, 1) to[out=90,in=-90](-.8, 2);
\draw [knot] (-.8, 1) to[out=90,in=-90](0, 2);
\draw (-.25, 2) rectangle (.25, 2.5);
\node at (0, 2.25) {$S$};

\draw (-.8, 2) -- (-.8, 4);
\draw (0, 2.5) -- (0,3);
\draw (-.15, 3) rectangle (.95, 3.5);
\node at (0.4, 3.25) {$\widehat{\mu}$};
\draw (.8, 3.5) -- (.8, 4);
\draw (-.95, 4) rectangle (.95, 4.5);
\node at (0, 4.25) {$\rho$};
}\]
where we use the recursive formula
$S(sa_{1,n})=-sa_{1,n}-\sum\limits_{i=1}^{n-1}\widehat{\mu}(sa_{1,i}\bigotimes S(sa_{i+1,n}))$.
\end{proof}

\begin{rem}
If $(M,\rho_M^n)$ is a right $A_\infty$-module, then $\iota(M)$ also gives $M$ a left $A_\infty$-module structure. If $M$ is strictly unital, then so is $\iota(M)$. Indeed,
\[
\beta_M^{1,0}(s1_A\otimes sx)=(-1)^{|sx|}\rho^2(sx\otimes S(s1_A))=(-1)^{|x|}\rho^2(sx\otimes s1_A)=sx,
\]
where we use $S(sa)=-sa$ for $sa\in sA$ and $\rho^2(sx\otimes s1_A)=(-1)^{|x|}sx$ for any $x\in M$; 
see Remark \ref{rem:unitalright}.
\end{rem}

\begin{rem}\label{rem:iotafaithfunonfull}
Recall the forgetful functor $F^{\RR}\colon \ED^{\BB}_\infty(A)\to \ED^{\RR}_\infty(A)$ from Remark \ref{rem:forgetfulfunctor}. Since $F^{\RR}\circ\iota=\id$, the functor $\iota\colon \ED^{\RR}_\infty(A)\to \ED^{\BB}_\infty(A)$ is faithful. However, $\iota$ is not full in general; see Remark \ref{rem:iotanotful}. This lack of fullness creates a technical obstacle: although we define a tensor product on $\ED^{\RR}_\infty(A)$ by using the tensor product on $\ED^{\BB}_\infty(A)$ and the forgetful functor $F^{\RR}$, the monoidal axioms cannot be transferred formally along $\iota$. This is why the unit and associativity constraints are checked by explicit $A_\infty$-bimodule quasi-isomorphisms. For further discussion, see Remark \ref{rem:monoidalbrace}.
\end{rem}

\vspace{0.5cm}
The next step is to show that the induced bimodule $\iota(A)$ is quasi-isomorphic to the regular bimodule $A$;
this will later provide the unit constraint for the monoidal structure.
\begin{thm}\label{thm:iotabracemodule}
Let $(A, m^n; \mu^{p, q})$ be a $B_\infty$-algebra.
Then $\iota(A)$ is quasi-isomorphic to $A$ as $A_\infty$-bimodules.
\end{thm}

\begin{proof}
Define an $A_\infty$-bimodule morphism $I_A\colon\iota(A)\to A$ as follows. Its component
$I^{p,q}_A\colon (sA)^{\otimes p}\otimes s(\iota(A))\otimes (sA)^{\otimes q}\longrightarrow sA$
is represented by the following graph:
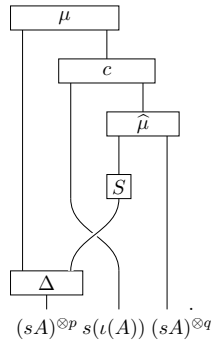
\vspace{0.7cm}
$$\tikzmath[scale=0.91]{
\draw (-1.5, .9) --(-1.5, 1.2);
\draw (0, .9) --(0, 1.7);
\draw (1, .9) --(1, 4.5);

\node at (-1.5, .5) {$(sA)^{\otimes p}$};
\node at (-.1, .5) {$s(\iota(A))$};
\node at (1.3, .5) {$(sA)^{\otimes q}$};

\draw (-2.25, 1.2) rectangle (-.75, 1.7);
\node at (-1.5, 1.45) {$\Delta$};
\draw (-2, 5.6) -- (-2, 1.7);
\draw [knot] (0, 1.7) to[out=90,in=-90](-1, 3.2);
\draw [knot] (-1, 1.7) to[out=90,in=-90](0, 3.2);
\draw (-1, 5.6) -- (-1, 3.2);
\draw (0, 4.5) -- (0, 3.7);
\draw (-.25, 3.2) rectangle (.25, 3.7);
\node at (0, 3.45) {$S$};
\draw (1.25, 4.5) rectangle (-.25, 5);
\node at (.5, 4.75) {$\widehat{\mu}$};

\draw (-1.25, 5.6) rectangle (.75, 6.1);
\node at (-.25, 5.85) {$c$};
\draw (.5, 5.6) -- (.5, 5);
\draw (-.25, 6.1) -- (-.25, 6.7);
\draw (-2, 5.6) -- (-2, 6.7);
\draw (-2.25, 6.7) rectangle (0, 7.2);
\node at (-1.125, 6.95) {$\mu$};
\node at (1.5, 0.9) {.};
\vspace{1cm}
}$$
In particular, $I_A^{0,0}=\id_A$. Hence it remains to verify that $I_A$ is a morphism of $A_\infty$-bimodules, namely 
$ I_A\circ\widehat{\beta}_{\iota(A)}=\beta_A\circ\widehat I_A$.

The term $I_A\circ\widehat{\beta}_{\iota(A)}$ on the left side is:
\vspace{0.7cm}
$$\tikzmath{
\draw (-1.5, -.4) --(-1.5, .9);
\draw (-0.25, -.4) --(-.25, 2.4);
\draw (1, -.4) --(1, 2.4);

\node at (-1.5, -.8) {$(sA)^{\otimes p}$};
\node at (-.1, -.8) {$s(\iota(A))$};
\node at (1.3, -.8) {$(sA)^{\otimes q}$};
\draw (-1.75, 1.4) rectangle (-1.25, .9);
\node at (-1.5, 1.15) {$\widehat{m}$};
\draw (-1.5, 1.4) -- (-1.5, 2.4);
\draw (-1.75, 2.4) rectangle (1.25, 2.9);
\node at (-.25, 2.65) {$I_A$};
}\hspace{-1 mm} + \hspace{2 mm}\tikzmath{
\draw (-1.5, -.4) --(-1.5, -.1);
\draw (0, -.4) --(0, 1.2);
\draw (1.5, -.4) --(1.5, -.1);

\node at (-1.5, -.8) {$(sA)^{\otimes p}$};
\node at (-.1, -.8) {$s(\iota(A))$};
\node at (1.3, -.8) {$(sA)^{\otimes q}$};

\draw (-2.25, -.1) rectangle (-.75, .4);
\node at (-1.5, .15) {$\Delta$};
\draw (2.25, -.1) rectangle (.75, .4);
\node at (1.5, .15) {$\Delta$};
\draw (-1, .4) --(-1, 1.2);
\draw (1, .4) --(1, 1.2);
\draw (-1.25, 1.2) rectangle (1.25, 1.7);
\node at (0, 1.45) {$\beta_{\iota(A)}$};
\draw (0, 1.7) -- (0, 2.4);
\draw (-2, .4) -- (-2, 2.4);
\draw (2, .4) -- (2, 2.4);
\draw (-2.25, 2.4) rectangle (2.25, 2.9);
\node at (0, 2.65) {$I_A$};
}\hspace{2 mm} + \hspace{-1 mm}\tikzmath{
\draw (1.5, -.4) --(1.5, .9);
\draw (.25, -.4) --(.25, 2.4);
\draw (-1, -.4) --(-1, 2.4);

\node at (1.6, -.8) {$(sA)^{\otimes p}$};
\node at (.1, -.8) {$s(\iota(A))$};
\node at (-1.3, -.8) {$(sA)^{\otimes q}$};
\draw (1.75, 1.4) rectangle (1.25, .9);
\node at (1.5, 1.15) {$\widehat{m}$};
\draw (1.5, 1.4) -- (1.5, 2.4);
\draw (1.75, 2.4) rectangle (-1.25, 2.9);
\node at (.25, 2.65) {$I_A$};
}
$$
\vspace{-1.7cm}
\begin{equation}\label{equ:mudelta}
= \hspace{3 mm}\tikzmath{
\draw (-1.5, -.4) --(-1.5, -.1);
\draw (0, -.4) --(0, 1.7);
\draw (1, -.4) --(1, 4.5);

\draw (-1.75, .4) rectangle (-1.25, -.1);
\node at (-1.5, .15) {$\widehat{m}$};
\draw (-1.5, .4) -- (-1.5, 1.2);
\draw (-2.25, 1.2) rectangle (-.75, 1.7);
\node at (-1.5, 1.45) {$\Delta$};
\draw (-2, 5.6) -- (-2, 1.7);
\draw [knot] (0, 1.7) to[out=90,in=-90](-1, 3.2);
\draw [knot] (-1, 1.7) to[out=90,in=-90](0, 3.2);
\draw (-1, 5.6) -- (-1, 3.2);
\draw (0, 4.5) -- (0, 3.7);
\draw (-.25, 3.2) rectangle (.25, 3.7);
\node at (0, 3.45) {$S$};
\draw (1.25, 4.5) rectangle (-.25, 5);
\node at (.5, 4.75) {$\widehat{\mu}$};

\draw (-1.25, 5.6) rectangle (.75, 6.1);
\node at (-.25, 5.85) {$c$};
\draw (.5, 5.6) -- (.5, 5);
\draw (-.25, 6.1) -- (-.25, 6.7);
\draw (-2, 5.6) -- (-2, 6.7);
\draw (-2.25, 6.7) rectangle (0, 7.2);
\node at (-1.125, 6.95) {$\mu$};
}\hspace{1 mm} + \hspace{2 mm}\tikzmath{
\draw (-1.5, -.4) --(-1.5, -.1);
\draw (0, -.4) --(0, .4);
\draw (1.5, -.4) --(1.5, -.1);

\draw (-2.25, -.1) rectangle (-.75, .4);
\node at (-1.5, .15) {$\Delta$};
\draw (-2, .4) -- (-2, 5.6);
\draw (2, .4) -- (2, 8.9);
\draw (2.25, -.1) rectangle (.75, .4);
\node at (1.5, .15) {$\Delta$};
\draw (1, .4) --(1, 3.2);
\draw [knot] (0, .4) to[out=90,in=-90](-1, 1.9);
\draw [knot] (-1, .4) to[out=90,in=-90](0, 1.9);
\draw (-.25, 1.9) rectangle (.25, 2.4);
\node at (0, 2.15) {$S$};
\draw (-1, 1.9) -- (-1, 4.3);
\draw (0, 2.4) -- (0, 3.2);
\draw (.5, 3.7) -- (.5, 4.3);
\draw (1.25, 3.2) rectangle (-.25, 3.7);
\node at (.5, 3.45) {$\widehat{\mu}$};
\draw (-1.25, 4.3) rectangle (.75, 4.8);
\node at (-.25, 4.55) {$m\circ c$};
\draw (0, 4.8) -- (0, 6.1);

\draw (-2.25, 5.6) rectangle (-.75, 6.1);
\node at (-1.5, 5.85) {$\Delta$};
\draw [knot] (0, 6.1) to[out=90,in=-90](-1, 7.6);
\draw [knot] (-1, 6.1) to[out=90,in=-90](0, 7.6);
\draw (-1, 10) -- (-1, 7.6);
\draw (0, 8.9) -- (0, 8.1);
\draw (-.25, 7.6) rectangle (.25, 8.1);
\node at (0, 7.85) {$S$};
\draw (2.25, 8.9) rectangle (-.25, 9.4);
\node at (1, 9.15) {$\widehat{\mu}$};

\draw (-1.25, 10) rectangle (.75, 10.5);
\node at (-.25, 10.25) {$c$};
\draw (.5, 9.4) -- (.5, 10);
\draw (-.25, 10.5) -- (-.25, 11.1);
\draw (-2, 6.1) -- (-2, 11.1);
\draw (-2.25, 11.1) rectangle (0, 11.6);
\node at (-1.125, 11.35) {$\mu$};
}\hspace{2 mm} + \hspace{1 mm}\tikzmath{
\draw (-1.5, -.4) --(-1.5, 1.2);
\draw (0, -.4) --(0, 1.7);
\draw (1, -.4) --(1, -.1);
\draw (1, .4) --(1, 4.5);

\draw (.75, .4) rectangle (1.25, -.1);
\node at (1, .15) {$\widehat{m}$};
\draw (-2.25, 1.2) rectangle (-.75, 1.7);
\node at (-1.5, 1.45) {$\Delta$};
\draw (-2, 5.6) -- (-2, 1.7);
\draw [knot] (0, 1.7) to[out=90,in=-90](-1, 3.2);
\draw [knot] (-1, 1.7) to[out=90,in=-90](0, 3.2);
\draw (-1, 5.6) -- (-1, 3.2);
\draw (0, 4.5) -- (0, 3.7);
\draw (-.25, 3.2) rectangle (.25, 3.7);
\node at (0, 3.45) {$S$};
\draw (1.25, 4.5) rectangle (-.25, 5);
\node at (.5, 4.75) {$\widehat{\mu}$};

\draw (-1.25, 5.6) rectangle (.75, 6.1);
\node at (-.25, 5.85) {$c$};
\draw (.5, 5.6) -- (.5, 5);
\draw (-.25, 6.1) -- (-.25, 6.7);
\draw (-2, 5.6) -- (-2, 6.7);
\draw (-2.25, 6.7) rectangle (0, 7.2);
\node at (-1.125, 6.95) {$\mu$};
\node at (1.5, -.4) {.};
}
\end{equation}
The formula \eqref{equ:mudelta} follows from the definition of $\beta_{\iota(A)}$ and $I_A$.
Since $\widehat m$ is a coderivation, that is,
$\Delta\circ\widehat m=(\widehat m\otimes\id+\id\otimes\widehat m)\circ\Delta$, the first term in \eqref{equ:mudelta} becomes:
$$\tikzmath{
\draw (-1.5, -.4) --(-1.5, -.1);
\draw (0, -.4) --(0, 1.7);
\draw (1, -.4) --(1, 4.5);

\draw (-2.25, .4) rectangle (-.75, -.1);
\node at (-1.5, .15) {$\Delta$};
\draw (-1.75, 1.2) rectangle (-2.25, 1.7);
\node at (-2, 1.45) {$\widehat{m}$};
\draw (-2, .4) -- (-2, 1.2);
\draw (-1, .4) -- (-1, 1.7);
\draw (-2, 5.6) -- (-2, 1.7);
\draw [knot] (0, 1.7) to[out=90,in=-90](-1, 3.2);
\draw [knot] (-1, 1.7) to[out=90,in=-90](0, 3.2);
\draw (-1, 5.6) -- (-1, 3.2);
\draw (0, 4.5) -- (0, 3.7);
\draw (-.25, 3.2) rectangle (.25, 3.7);
\node at (0, 3.45) {$S$};
\draw (1.25, 4.5) rectangle (-.25, 5);
\node at (.5, 4.75) {$\widehat{\mu}$};

\draw (-1.25, 5.6) rectangle (.75, 6.1);
\node at (-.25, 5.85) {$c$};
\draw (.5, 5.6) -- (.5, 5);
\draw (-.25, 6.1) -- (-.25, 6.7);
\draw (-2, 5.6) -- (-2, 6.7);
\draw (-2.25, 6.7) rectangle (0, 7.2);
\node at (-1.125, 6.95) {$\mu$};
}\hspace{1 mm} + \hspace{2 mm}\tikzmath{
\draw (-1.5, -.4) --(-1.5, -.1);
\draw (0, -.4) --(0, 1.7);
\draw (1, -.4) --(1, 4.5);

\draw (-2.25, .4) rectangle (-.75, -.1);
\node at (-1.5, .15) {$\Delta$};
\draw (-.75, 1.2) rectangle (-1.25, 1.7);
\node at (-1, 1.45) {$\widehat{m}$};
\draw (-1, .4) -- (-1, 1.2);
\draw (-2, .4) -- (-2, 1.7);
\draw (-2, 5.6) -- (-2, 1.7);
\draw [knot] (0, 1.7) to[out=90,in=-90](-1, 3.2);
\draw [knot] (-1, 1.7) to[out=90,in=-90](0, 3.2);
\draw (-1, 5.6) -- (-1, 3.2);
\draw (0, 4.5) -- (0, 3.7);
\draw (-.25, 3.2) rectangle (.25, 3.7);
\node at (0, 3.45) {$S$};
\draw (1.25, 4.5) rectangle (-.25, 5);
\node at (.5, 4.75) {$\widehat{\mu}$};

\draw (-1.25, 5.6) rectangle (.75, 6.1);
\node at (-.25, 5.85) {$c$};
\draw (.5, 5.6) -- (.5, 5);
\draw (-.25, 6.1) -- (-.25, 6.7);
\draw (-2, 5.6) -- (-2, 6.7);
\draw (-2.25, 6.7) rectangle (0, 7.2);
\node at (-1.125, 6.95) {$\mu$};
}\hspace{2 mm} \overset{\eqref{S:differential}}{=} \hspace{2 mm}\tikzmath{
\draw (-1.5, -.4) --(-1.5, -.1);
\draw (0, -.4) --(0, 1.7);
\draw (1, -.4) --(1, 4.5);

\draw (-2.25, .4) rectangle (-.75, -.1);
\node at (-1.5, .15) {$\Delta$};
\draw (-1.75, 1.2) rectangle (-2.25, 1.7);
\node at (-2, 1.45) {$\widehat{m}$};
\draw (-2, .4) -- (-2, 1.2);
\draw (-1, .4) -- (-1, 1.7);
\draw (-2, 5.6) -- (-2, 1.7);
\draw [knot] (0, 1.7) to[out=90,in=-90](-1, 3.2);
\draw [knot] (-1, 1.7) to[out=90,in=-90](0, 3.2);
\draw (-1, 5.6) -- (-1, 3.2);
\draw (0, 4.5) -- (0, 3.7);
\draw (-.25, 3.2) rectangle (.25, 3.7);
\node at (0, 3.45) {$S$};
\draw (1.25, 4.5) rectangle (-.25, 5);
\node at (.5, 4.75) {$\widehat{\mu}$};

\draw (-1.25, 5.6) rectangle (.75, 6.1);
\node at (-.25, 5.85) {$c$};
\draw (.5, 5.6) -- (.5, 5);
\draw (-.25, 6.1) -- (-.25, 6.7);
\draw (-2, 5.6) -- (-2, 6.7);
\draw (-2.25, 6.7) rectangle (0, 7.2);
\node at (-1.125, 6.95) {$\mu$};
}\hspace{1 mm} + \hspace{2 mm}\tikzmath{
\draw (-1.5, -.4) --(-1.5, -.1);
\draw (0, -.4) --(0, 1.7);
\draw (1, -.4) --(1, 4.5);

\draw (-2.25, .4) rectangle (-.75, -.1);
\node at (-1.5, .15) {$\Delta$};
\draw (-.75, 1.2) rectangle (-1.25, 1.7);
\node at (-1, 1.45) {$S$};
\draw (-1, .4) -- (-1, 1.2);
\draw (-2, .4) -- (-2, 1.7);
\draw (-2, 5.6) -- (-2, 1.7);
\draw [knot] (0, 1.7) to[out=90,in=-90](-1, 3.2);
\draw [knot] (-1, 1.7) to[out=90,in=-90](0, 3.2);
\draw (-1, 5.6) -- (-1, 3.2);
\draw (0, 4.5) -- (0, 3.7);
\draw (-.25, 3.2) rectangle (.25, 3.7);
\node at (0, 3.45) {$\widehat{m}$};
\draw (1.25, 4.5) rectangle (-.25, 5);
\node at (.5, 4.75) {$\widehat{\mu}$};

\draw (-1.25, 5.6) rectangle (.75, 6.1);
\node at (-.25, 5.85) {$c$};
\draw (.5, 5.6) -- (.5, 5);
\draw (-.25, 6.1) -- (-.25, 6.7);
\draw (-2, 5.6) -- (-2, 6.7);
\draw (-2.25, 6.7) rectangle (0, 7.2);
\node at (-1.125, 6.95) {$\mu$};
\node at (1.5, -.4) {.};
}
$$
Using
$\widehat{\mu}(\widehat{m}\otimes\id+\id\otimes\widehat{m})=\widehat{m}\circ\widehat{\mu}$ together with \eqref{equ:mwithc} and the compatibility of $\widehat{\mu}(S\otimes \id)$ with $\Delta$ by Proposition \ref{prop:dualitybinfinity}, the expression \eqref{equ:mudelta} becomes
$$\tikzmath{
\draw (-1.5, -.4) --(-1.5, -.1);
\draw (0, -.4) --(0, 1.7);
\draw (1, -.4) --(1, 4.5);

\draw (-2.25, .4) rectangle (-.75, -.1);
\node at (-1.5, .15) {$\Delta$};
\draw (-1, .4) -- (-1, 1.7);
\draw (-2, 5.6) -- (-2, .4);
\draw [knot] (0, 1.7) to[out=90,in=-90](-1, 3.2);
\draw [knot] (-1, 1.7) to[out=90,in=-90](0, 3.2);
\draw (-1, 5.6) -- (-1, 3.2);
\draw (0, 4.5) -- (0, 3.7);
\draw (-.25, 3.2) rectangle (.25, 3.7);
\node at (0, 3.45) {$S$};
\draw (1.25, 4.5) rectangle (-.25, 5);
\node at (.5, 4.75) {$\widehat{\mu}$};

\draw (-1.25, 5.6) rectangle (.75, 6.1);
\node at (-.25, 5.85) {$c$};
\draw (.5, 5.6) -- (.5, 5);
\draw (-.25, 6.1) -- (-.25, 6.7);
\draw (-2, 5.6) -- (-2, 6.7);
\draw (-2.25, 6.7) rectangle (0, 7.2);
\node at (-1.125, 6.95) {$m\circ \widehat{\mu}$};
}\hspace{1 mm} \overset{(\ref{equ:mpq})}{=} \hspace{1 mm} \tikzmath{
\draw (-3, 1.7) --(-3, 7.7);
\draw (-2, .8) --(-2, 1.2);
\draw (0, .8) --(0, 1.7);
\draw (1.5, .8) --(1.5, 1.2);

\draw (-3.25, 1.2) rectangle (-.75, 1.7);
\node at (-2, 1.45) {$\Delta^{(2)}$};
\draw (2.25, 1.2) rectangle (.75, 1.7);
\node at (1.5, 1.45) {$\Delta$};
\draw (1, 4.5) -- (1, 1.7);
\draw (-2, 5.6) -- (-2, 1.7);
\draw [knot] (0, 1.7) to[out=90,in=-90](-1, 3.2);
\draw [knot] (-1, 1.7) to[out=90,in=-90](0, 3.2);
\draw (-1, 5.6) -- (-1, 3.2);
\draw (0, 4.5) -- (0, 3.7);
\draw (-.25, 3.2) rectangle (.25, 3.7);
\node at (0, 3.45) {$S$};
\draw (1.25, 4.5) rectangle (-.25, 5);
\node at (.5, 4.75) {$\widehat{\mu}$};

\draw (-1.25, 5.6) rectangle (.75, 6.1);
\node at (-.25, 5.85) {$c$};
\draw (.5, 5.6) -- (.5, 5);
\draw (-.25, 6.1) -- (-.25, 6.7);
\draw (-2, 5.6) -- (-2, 6.7);
\draw (-2.25, 6.7) rectangle (0, 7.2);
\node at (-1.125, 6.95) {$\mu$};
\draw (2, 1.7) -- (2, 7.7);
\draw (-1.125, 7.2) -- (-1.125, 7.7);
\draw (-3.25, 7.7) rectangle (2.25, 8.2);
\node at (-.5, 7.95) {$m\circ c^{(2)}$};
\node at (2.5, 0.8) {.};
}
$$
This is precisely $\beta_A\circ\widehat I_A$.
\end{proof}

We also need the following definition of the opposite $B_{\infty}$-algebra of any given $B_{\infty}$-algebra; see \cite[Definition 5.5]{CLW}.
\begin{defn}\label{defnopposite}
Let $(A, m^n; \mu^{p, q})$ be a $B_{\infty}$-algebra. The {\it opposite $B_{\infty}$-algebra} of $A$ is defined to be the $B_{\infty}$-algebra
$(A, m^n; (\mu^{p, q})^{\rm opp})$, where $(\mu^{p, q})^{\rm opp}\colon (sA)^{\otimes q} \otimes (sA)^{\otimes p} \to sA$ is given by
$$
(\mu^{p, q})^{\rm opp} (sa_{1, q} \otimes sb_{1, p}) = (-1)^{\left((|a_1|-1)+\dotsb+(|a_q|-1)\right)\left((|b_1|-1)+\dotsb+(|b_p|-1)\right) } \mu^{p, q}(sb_{1, p} \otimes sa_{1, q}).
$$
\end{defn}

We simply write the opposite $B_\infty$-algebra $(A, m^n; (\mu^{p, q})^{\rm opp})$ as $A^{\rm opp}$.
We denote the dg bialgebra associated with $(A, m^n; (\mu^{p, q})^{\rm opp})$ by $(T^c(sA), \Delta, \widehat{m}, \widehat{\mu}^{\opp})$.
Since $\widehat{\mu}^{\rm opp}=\widehat{\mu^{\rm opp}}$, the dg bialgebra associated with $A^{\rm opp}$ is also given by $(T^c(sA), \Delta, \widehat{m}, \widehat{\mu}^{\rm opp})$. In particular, $A^{\rm opp}$ and $A$ have the same underlying $A_\infty$-algebra structure. As $B_\infty$-algebras, we have $(A^{\rm opp})^{\rm opp} = A$.


\section{Monoidal triangulated categories arising from $B_\infty$-algebras}\label{section:monoidaltri}

This section proves the monoidal part of the main theorem. Let $(A,m^n;\mu^{p,q})$ be a $B_\infty$-algebra. Proposition \ref{prop:dualitybinfinity} gives an induction functor
\[
\iota\colon \ED^{\RR}_\infty(A)\longrightarrow \ED^{\BB}_\infty(A)
\]
from right $A_\infty$-modules to $A_\infty$-bimodules. We use this functor to define a tensor product on the derived category of right modules by
\[
M\boxtimes_A N=M\overset{\infty}{\otimes}_A\iota(N).
\]
The main point is that this tensor product is not obtained by formally transferring the monoidal structure on bimodules. Since $\iota$ is not full in general, the associativity and unit constraints must be constructed and checked explicitly.

By a {\it monoidal triangulated category} we mean a monoidal category $\EC$ in the sense of \cite[Definition 2.1.1]{EGNO} which is also triangulated, and whose tensor product
\[
-\otimes -\colon \EC\times \EC\longrightarrow \EC
\]
is exact in each variable; see, for example, \cite[Definition 2.1]{NVY02}.

\begin{thm}\label{thm:monoidalcatofrightmod}
Let $(A,m^n;\mu^{p,q})$ be a $B_\infty$-algebra. Then $\ED^{\RR}_\infty(A)$ is a monoidal triangulated category with unit object $A$. For right $A_\infty$-modules $M$ and $N$, the tensor product is
\begin{align}
M \boxtimes_A N := M\overset{\infty}{\otimes}_A \iota(N).
\end{align}
\end{thm}

\begin{rem}\label{rem:monoidalbrace}
As mentioned in Remark \ref{rem:iotafaithfunonfull}, the functor $\iota$ is generally not fully faithful. Therefore the monoidal structure on $\ED^{\RR}_\infty(A)$ is not a formal consequence of the monoidal structure on $\ED^{\BB}_\infty(A)$. The monoidal axioms for $\boxtimes_A$ must be verified directly. In particular, the unit axiom requires an $A_\infty$-bimodule quasi-isomorphism $\iota(A)\simeq A$.
\end{rem}

We first prove the compatibility of $\iota$ with the tensor product. This is the key input for the associativity constraint: the bimodules $\iota(M)\overset{\infty}{\otimes}_A \iota(N)$ and $\iota(M\overset{\infty}{\otimes}_A \iota(N))$ must be identified up to $A_{\infty}$-quasi-isomorphism.
\begin{lem}\label{lem:isomorphismmn}
Let $(A,m^n;\mu^{p,q})$ be a $B_\infty$-algebra. For any right $A_\infty$-modules $(M,\rho_M^n)$ and $(N,\rho_N^n)$ over $A$, there is an $A_\infty$-bimodule quasi-isomorphism
\[
 \psi\colon \iota(M)\overset{\infty}{\otimes}_A \iota(N)\xrightarrow{\simeq} \iota(M\overset{\infty}{\otimes}_A \iota(N)).
\]
\end{lem}

\begin{proof}
We define maps
$$\psi^{p,q}:(sA)^{\otimes p} \otimes ( \iota(M)\overset{\infty}{\otimes}_A \iota(N)) \otimes (sA)^{\otimes q} \rightarrow \iota(M\overset{\infty}{\otimes}_A \iota(N))$$
as follows. If $q>0$, set $\psi^{p,q}=0$. If $q=0$, define
$$\psi^{p,0}\colon (sA)^{\otimes p} \otimes s M \otimes T^c(sA) \otimes N \to sM\otimes T^c(sA) \otimes N $$
by
\begin{align*}
\psi^{p,0}(sa_{1,p}\otimes sx\otimes sc_{1,k} \otimes y)= (-1)^{|sa_{1,p}||sx|} sx\otimes \widehat{\mu}(S(sa_{1,p})\otimes sc_{1,k})\otimes y,
\end{align*}
 Diagrammatically, this is
$$\psi^{p,0}= \tikzmath{
\draw (-1, .6) --(-1, 1.5);
\draw (0, .6) --(0, 1.5);
\draw (1, .6) --(1, 4.2);
\draw (2, .6) --(2, 5.3);
\node at (2.5, -.1) {.};

\node at (-1, 0) {$(sA)^{\otimes p}$};
\node at (0, 0) {$sM$};
\node at (1, 0) {$T^c(sA)$};
\node at (2, 0) {$N$};

\draw [knot] (0, 1.5) to[out=90,in=-90](-1, 3);
\draw [knot] (-1, 1.5) to[out=90,in=-90](0, 3);
\draw (-1, 3) rectangle (-1, 5.3);

\draw (-.25, 3) rectangle (.25, 3.5);
\node at (0, 3.25) {$S$};
\draw (0, 3.5) -- (0,4.2);
\draw (-.25, 4.2) rectangle (1.25, 4.7);
\node at (0.5, 4.45) {$\widehat{\mu}$};

\draw (.5, 4.7) -- (.5, 5.3);
}$$
Here we use the natural identification $s(M\otimes (sA)^{\otimes m}\otimes N) \simeq sM\otimes (sA)^{\otimes m} \otimes N$.

Since $\psi^{0,0}=\id_{sM\otimes T^c(sA) \otimes N }$, it suffices to prove that $\psi$ is an $A_{\infty}$-bimodule morphism, namely
\begin{align}\label{align:Hbimodule}
\psi \circ \widehat{\beta}_{\iota(M)\overset{\infty}{\otimes}_A \iota(N)} = \beta_{\iota(M\overset{\infty}{\otimes}_A \iota(N))}\circ\widehat{\psi}.
\end{align}
Both sides of \eqref{align:Hbimodule} are evaluated on $(sA)^{\otimes p}\otimes sM\otimes (sA)^{\otimes m}\otimes N \otimes (sA)^{\otimes q}$. We prove the identity in the following two cases: $q=0$ and $q>0$.

For $q=0$, the left-hand side of the equation \eqref{align:Hbimodule} equals
$$\tikzmath{
\draw (-1, 0) --(-1, 2);
\draw (0, 0) --(0, 4);
\draw (1, 0) --(1, 4);
\draw (2, 0) --(2, 4);

\node at (-1, -.5) {$(sA)^{\otimes p}$};
\node at (0, -.5) {$sM$};
\node at (1, -.5) {$(sA)^{\otimes m}$};
\node at (2, -.5) {$N$};

\draw (-1.25, 2) rectangle (-.75, 2.5);
\node at (-1, 2.25) {$\widehat{m}$};
\draw (-1, 2.5) -- (-1, 4);
\draw (-1.25, 4) rectangle (2.25, 4.5);
\node at (0.5, 4.25) {$\psi$};
}\hspace{1 mm} + \hspace{1 mm}
\tikzmath{
\draw (-1.5, 0) --(-1.5, 1.2);
\draw (0, 0) --(0, 2.5);
\draw (1, 0) --(1, 2.5);
\draw (2, 0) --(2, 2.5);
\node at (2.5, -.6) {,};

\node at (-1.5, -.5) {$(sA)^{\otimes p}$};
\node at (-.3, -.5) {$sM$};
\node at (1, -.5) {$(sA)^{\otimes m}$};
\node at (2, -.5) {$N$};

\draw (-2.25, 1.2) rectangle (-.75, 1.7);
\node at (-1.5, 1.45) {$\Delta$};
\draw (-2, 1.7) -- (-2, 4);
\draw (-1, 1.7) -- (-1, 2.5);
\draw (-1.25, 2.5) rectangle (2.25, 3.2);
\node at (0.5, 2.82) {$\beta_{\iota(M)\overset{\infty}{\otimes}_A \iota(N)}$};
\draw (-2.25, 4) rectangle (2.25, 4.5);
\node at (0, 4.25) {$\psi$};
\draw (-.8, 3.2) -- (-.8, 4);
\draw (.5, 3.2) -- (.5, 4);
\draw (1.8, 3.2) -- (1.8, 4);
}
$$
The second summand in the diagram is further equal to
$$\tikzmath{
\draw (-1, 0) --(-1, 4);
\draw (0, 0) --(0, 4);
\draw (1, 0) --(1, 2);
\draw (2, 0) --(2, 4);
\node at (-1, -.3) {$(sA)^{\otimes p}$};
\node at (0, -.3) {$sM$};
\node at (1, -.3) {$(sA)^{\otimes m}$};
\node at (2, -.3) {$N$};

\draw (1.25, 2) rectangle (.75, 2.5);
\node at (1, 2.25) {$\widehat{m}$};
\draw (1, 2.5) -- (1, 4);
\draw (-1.25, 4) rectangle (2.25, 4.5);
\node at (0.5, 4.25) {$\psi$};
}\hspace{1 mm} + \hspace{1 mm}\tikzmath{
\draw (-1.5, 0) --(-1.5, 1.2);
\draw (0, 0) --(0, 2.7);
\draw (1.5, 0) --(1.5, 1.2);
\draw (3, 0) --(3, 4);

\draw (-2.25, 1.2) rectangle (-.75, 1.7);
\node at (-1.5, 1.45) {$\Delta$};
\draw (2.25, 1.2) rectangle (.75, 1.7);
\node at (1.5, 1.45) {$\Delta$};
\draw (-2, 1.7) -- (-2, 4);
\draw (-1, 1.7) -- (-1, 2.7);
\draw (2, 1.7) -- (2, 4);
\draw (1, 1.7) -- (1, 2.7);
\draw (-1.25, 2.7) rectangle (1.25, 3.2);
\node at (0, 2.92) {$\beta_{\iota(M)}$};
\draw (-2.25, 4) rectangle (3.25, 4.5);
\node at (0.5, 4.25) {$\psi$};
\draw (., 3.2) -- (., 4);
}\hspace{1 mm} - \hspace{1 mm}\tikzmath{
\draw (-1, 0) --(-1, 5.2);
\draw (0, 0) --(0, 5.2);
\draw (1.5, 0) --(1.5, 1.2);
\draw (3, 0) --(3, 1.2);
\node at (3.5, 0) {,};

\draw (2.25, 1.2) rectangle (.75, 1.7);
\node at (1.5, 1.45) {$\Delta$};
\draw (3.25, 1.2) rectangle (2.75, 1.7);
\node at (3, 1.45) {$s$};
\draw (1, 1.7) --(1, 5.2);

\draw (3.25, 3.2) rectangle (1.75, 2.7);
\node at (2.5, 2.92) {$\beta_{\iota(N)}$};
\draw (2, 1.7) --(2, 2.7);
\draw (3, 1.7) --(3, 2.7);
\draw (2.5, 3.2) --(2.5, 4);
\draw (2.8, 4) rectangle (2.2, 4.5);
\node at (2.5, 4.25) {$s^{-1}$};
\draw (2.5, 4.5) --(2.5, 5.2);

\draw (-1.25, 5.7) rectangle (3.25, 5.2);
\node at (1, 5.45) {$\psi$};
}$$
where the lowercase $s$ is the shift functor.
By the definition of $\psi$, the left-hand side of \eqref{align:Hbimodule} becomes
\begin{figure}[H]
\centering
$$\tikzmath{
\draw (-1, 0.4) --(-1, 2);
\draw (0, 0.4) --(0, 2.5);
\draw (1, 0.4) --(1, 5.3);
\draw (2, 0.4) --(2, 6.2);

\draw (-1.25, 2) rectangle (-.75, 2.5);
\node at (-1, 2.25) {$\widehat{m}$};
\draw [knot] (0, 2.5) to[out=90,in=-90](-1, 4);
\draw [knot] (-1, 2.5) to[out=90,in=-90](0, 4);
\draw (-.25, 4) rectangle (.25, 4.5);
\node at (0, 4.25) {$S$};
\draw (0, 4.5) -- (0,5.3);
\draw (-.25, 5.3) rectangle (1.25, 5.8);
\node at (0.5, 5.55) {$\widehat{\mu}$};
\draw (-1, 4) -- (-1, 6.2);
\draw (0.5, 5.8) -- (0.5, 6.2);
}\hspace{1 mm} + \hspace{1 mm}
\tikzmath{
\draw (-1, 0.4) --(-1, 2.5);
\draw (0, 0.4) --(0, 2.5);
\draw (1, 0.4) --(1, 2);
\draw (2, 0.4) --(2, 6.2);

\draw (1.25, 2) rectangle (.75, 2.5);
\node at (1, 2.25) {$\widehat{m}$};
\draw (1, 2.5) -- (1, 5.3);
\draw [knot] (0, 2.5) to[out=90,in=-90](-1, 4);
\draw [knot] (-1, 2.5) to[out=90,in=-90](0, 4);
\draw (-.25, 4) rectangle (.25, 4.5);
\node at (0, 4.25) {$S$};
\draw (0, 4.5) -- (0,5.3);
\draw (-.25, 5.3) rectangle (1.25, 5.8);
\node at (0.5, 5.55) {$\widehat{\mu}$};
\draw (-1, 4) -- (-1, 6.2);
\draw (0.5, 5.8) -- (0.5, 6.2);
}\hspace{1 mm} + \hspace{1 mm}\tikzmath{
\draw (-1.5, 0.7) --(-1.5, 1.2);
\draw (0, 0.7) --(0, 1.7);
\draw (1.5, 0.7) --(1.5, 1.2);
\draw (3, 0.7) --(3, 10);

\draw (-2.25, 1.2) rectangle (-.75, 1.7);
\node at (-1.5, 1.45) {$\Delta$};
\draw (2.25, 1.2) rectangle (.75, 1.7);
\node at (1.5, 1.45) {$\Delta$};

\draw [knot] (0, 1.7) to[out=90,in=-90](-1, 3.2);
\draw [knot] (-1, 1.7) to[out=90,in=-90](0, 3.2);
\draw (-1, 3.2) --(-1, 5.8);
\draw (-2, 1.7) --(-2, 6.3);
\draw (1, 1.7) --(1, 4.5);
\draw (2, 1.7) --(2, 9.1);

\draw (-.25, 3.2) rectangle (.25, 3.7);
\node at (0, 3.45) {$S$};
\draw (0, 3.7) --(0, 4.5);
\draw (-.25, 4.5) rectangle (1.25, 5);
\node at (0.5, 4.75) {$\widehat{\mu}$};
\draw (.5, 5) --(.5, 5.8);
\draw (-1.25, 5.8) rectangle (.75, 6.3);
\node at (-.25, 6.05) {$\rho_{M}$};

\draw [knot] (-1, 6.3) to[out=90,in=-90](-1, 10);
\draw [knot] (-2, 6.3) to[out=90,in=-90](1, 7.8);

\draw (1.25, 7.8) rectangle (.75, 8.3);
\node at (1, 8.05) {$S$};
\draw (1, 8.3) --(1, 9.1);
\draw (.75, 9.1) rectangle (2.25, 9.6);
\node at (1.5, 9.35) {$\widehat{\mu}$};
\draw (1.5, 9.6) --(1.5, 10);
}\hspace{1 mm} - \hspace{1 mm}\tikzmath{
\draw (-1, 0.7) --(-1, 6.3);
\draw (0, 0.7) --(0, 6.3);
\draw (1.5, 0.7) --(1.5, 1.2);
\draw (3, 0.7) --(3, 1.2);
\node at (3.5, 0.8) {.};

\draw (2.25, 1.2) rectangle (.75, 1.7);
\node at (1.5, 1.45) {$\Delta$};
\draw (3.25, 1.2) rectangle (2.75, 1.7);
\node at (3, 1.45) {$s$};
\draw (1, 1.7) --(1, 9.1);
\draw [knot] (3, 1.7) to[out=90,in=-90](2, 3.2);
\draw [knot] (2, 1.7) to[out=90,in=-90](3, 3.2);

\draw (3.25, 3.2) rectangle (2.75, 3.7);
\node at (3, 3.45) {$S$};
\draw (2, 3.2) --(2, 4.5);
\draw (3, 3.7) --(3, 4.5);
\draw (3.25, 4.5) rectangle (1.75, 5);
\node at (2.5, 4.75) {$\rho_N$};
\draw (2.5, 5) --(2.5, 5.8);
\draw (2.2, 5.8) rectangle (2.8, 6.3);
\node at (2.5, 6.05) {$s^{-1}$};

\draw [knot] (0, 6.3) to[out=90,in=-90](-1, 7.8);
\draw [knot] (-1, 6.3) to[out=90,in=-90](0, 7.8);
\draw (-1, 7.8) --(-1, 10);
\draw (2.5, 6.3) --(2.5, 10);
\draw (-.25, 7.8) rectangle (.25, 8.3);
\node at (0, 8.05) {$S$};
\draw (0, 8.3) --(0, 9.1);
\draw (-.25, 9.1) rectangle (1.25, 9.6);
\node at (0.5, 9.35) {$\widehat{\mu}$};
\draw (0.5, 10) --(0.5, 9.6);
}$$
\caption{The left-hand side of \eqref{align:Hbimodule} when $q=0$.}
\label{fig:LHS}
\end{figure}

For $q=0$ the right-hand side of \eqref{align:Hbimodule} equals
$$\tikzmath{
\draw (-1.5, 0.6) --(-1.5, 1.2);
\draw (0, 0.6) --(0, 2.5);
\draw (1, 0.6) --(1, 2.5);
\draw (2, 0.6) --(2, 2.5);
\node at (2.6, -.) {.};

\node at (-1.5, .1) {$(sA)^{\otimes p}$};
\node at (-.3, .1) {$sM$};
\node at (1, .1) {$(sA)^{\otimes m}$};
\node at (2, .1) {$N$};

\draw (-2.25, 1.2) rectangle (-.75, 1.7);
\node at (-1.5, 1.45) {$\Delta$};
\draw (-2, 1.7) -- (-2, 4);
\draw (-1, 1.7) -- (-1, 2.5);
\draw (-1.25, 2.5) rectangle (2.25, 3);
\node at (0.5, 2.75) {$\psi$};
\draw (-2.25, 4) rectangle (2.25, 4.7);
\node at (0, 4.35) {$\beta_{\iota(M\overset{\infty}{\otimes}_A \iota(N))}$};
\draw (-.8, 3) -- (-.8, 4);
\draw (.5, 3) -- (.5, 4);
\draw (1.8, 3) -- (1.8, 4);
}$$
More specifically, we have
\begin{figure}[H]
\centering
$$\tikzmath{
\draw (-1.5, 0.4) --(-1.5, 1.2);
\draw (0, 0.4) --(0, 1.7);
\draw (1, 0.4) --(1, 4.5);
\draw (2, 0.4) --(2, 5);

\draw (-2.25, 1.2) rectangle (-.75, 1.7);
\node at (-1.5, 1.45) {$\Delta$};
\draw [knot] (0, 1.7) to[out=90,in=-90](-1, 3.2);
\draw [knot] (-1, 1.7) to[out=90,in=-90](0, 3.2);
\draw (-2, 1.7) --(-2, 5);
\draw (-1, 3.2) --(-1, 5);
\draw (-.25, 3.2) rectangle (.25, 3.7);
\node at (0, 3.45) {$S$};
\draw (0, 3.7) --(0, 4.5);
\draw (-.25, 4.5) rectangle (1.25, 5);
\node at (0.5, 4.75) {$\widehat{\mu}$};
\draw [knot] (-1, 5) to[out=90,in=-90](-1, 7.8);
\draw [knot] (0.5, 5) to[out=90,in=-90](.5, 7.8);
\draw [knot] (2, 5) to[out=90,in=-90](2, 7.8);
\draw [knot] (-2, 5) to[out=90,in=-90](3, 6.5);

\draw (3.25, 6.5) rectangle (2.75, 7);
\node at (3, 6.75) {$S$};
\draw (3, 7) --(3, 7.8);
\draw (-1.25, 7.8) rectangle (3.25, 8.5);
\node at (1, 8.15) {$\beta_{M\overset{\infty}{\otimes}_A \iota(N)}$};
}\hspace{1 mm} = \hspace{1 mm}\tikzmath{
\draw (-1, 0.4) --(-1, 1.2);
\draw (0, 0.4) --(0, 1.2);
\draw (1, 0.4) --(1, 4);
\draw (2, 0.4) --(2, 7.3);

\draw [knot] (0, 1.2) to[out=90,in=-90](-1, 2.7);
\draw [knot] (-1, 1.2) to[out=90,in=-90](0, 2.7);
\draw (-.25, 3.2) rectangle (.25, 2.7);
\node at (0, 2.95) {$S$};
\draw (0, 3.2) --(0, 4);
\draw (-.25, 4.5) rectangle (1.25, 4);
\node at (0.5, 4.25) {$\widehat{\mu}$};
\draw (-1, 2.7) --(-1, 7.3);
\draw (.5, 4.5) --(.5, 5.3);
\draw (-.25, 5.3) rectangle (1.25, 5.8);
\node at (0.5, 5.55) {$\widehat{m}$};
\draw (.5, 5.8) --(.5, 7.3);

}\hspace{1 mm} + \hspace{1 mm}\tikzmath{
\draw (-1, 0.4) --(-1, 1.2);
\draw (0, 0.4) --(0, 1.2);
\draw (1, 0.4) --(1, 4);
\draw (2, 0.4) --(2, 7.3);

\draw [knot] (0, 1.2) to[out=90,in=-90](-1, 2.7);
\draw [knot] (-1, 1.2) to[out=90,in=-90](0, 2.7);
\draw (-.25, 3.2) rectangle (.25, 2.7);
\node at (0, 2.95) {$S$};
\draw (0, 3.2) --(0, 4);
\draw (-.25, 4.5) rectangle (1.25, 4);
\node at (0.5, 4.25) {$\widehat{\mu}$};
\draw (-1, 2.7) --(-1, 6.4);
\draw (.5, 4.5) --(.5, 5.3);
\draw (-.25, 5.3) rectangle (1.25, 5.8);
\node at (0.5, 5.55) {$\Delta$};
\draw (., 5.8) --(., 6.4);
\draw (1, 5.8) --(1, 7.3);

\draw (-1.25, 6.4) rectangle (.25, 6.9);
\node at (-.5, 6.65) {$\rho_M$};
\draw (-.5, 6.9) --(-.5, 7.3);
}\hspace{1 mm} - \hspace{1 mm}\tikzmath{
\draw (-1.5, 0.4) --(-1.5, 1.2);
\draw (0, 0.4) --(0, 1.7);
\draw (1, 0.4) --(1, 4.5);
\draw (2, 0.4) --(2, 4.5);
\node at (3, 0) {.};

\draw (-2.25, 1.2) rectangle (-.75, 1.7);
\node at (-1.5, 1.45) {$\Delta$};
\draw [knot] (0, 1.7) to[out=90,in=-90](-1, 3.2);
\draw [knot] (-1, 1.7) to[out=90,in=-90](0, 3.2);
\draw (-2, 1.7) --(-2, 5);
\draw (-1, 3.2) --(-1, 5);
\draw (-.25, 3.2) rectangle (.25, 3.7);
\node at (0, 3.45) {$S$};
\draw (0, 3.7) --(0, 4.5);
\draw (-.25, 4.5) rectangle (1.25, 5);
\node at (0.5, 4.75) {$\widehat{\mu}$};
\draw [knot] (-1, 5) to[out=90,in=-90](-1, 8.7);
\draw [knot] (0.5, 5) to[out=90,in=-90](.5, 6.5);
\draw [knot] (2, 5) to[out=90,in=-90](2, 7.8);
\draw [knot] (-2, 5) to[out=90,in=-90](3, 6.5);

\draw (2.25, 4.5) rectangle (1.75, 5);
\node at (2, 4.75) {$s$};
\draw (3.25, 6.5) rectangle (2.75, 7);
\node at (3, 6.75) {$S$};
\draw (-.25, 6.5) rectangle (1.25, 7);
\node at (.5, 6.75) {$\Delta$};
\draw (0, 7) --(0, 8.7);
\draw (1, 7) --(1, 7.8);
\draw (3, 7) --(3, 7.8);
\draw (3.25, 7.8) rectangle (.75, 8.3);
\node at (2, 8.05) {$s^{-1}\circ\beta_{\iota(N)}$};
\draw (2, 8.3) --(2, 8.7);
}$$
\caption{The right-hand side of \eqref{align:Hbimodule} when $q=0$.}
\label{fig:RHS}
\end{figure}

We now compare the summands in Figures \ref{fig:LHS} and \ref{fig:RHS}. Observe that:
\begin{itemize}
\item The first two summands in Figure \ref{fig:LHS} combine to give the first summand in Figure \ref{fig:RHS}, because $\varphi=\widehat{\mu}(S\otimes\id)$ is compatible with $\widehat{m}$.
\item The third summand in Figure \ref{fig:LHS} coincides with the second summand in Figure \ref{fig:RHS}, since $\Delta\varphi= (\varphi\otimes\varphi)(\id\otimes\tau\otimes\id)(\Delta^{\op}\otimes\Delta)$.
\item Finally, using Lemma \ref{lem:cofreehopfidentity}, the last summand in Figure \ref{fig:RHS} can be simplified as follows:
\end{itemize}
$$\tikzmath{
\draw (-1.5, 0) --(-1.5, 1.2);
\draw (0, 0) --(0, 1.7);
\draw (1, 0) --(1, 4.5);
\draw (2, 0) --(2, 4.5);

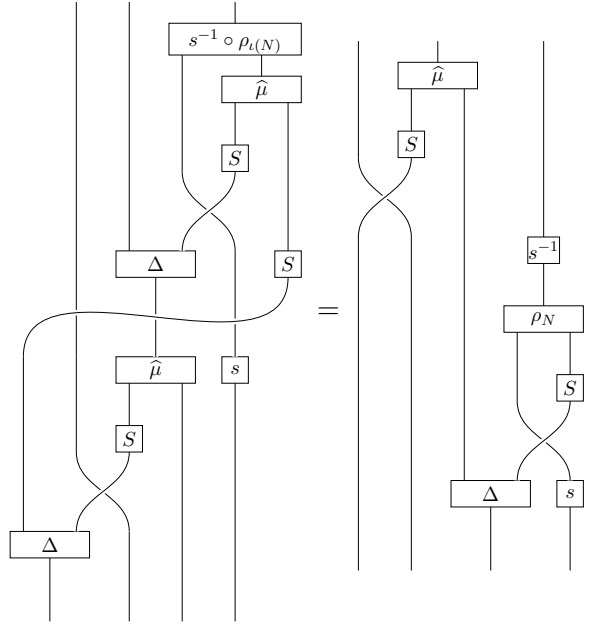
\draw (-2.25, 1.2) rectangle (-.75, 1.7);
\node at (-1.5, 1.45) {$\Delta$};
\draw [knot] (0, 1.7) to[out=90,in=-90](-1, 3.2);
\draw [knot] (-1, 1.7) to[out=90,in=-90](0, 3.2);
\draw (-2, 1.7) --(-2, 5);
\draw (-1, 3.2) --(-1, 5);
\draw (-.25, 3.2) rectangle (.25, 3.7);
\node at (0, 3.45) {$S$};
\draw (0, 3.7) --(0, 4.5);
\draw (-.25, 4.5) rectangle (1.25, 5);
\node at (0.5, 4.75) {$\widehat{\mu}$};
\draw (2.25, 4.5) rectangle (1.75, 5);
\node at (2, 4.75) {$s$};
\draw [knot] (-1, 5) to[out=90,in=-90](-1, 11.7);
\draw [knot] (0.5, 5) to[out=90,in=-90](.5, 6.5);
\draw [knot] (2, 5) to[out=90,in=-90](2, 7);
\draw [knot] (-2, 5) to[out=90,in=-90](3, 6.5);

\draw (3.25, 6.5) rectangle (2.75, 7);
\node at (3, 6.75) {$S$};
\draw (-.25, 6.5) rectangle (1.25, 7);
\node at (.5, 6.75) {$\Delta$};
\draw (0, 7) --(0, 11.7);
\draw [knot] (2, 7) to[out=90,in=-90](1, 8.5);
\draw [knot] (1, 7) to[out=90,in=-90](2, 8.5);
\draw (3, 7) --(3, 9.8);
\draw (2.25, 8.5) rectangle (1.75, 9);
\node at (2, 8.75) {$S$};
\draw (2, 9) --(2, 9.8);
\draw (1.75, 9.8) rectangle (3.25, 10.3);
\node at (2.5, 10.05) {$\widehat{\mu}$};
\draw (1, 8.5) --(1, 10.7);
\draw (2.5, 10.3) --(2.5, 10.7);
\draw (3.25, 10.7) rectangle (.75, 11.3);
\node at (2, 10.95) {$s^{-1}\circ\rho_{\iota(N)}$};
\draw (2, 11.3) --(2, 11.7);
}\hspace{1 mm} = \hspace{1 mm}\tikzmath{
\draw (-1, 0) --(-1, 6.3);
\draw (0, 0) --(0, 6.3);
\draw (1.5, 0) --(1.5, 1.2);
\draw (3, 0) --(3, 1.2);
\node at (3.5, 0) {,};

\draw (2.25, 1.2) rectangle (.75, 1.7);
\node at (1.5, 1.45) {$\Delta$};
\draw (3.25, 1.2) rectangle (2.75, 1.7);
\node at (3, 1.45) {$s$};
\draw (1, 1.7) --(1, 9.1);
\draw [knot] (3, 1.7) to[out=90,in=-90](2, 3.2);
\draw [knot] (2, 1.7) to[out=90,in=-90](3, 3.2);

\draw (3.25, 3.2) rectangle (2.75, 3.7);
\node at (3, 3.45) {$S$};
\draw (2, 3.2) --(2, 4.5);
\draw (3, 3.7) --(3, 4.5);
\draw (3.25, 4.5) rectangle (1.75, 5);
\node at (2.5, 4.75) {$\rho_N$};
\draw (2.5, 5) --(2.5, 5.8);
\draw (2.2, 5.8) rectangle (2.8, 6.3);
\node at (2.5, 6.05) {$s^{-1}$};

\draw [knot] (0, 6.3) to[out=90,in=-90](-1, 7.8);
\draw [knot] (-1, 6.3) to[out=90,in=-90](0, 7.8);
\draw (-1, 7.8) --(-1, 10);
\draw (2.5, 6.3) --(2.5, 10);
\draw (-.25, 7.8) rectangle (.25, 8.3);
\node at (0, 8.05) {$S$};
\draw (0, 8.3) --(0, 9.1);
\draw (-.25, 9.1) rectangle (1.25, 9.6);
\node at (0.5, 9.35) {$\widehat{\mu}$};
\draw (0.5, 10) --(0.5, 9.6);
}$$
which is the last summand of Figure \ref{fig:LHS}.
Thus the two sides of \eqref{align:Hbimodule} agree when $q=0$.

Now assume $q>0$. In this case, the left-hand side of \eqref{align:Hbimodule} is
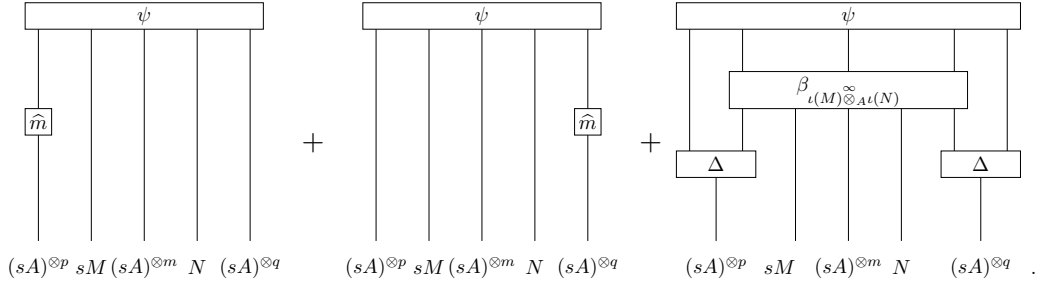
$$\tikzmath{
\draw (-1, 0) --(-1, 2);
\draw (0, 0) --(0, 4);
\draw (1, 0) --(1, 4);
\draw (2, 0) --(2, 4);
\draw (3, 0) --(3, 4);

\node at (-1, -.5) {$(sA)^{\otimes p}$};
\node at (0, -.5) {$sM$};
\node at (1, -.5) {$(sA)^{\otimes m}$};
\node at (2, -.5) {$N$};
\node at (3, -.5) {$(sA)^{\otimes q}$};

\draw (-1.25, 2) rectangle (-.75, 2.5);
\node at (-1, 2.25) {$\widehat{m}$};
\draw (-1, 2.5) -- (-1, 4);
\draw (-1.25, 4) rectangle (3.25, 4.5);
\node at (1, 4.25) {$\psi$};
}\hspace{1 mm} + \hspace{1 mm}\tikzmath{
\draw (-1, 0) --(-1, 4);
\draw (0, 0) --(0, 4);
\draw (1, 0) --(1, 4);
\draw (2, 0) --(2, 4);
\draw (3, 0) --(3, 2);

\node at (-1, -.5) {$(sA)^{\otimes p}$};
\node at (0, -.5) {$sM$};
\node at (1, -.5) {$(sA)^{\otimes m}$};
\node at (2, -.5) {$N$};
\node at (3, -.5) {$(sA)^{\otimes q}$};

\draw (3.25, 2) rectangle (2.75, 2.5);
\node at (3, 2.25) {$\widehat{m}$};
\draw (3, 2.5) -- (3, 4);
\draw (-1.25, 4) rectangle (3.25, 4.5);
\node at (1, 4.25) {$\psi$};
}\hspace{1 mm} + \hspace{1 mm}\tikzmath{
\draw (-2.5, 0) --(-2.5, 1.2);
\draw (-1, 0) --(-1, 2.5);
\draw (0, 0) --(0, 2.5);
\draw (1, 0) --(1, 2.5);
\draw (2.5, 0) --(2.5, 1.2);
\node at (3.5, -.6) {.};

\node at (-2.5, -.5) {$(sA)^{\otimes p}$};
\node at (-1.3, -.5) {$sM$};
\node at (0, -.5) {$(sA)^{\otimes m}$};
\node at (1, -.5) {$N$};
\node at (2.5, -.5) {$(sA)^{\otimes q}$};

\draw (-3.25, 1.2) rectangle (-1.75, 1.7);
\node at (-2.5, 1.45) {$\Delta$};
\draw (3.25, 1.2) rectangle (1.75, 1.7);
\node at (2.5, 1.45) {$\Delta$};
\draw (-3, 1.7) -- (-3, 4);
\draw (-2, 1.7) -- (-2, 2.5);
\draw (3, 1.7) -- (3, 4);
\draw (2, 1.7) -- (2, 2.5);
\draw (-2.25, 2.5) rectangle (2.25, 3.2);
\node at (0, 2.82) {$\beta_{\iota(M)\overset{\infty}{\otimes}_A \iota(N)}$};
\draw (-3.25, 4) rectangle (3.25, 4.5);
\node at (0, 4.25) {$\psi$};
\draw (., 3.2) -- (., 4);
\draw (-2, 3.2) -- (-2, 4);
\draw (2, 3.2) -- (2, 4);
}
$$
By the definition of $\psi$, the first two summands vanish, and the third summand is equal to
$$\tikzmath{
\draw (-2, 0) --(-2, 2.8);
\draw (-1, 0) --(-1, 1.2);
\draw (0, 0) --(0, 1.2);
\draw (1, 0) --(1, 1.2);
\draw (2, 0) --(2, 1.2);

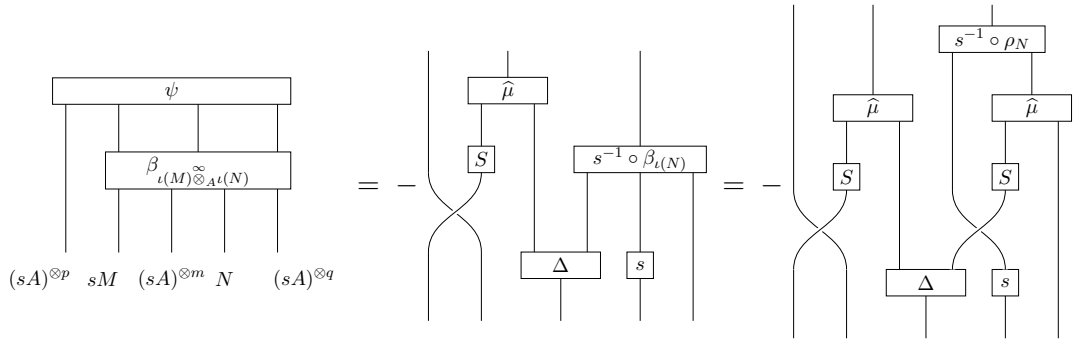
\node at (-2.5, -.5) {$(sA)^{\otimes p}$};
\node at (-1.3, -.5) {$sM$};
\node at (0, -.5) {$(sA)^{\otimes m}$};
\node at (1, -.5) {$N$};
\node at (2.5, -.5) {$(sA)^{\otimes q}$};
\draw (-1.25, 1.2) rectangle (2.25, 1.9);
\node at (0.5, 1.52) {$\beta_{\iota(M)\overset{\infty}{\otimes}_A \iota(N)}$};
\draw (-2.25, 2.8) rectangle (2.25, 3.3);
\node at (-., 3.05) {$\psi$};
\draw (-1, 1.9) -- (-1, 2.8);
\draw (.5, 1.9) -- (.5, 2.8);
\draw (2, 1.9) -- (2, 2.8);
}\hspace{1 mm} = \hspace{1 mm}- \hspace{1 mm}\tikzmath{
\draw (-2, 0.4) --(-2, 1.7);
\draw (-1, 0.4) --(-1, 1.7);
\draw (0.5, 0.4) --(0.5, 1.2);
\draw (2, 0.4) --(2, 1.2);
\draw (3, 0.4) --(3, 3.2);

\draw (-.25, 1.2) rectangle (1.25, 1.7);
\node at (.5, 1.45) {$\Delta$};
\draw (1.75, 1.2) rectangle (2.25, 1.7);
\node at (2, 1.45) {$s$};
\draw [knot] (-1, 1.7) to[out=90,in=-90](-2, 3.2);
\draw [knot] (-2, 1.7) to[out=90,in=-90](-1, 3.2);
\draw (-.75, 3.2) rectangle (-1.25, 3.7);
\node at (-1, 3.45) {$S$};
\draw (2, 3.2) --(2, 1.7);
\draw (1, 3.2) --(1, 1.7);
\draw (0, 1.7) --(0, 4.5);
\draw (3.25, 3.2) rectangle (.75, 3.7);
\node at (2, 3.45) {$s^{-1}\circ\beta_{\iota(N)}$};
\draw (-1.25, 4.5) rectangle (.25, 5);
\node at (-0.5, 4.75) {$\widehat{\mu}$};
\draw (-1, 3.7) --(-1, 4.5);
\draw (-2, 3.2) --(-2, 5.5);
\draw (-.5, 5) --(-.5, 5.5);
\draw (2, 3.7) --(2, 5.5);
}
\hspace{1 mm} = \hspace{1 mm}- \hspace{1 mm}\tikzmath{
\draw (-2, 0.4) --(-2, 1.7);
\draw (-1, 0.4) --(-1, 1.7);
\draw (0.5, 0.4) --(0.5, 1.2);
\draw (2, 0.4) --(2, 1.2);
\draw (3, 0.4) --(3, 4.5);
\node at (3.5, .) {.};

\draw (-.25, 1.2) rectangle (1.25, 1.7);
\node at (.5, 1.45) {$\Delta$};
\draw (1.75, 1.2) rectangle (2.25, 1.7);
\node at (2, 1.45) {$s$};
\draw [knot] (-1, 1.7) to[out=90,in=-90](-2, 3.2);
\draw [knot] (-2, 1.7) to[out=90,in=-90](-1, 3.2);
\draw [knot] (2, 1.7) to[out=90,in=-90](1, 3.2);
\draw [knot] (1, 1.7) to[out=90,in=-90](2, 3.2);
\draw (1.75, 3.2) rectangle (2.25, 3.7);
\node at (2, 3.45) {$S$};
\draw (-.75, 3.2) rectangle (-1.25, 3.7);
\node at (-1, 3.45) {$S$};
\draw (2, 3.7) --(2, 4.5);
\draw (3.25, 4.5) rectangle (1.75, 5);
\node at (2.5, 4.75) {$\widehat{\mu}$};
\draw (-1, 3.7) --(-1, 4.5);
\draw (0, 1.7) --(0, 4.5);
\draw (-1.25, 4.5) rectangle (.25, 5);
\node at (-0.5, 4.75) {$\widehat{\mu}$};

\draw (-2, 3.2) --(-2, 6.7);
\draw (-.5, 5) --(-.5, 6.7);
\draw (1, 3.2) --(1, 5.8);
\draw (2.5, 5) --(2.5, 5.8);
\draw (2.75, 5.8) rectangle (.75, 6.3);
\node at (1.75, 6.05) {$s^{-1}\circ\rho_N$};
\draw (1.75, 6.3) -- (1.75, 6.7);
}
$$
Similarly, the right-hand side of \eqref{align:Hbimodule} becomes
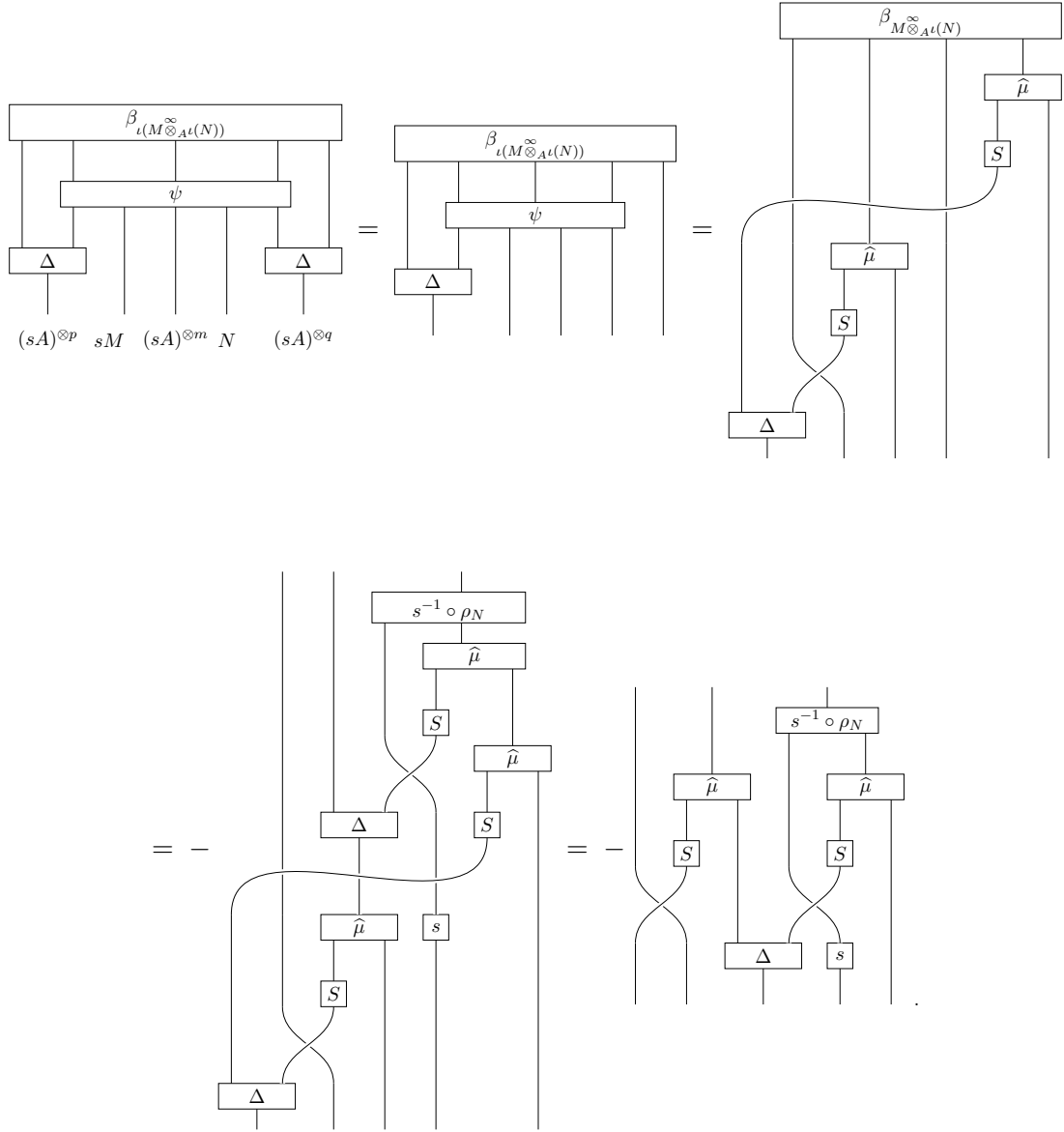
$$\tikzmath{
\draw (-2.5, 0.4) --(-2.5, 1.2);
\draw (-1, 0.4) --(-1, 2.5);
\draw (0, 0.4) --(0, 2.5);
\draw (1, 0.4) --(1, 2.5);
\draw (2.5, 0.4) --(2.5, 1.2);

\node at (-2.5, -.1) {$(sA)^{\otimes p}$};
\node at (-1.3, -.1) {$sM$};
\node at (0, -.1) {$(sA)^{\otimes m}$};
\node at (1, -.1) {$N$};
\node at (2.5, -.1) {$(sA)^{\otimes q}$};

\draw (-3.25, 1.2) rectangle (-1.75, 1.7);
\node at (-2.5, 1.45) {$\Delta$};
\draw (3.25, 1.2) rectangle (1.75, 1.7);
\node at (2.5, 1.45) {$\Delta$};
\draw (-3, 1.7) -- (-3, 3.8);
\draw (-2, 1.7) -- (-2, 2.5);
\draw (2, 1.7) -- (2, 2.5);
\draw (3, 1.7) -- (3, 3.8);
\draw (2.25, 2.5) rectangle (-2.25, 3);
\node at (0, 2.75) {$\psi$};
\draw (-2, 3) -- (-2, 3.8);
\draw (0, 3) -- (0, 3.8);
\draw (2, 3) -- (2, 3.8);
\draw (-3.25, 4.5) rectangle (3.25, 3.8);
\node at (0, 4.12) {$\beta_{\iota(M\overset{\infty}{\otimes}_A \iota(N))}$};
}\hspace{1 mm} = \hspace{1 mm}\tikzmath{
\draw (-2.5, 0.4) --(-2.5, 1.2);
\draw (-1, 0.4) --(-1, 2.5);
\draw (0, 0.4) --(0, 2.5);
\draw (1, 0.4) --(1, 2.5);
\draw (2, 0.4) --(2, 3.8);

\draw (-3.25, 1.2) rectangle (-1.75, 1.7);
\node at (-2.5, 1.45) {$\Delta$};
\draw (-3, 1.7) -- (-3, 3.8);
\draw (-2, 1.7) -- (-2, 2.5);
\draw (1.25, 2.5) rectangle (-2.25, 3);
\node at (-.5, 2.75) {$\psi$};
\draw (-2, 3) -- (-2, 3.8);
\draw (-0.5, 3) -- (-.5, 3.8);
\draw (1, 3) -- (1, 3.8);
\draw (-3.25, 4.5) rectangle (2.25, 3.8);
\node at (-.5, 4.12) {$\beta_{\iota(M\overset{\infty}{\otimes}_A \iota(N))}$};
}\hspace{1 mm} = \hspace{1 mm}\tikzmath{
\draw (-1.5, 0.8) --(-1.5, 1.2);
\draw (0, 0.8) --(0, 1.7);
\draw (1, 0.8) --(1, 4.5);
\draw (2, 0.8) --(2, 5);
\draw (4, 0.8) --(4, 7.8);

\draw (-2.25, 1.2) rectangle (-.75, 1.7);
\node at (-1.5, 1.45) {$\Delta$};
\draw [knot] (0, 1.7) to[out=90,in=-90](-1, 3.2);
\draw [knot] (-1, 1.7) to[out=90,in=-90](0, 3.2);
\draw (-2, 1.7) --(-2, 5);
\draw (-1, 3.2) --(-1, 5);
\draw (-.25, 3.2) rectangle (.25, 3.7);
\node at (0, 3.45) {$S$};
\draw (0, 3.7) --(0, 4.5);
\draw (-.25, 4.5) rectangle (1.25, 5);
\node at (0.5, 4.75) {$\widehat{\mu}$};
\draw [knot] (-1, 5) to[out=90,in=-90](-1, 9);
\draw [knot] (0.5, 5) to[out=90,in=-90](.5, 9);
\draw [knot] (2, 5) to[out=90,in=-90](2, 9);
\draw [knot] (-2, 5) to[out=90,in=-90](3, 6.5);

\draw (3.25, 6.5) rectangle (2.75, 7);
\node at (3, 6.75) {$S$};
\draw (3, 7) --(3, 7.8);
\draw (2.75, 7.8) rectangle (4.25, 8.3);
\node at (3.5, 8.05) {$\widehat{\mu}$};
\draw (3.5, 9) --(3.5, 8.3);
\draw (-1.25, 9) rectangle (4.25, 9.7);
\node at (1.5, 9.32) {$\beta_{M\overset{\infty}{\otimes}_A \iota(N)}$};
}$$
$$ \hspace{1 mm} = \hspace{1 mm} - \hspace{1 mm}\tikzmath{
\draw (-1.5, 0.8) --(-1.5, 1.2);
\draw (0, 0.8) --(0, 1.7);
\draw (1, 0.8) --(1, 4.5);
\draw (2, 0.8) --(2, 4.5);
\draw (4, 0.8) --(4, 7.8);

\draw (-2.25, 1.2) rectangle (-.75, 1.7);
\node at (-1.5, 1.45) {$\Delta$};
\draw [knot] (0, 1.7) to[out=90,in=-90](-1, 3.2);
\draw [knot] (-1, 1.7) to[out=90,in=-90](0, 3.2);
\draw (-2, 1.7) --(-2, 5);
\draw (-1, 3.2) --(-1, 5);
\draw (-.25, 3.2) rectangle (.25, 3.7);
\node at (0, 3.45) {$S$};
\draw (0, 3.7) --(0, 4.5);
\draw (-.25, 4.5) rectangle (1.25, 5);
\node at (0.5, 4.75) {$\widehat{\mu}$};
\draw (2.25, 4.5) rectangle (1.75, 5);
\node at (2, 4.75) {$s$};
\draw [knot] (-1, 5) to[out=90,in=-90](-1, 11.7);
\draw [knot] (0.5, 5) to[out=90,in=-90](.5, 6.5);
\draw [knot] (2, 5) to[out=90,in=-90](2, 7);
\draw [knot] (-2, 5) to[out=90,in=-90](3, 6.5);

\draw (3.25, 6.5) rectangle (2.75, 7);
\node at (3, 6.75) {$S$};
\draw (3, 7) --(3, 7.8);
\draw (2.75, 7.8) rectangle (4.25, 8.3);
\node at (3.5, 8.05) {$\widehat{\mu}$};
\draw (-.25, 6.5) rectangle (1.25, 7);
\node at (.5, 6.75) {$\Delta$};
\draw (0, 7) --(0, 11.7);
\draw [knot] (2, 7) to[out=90,in=-90](1, 8.5);
\draw [knot] (1, 7) to[out=90,in=-90](2, 8.5);
\draw (2.25, 8.5) rectangle (1.75, 9);
\node at (2, 8.75) {$S$};
\draw (3.5, 8.3) --(3.5, 9.8);
\draw (2, 9) --(2, 9.8);
\draw (1.75, 9.8) rectangle (3.75, 10.3);
\node at (2.75, 10.05) {$\widehat{\mu}$};
\draw (1, 8.5) --(1, 10.7);
\draw (2.5, 10.3) --(2.5, 10.7);
\draw (3.75, 10.7) rectangle (.75, 11.3);
\node at (2.25, 10.95) {$s^{-1}\circ\rho_{N}$};
\draw (2.5, 11.3) --(2.5, 11.7);
}\hspace{1 mm} = \hspace{1 mm}- \hspace{1 mm}\tikzmath{
\draw (-2, 0.5) --(-2, 1.7);
\draw (-1, 0.5) --(-1, 1.7);
\draw (0.5, 0.5) --(0.5, 1.2);
\draw (2, 0.5) --(2, 1.2);
\draw (3, 0.5) --(3, 4.5);
\node at (3.5, 0.5) {.};

\draw (-.25, 1.2) rectangle (1.25, 1.7);
\node at (.5, 1.45) {$\Delta$};
\draw (1.75, 1.2) rectangle (2.25, 1.7);
\node at (2, 1.45) {$s$};
\draw [knot] (-1, 1.7) to[out=90,in=-90](-2, 3.2);
\draw [knot] (-2, 1.7) to[out=90,in=-90](-1, 3.2);
\draw [knot] (2, 1.7) to[out=90,in=-90](1, 3.2);
\draw [knot] (1, 1.7) to[out=90,in=-90](2, 3.2);
\draw (1.75, 3.2) rectangle (2.25, 3.7);
\node at (2, 3.45) {$S$};
\draw (-.75, 3.2) rectangle (-1.25, 3.7);
\node at (-1, 3.45) {$S$};
\draw (2, 3.7) --(2, 4.5);
\draw (3.25, 4.5) rectangle (1.75, 5);
\node at (2.5, 4.75) {$\widehat{\mu}$};
\draw (-1, 3.7) --(-1, 4.5);
\draw (0, 1.7) --(0, 4.5);
\draw (-1.25, 4.5) rectangle (.25, 5);
\node at (-0.5, 4.75) {$\widehat{\mu}$};

\draw (-2, 3.2) --(-2, 6.7);
\draw (-.5, 5) --(-.5, 6.7);
\draw (1, 3.2) --(1, 5.8);
\draw (2.5, 5) --(2.5, 5.8);
\draw (2.75, 5.8) rectangle (.75, 6.3);
\node at (1.75, 6.05) {$s^{-1}\circ\rho_N$};
\draw (1.75, 6.3) -- (1.75, 6.7);
}
$$
This is the same expression obtained from the left-hand side. Hence \eqref{align:Hbimodule} holds also for $q>0$, and $\psi$ is an $A_\infty$-bimodule quasi-isomorphism.
\end{proof}

\begin{rem}
The same computation shows that the $A_\infty$-bimodule morphism
$$\psi\colon \iota(M)\overset{\infty}{\otimes}_A \iota(N) \to \iota(M\overset{\infty}{\otimes}_A \iota(N))$$
is functorial. More precisely, for any $f\in \Hom(T^c(sA) \otimes sM , sM')$ and $g\in \Hom(T^c(sA) \otimes sN , sN')$, the following diagram commutes in $\EC^{\RR}_\infty(A)$; compare Remark \ref{rem:tensormorphism}. Here composition is as in formula \eqref{align:compositionmorphisms}:
$$
\xymatrix{
\iota(M)\overset{\infty}{\otimes}_A \iota(N)\ar[r]^{\psi} \ar[d]^{\iota(f)\overset{\infty}{\otimes}_A \iota(g)} & \iota(M\overset{\infty}{\otimes}_A \iota(N)) \ar[d]^{\iota(f\overset{\infty}{\otimes}_A \iota(g))}\\
\iota(M')\overset{\infty}{\otimes}_A \iota(N') \ar[r]^{\psi} & \iota(M'\overset{\infty}{\otimes}_A \iota(N')) .
}$$
\end{rem}

We now prove associativity of the tensor product.
\begin{prop}\label{prop:asso-rightmodule}
Let $(A,m^n;\mu^{p,q})$ be a $B_\infty$-algebra. Then the tensor product $\boxtimes_A$ is associative. More precisely, for right $A_\infty$-modules $M$, $N$, and $K$, there is a natural isomorphism in $\ED^{\RR}_\infty(A)$
$$a_{M,N,K}\colon (M \boxtimes_A N)\boxtimes_A K \rightarrow M \boxtimes_A (N\boxtimes_A K),$$
and these isomorphisms satisfy the pentagon axiom: for all right $A_\infty$-modules $M$, $N$, $K$, and $W$, the diagram
\[
\xymatrix{
& ((M\boxtimes_A N)\boxtimes_A K)\boxtimes_A W \ar[dl]^{\; a_{M,N,K}\boxtimes_A \id_W}\ar[rd]_{a_{M\boxtimes_A N, K, W}\quad} & \\
(M\boxtimes_A (N \boxtimes_A K))\ar[d]^{a_{M, N\boxtimes_A K, W}} \boxtimes_A W & & (M\boxtimes_A N) \boxtimes_A (K \boxtimes_A W)\ar[d]^{a_{M, N, K\boxtimes_A W}} \\
M\boxtimes_A ((N \boxtimes_A K) \boxtimes_A W)\ar[rr]^{\id_M\boxtimes_A a_{N,K,W}} & & M\boxtimes_A (N \boxtimes_A( K \boxtimes_A W))
 }
\]
is commutative.
\end{prop}
\begin{proof}
Define $a_{M,N,K}$ to be the composite \[
 \xymatrix@C=4pc{
 (M\overset{\infty}{\otimes}_A \iota(N)) \overset{\infty}{\otimes}_A \iota(K) \ar[r]^-{\widetilde a} & M\overset{\infty}{\otimes}_A (\iota(N) \overset{\infty}{\otimes}_A \iota(K)) \ar[r]^{\id_M \overset{\infty}{\otimes}_A \psi} & M \overset{\infty}{\otimes}_A \iota(N\overset{\infty}{\otimes}_A \iota(K)),
 }
\] where $\widetilde a$ is the natural associativity map for $A_\infty$-bimodules from \cite[Corollary 1]{Fer12}, and $\psi$ is the quasi-isomorphism of Lemma \ref{lem:isomorphismmn}.
Since $\widetilde a$ and $\psi$ are natural, the naturality of $a$ follows immediately.

It remains to show that $a_{M,N,K}$ induces an isomorphism in $\ED^{\RR}_\infty(A)$. For this it is enough to show that $\id_M\overset{\infty}{\otimes}_A\psi$ is a homotopy equivalence. The morphism $\psi$ is a quasi-isomorphism, and a quasi-isomorphism in $\EC^{\BB}_\infty(A)$ is a homotopy equivalence; see, for example, \cite[Theorem 4.2.1]{Kel01}. Thus there is a morphism $\psi'$ such that
\[
\psi \overset{\infty}{\circ} \psi'\simeq \id,
\qquad
\psi' \overset{\infty}{\circ} \psi \simeq \id.
\]
By the functoriality \eqref{align:tensorcomposition} of $\overset{\infty}{\otimes}_A$, we obtain
\[
(\id_M \overset{\infty}{\otimes}_A \psi ) \overset{\infty}{\circ} (\id_M \overset{\infty}{\otimes}_A \psi') \simeq \id,
\qquad
(\id_M \overset{\infty}{\otimes}_A \psi') \overset{\infty}{\circ} (\id_M \overset{\infty}{\otimes}_A \psi ) \simeq \id.
\]
Therefore $\id_M\overset{\infty}{\otimes}_A\psi$ is a homotopy equivalence.

Since $\iota(a_{N,K,W})=a_{N,K,W}$, expanding the definition of $a$ reduces the pentagon axiom to the commutativity of the following diagram.
\[
\xymatrix{
((M\overset{\infty}{\otimes}_A \iota(N)) \overset{\infty}{\otimes}_A \iota(K))\overset{\infty}{\otimes}_A \iota(W)\ar[d]^{\; \widetilde a\overset{\infty}{\otimes}_A \id}\ar[rr]_{\widetilde a} & & (M\overset{\infty}{\otimes}_A \iota(N))\overset{\infty}{\otimes}_A ((\iota(K)\overset{\infty}{\otimes}_A \iota(W))\ar[d]^{(\id\overset{\infty}{\otimes}_A \id)\overset{\infty}{\otimes}_A \psi} \\
(M\overset{\infty}{\otimes}_A ( \iota(N)\overset{\infty}{\otimes}_A \iota(K))) \overset{\infty}{\otimes}_A \iota(W)\ar[d]^{(\id\overset{\infty}{\otimes}_A \psi)\overset{\infty}{\otimes}_A\id} & & (M\overset{\infty}{\otimes}_A \iota(N))\overset{\infty}{\otimes}_A \iota(K\overset{\infty}{\otimes}_A \iota(W))\ar[d]^{\widetilde{a}} \\
(M\overset{\infty}{\otimes}_A \iota(N\overset{\infty}{\otimes}_A \iota(K)))\overset{\infty}{\otimes}_A \iota(W)\ar[d]^{\widetilde{a}} & & M\overset{\infty}{\otimes}_A (\iota(N)\overset{\infty}{\otimes}_A \iota(K\overset{\infty}{\otimes}_A \iota(W))) \ar[d]^{\id\overset{\infty}{\otimes}_A \psi} \\
M\overset{\infty}{\otimes}_A (\iota(N\overset{\infty}{\otimes}_A \iota(K))\overset{\infty}{\otimes}_A \iota(W))\ar[d]^{\id\overset{\infty}{\otimes}_A \psi} & &M\overset{\infty}{\otimes}_A \iota(N\overset{\infty}{\otimes}_A \iota(K\overset{\infty}{\otimes}_A \iota(W))) \\
M\overset{\infty}{\otimes}_A \iota((N\overset{\infty}{\otimes}_A \iota(K))\overset{\infty}{\otimes}_A \iota(W))\ar[rr]^{\id\overset{\infty}{\otimes}_A \widetilde{a}} & & M\overset{\infty}{\otimes}_A \iota(N\overset{\infty}{\otimes}_A (\iota(K)\overset{\infty}{\otimes}_A \iota(W)))\ar[u]_{\id\overset{\infty}{\otimes}_A (\id \overset{\infty}{\otimes}_A \psi)}
 }
\]
It suffices to verify the commutativity when evaluated on $((sM\otimes T^c(sA)\otimes \iota(N))\otimes T^c(sA)\otimes \iota(K))\otimes T^c(sA)\otimes \iota(W)$. The required equality is exactly the following pentagon identity for $\psi$:
\begin{equation}\label{equation:associativity}
\tikzmath{
\draw (-1.5, 0) --(-1.5, .3);
\draw (0, 0) --(0, 2.8);
\draw (1, 0) --(1, 7);
\node at (-1.5, -.3) {$T^c(sA)$};
\node at (-.1, -.3) {$T^c(sA)$};
\node at (1.25, -.3) {$T^c(sA)$};

\draw (-.75, .3) rectangle (-2.25, .8);
\node at (-1.5, .55) {$\Delta$};
\draw (-1, .8) --(-1, 1.6);
\draw (-1.25, 1.6) rectangle (-.75, 2.1);
\node at (-1, 1.85) {$S$};
\draw (-1, 2.1) --(-1, 2.8);
\draw (-1.25, 2.8) rectangle (.25, 3.3);
\node at (-.5, 3.05) {$\widehat{\mu}$};
\draw (-.5, 3.3) --(-.5, 4.3);
\draw (-2, .8) --(-2, 3.8);
\draw (-2.75, 3.8) rectangle (-1.25, 4.3);
\node at (-2, 4.05) {$\Delta$};
\draw (-2.6, 4.3) --(-2.6, 11.6);
\draw [knot] (-.5, 4.3) to[out=90,in=-90](-1.5, 5.8);
\draw [knot] (-1.5, 4.3) to[out=90,in=-90](-.5, 5.8);
\draw (-.25, 5.8) rectangle (-.75, 6.3);
\node at (-.5, 6.05) {$S$};
\draw (-1.5, 5.8) --(-1.5, 8.3);
\draw (-.5, 6.3) --(-.5, 7);
\draw (1.25, 7) rectangle (-.75, 7.5);
\node at (.25, 7.25) {$\widehat{\mu}$};
\draw (-1.6, 8.3) rectangle (-.1, 8.8);
\node at (-.85, 8.55) {$\Delta$};
\draw (-1.5, 8.8) -- (-1.5, 11.6);
\draw (-.35, 8.8) -- (-.35, 9.6);
\draw (-.6, 9.6) rectangle (-.1, 10.1);
\node at (-.35, 9.85) {$S$};
\draw (-.35, 10.1) --(-.35, 10.8);
\draw (-.6, 10.8) rectangle (.9, 11.3);
\node at (.15, 11.05) {$\widehat{\mu}$};
\draw (.65, 7.5) -- (.65, 10.8);
\draw (.15, 11.3) -- (.15, 11.6);
}\hspace{1 mm} = \hspace{1 mm} \tikzmath{
\draw (-2.5, 0) --(-2.5, .3);
\draw (0, 0) --(0, 3.8);
\draw (1, 0) --(1, 7);

\draw (-1.75, .3) rectangle (-3.25, .8);
\node at (-2.5, .55) {$\Delta$};
\draw (-2, .8) --(-2, 1.3);
\draw (-3, .8) --(-3, 11.6);
\draw (-.75, 1.3) rectangle (-2.25, 1.8);
\node at (-1.5, 1.55) {$\Delta$};
\draw (-2, 1.8) --(-2, 4.3);
\draw (-1, -.3) --(-1, 1.3);
\draw (-1, 1.8) --(-1, 2.6);
\draw (-1.25, 2.6) rectangle (-.75, 3.1);
\node at (-1, 2.85) {$S$};
\draw (-1, 3.1) --(-1, 3.8);
\draw (-1.25, 3.8) rectangle (.25, 4.3);
\node at (-.5, 4.05) {$\widehat{\mu}$};
\draw [knot] (-.5, 4.3) to[out=90,in=-90](-1.75, 5.8);
\draw [knot] (-2, 4.3) to[out=90,in=-90](-.5, 5.8);
\draw (-.25, 5.8) rectangle (-.75, 6.3);
\node at (-.5, 6.05) {$S$};
\draw (-.5, 6.3) --(-.5, 7);
\draw (1.25, 7) rectangle (-.75, 7.5);
\node at (.25, 7.25) {$\widehat{\mu}$};
\draw (-2., 8.3) rectangle (-.25, 8.8);
\node at (-1.125, 8.55) {$\Delta$};
\draw (-1.75, 5.8) -- (-1.75, 8.3);
\draw (-1.75, 8.8) -- (-1.75, 11.6);
\draw (-.5, 8.8) -- (-.5, 9.6);
\draw (-.75, 9.6) rectangle (-.25, 10.1);
\node at (-.5, 9.85) {$S$};
\draw (-.5, 10.1) --(-.5, 10.8);
\draw (-.75, 10.8) rectangle (1.25, 11.3);
\node at (.25, 11.05) {$\widehat{\mu}$};
\draw (1, 7.5) -- (1, 10.8);
\draw (.25, 11.3) -- (.25, 11.6);
}\hspace{1 mm} = \hspace{1 mm}\tikzmath{
\draw (-2, 0) --(-2, .3);
\draw (0, 0) --(0, .3);
\draw (1.5, 0) --(1.5, 2.8);
\node at (2, 0) {,};

\draw (-.75, .3) rectangle (.75, .8);
\node at (0, .55) {$\Delta$};
\draw (.5, .8) --(.5, 1.6);
\draw (-1.5, .8) --(-1.5, 1.6);
\draw (-.5, .8) --(-.5, 2.8);
\draw (.75, 1.6) rectangle (.25, 2.1);
\node at (0.5, 1.85) {$S$};
\draw (-2.75, .3) rectangle (-1.25, .8);
\node at (-2, .55) {$\Delta$};
\draw (-1.25, 1.6) rectangle (-1.75, 2.1);
\node at (-1.5, 1.85) {$S$};
\draw (.5, 2.1) --(.5, 2.8);
\draw (-1.5, 2.1) --(-1.5, 2.8);
\draw (-2.5, .8) --(-2.5, 3.6);
\draw (.25, 2.8) rectangle (1.75, 3.3);
\node at (1, 3.05) {$\widehat{\mu}$};
\draw (-.25, 2.8) rectangle (-1.75, 3.3);
\node at (-1, 3.05) {$\widehat{\mu}$};
\draw (-1, 3.3) --(-1, 3.6);
\draw (1, 3.3) --(1, 3.6);
}
\end{equation}
Here the first equality follows from coassociativity of $\Delta$, and the second follows from Lemma \ref{lem:cofreehopfidentity}.
The leftmost term in \eqref{equation:associativity} is precisely the composite along the lower-left path of the diagram, involving three copies of $\psi$. The rightmost term is the composite along the upper-right path, involving two copies of $\psi$. The elements of $M$, $N$, and $K$ do not affect this comparison, since the relevant maps act as the identity on them.
\end{proof}

We next verify the triangle axiom for $\boxtimes_A$. Recall the left and right unit isomorphisms
\[
l_X\colon A\overset{\infty}{\otimes}_A X\rightarrow X,
\qquad
r_X\colon X\overset{\infty}{\otimes}_A A\rightarrow X
\]
for $X \in \ED^{\BB}_\infty(A)$; see \cite[Proposition 2]{Fer12}.

\begin{lem}\label{lemm:Hwithunit}
Let $(A,m^n;\mu^{p,q})$ be a $B_{\infty}$-algebra. Then the following diagram commutes in the derived category $\ED^{\BB}_\infty(A)$ of $A_\infty$-bimodules, with composition as in \eqref{align:compositionmorphisms}:
$$
\xymatrix{
\iota(A)\overset{\infty}{\otimes}_A \iota(N)\ar[r]^-{\psi} \ar[d]^-{I_A\overset{\infty}{\otimes}_A\id} & \iota(A\overset{\infty}{\otimes}_A \iota(N)) \ar[d]^-{\iota(l_{\iota(N)})}\\
A\overset{\infty}{\otimes}_A \iota(N) \ar[r]^-{l_{\iota(N)}} & \iota(N)
}$$
where $I_A\colon \iota(A)\rightarrow A$ is the $A_\infty$-bimodule isomorphism defined in Theorem \ref{thm:iotabracemodule}, whose components are
 \[I^{p,q}_A(sa_{1,p}\otimes sx\otimes sb_{1,q})=\sum_{k=0}^p (-1)^{|sx||sa_{k+1,p}|}\mu(sa_{1,k}\bigotimes (sx\otimes \widehat{\mu}(S(sa_{k+1,p})\bigotimes sb_{1,q}))).\]
\end{lem}

\begin{proof}
It is enough to evaluate the two composites on $(sA)^{\otimes p}\otimes (sA\otimes (sA)^{\otimes m}\otimes N)\otimes (sA)^{\otimes q}$.
For $l_{\iota(N)}\overset{\infty}{\circ} (I_A\overset{\infty}{\otimes}_A \id)= l_{\iota(N)}\circ (\widehat{I_A\overset{\infty}{\otimes}_A \id})$, we have

$$\tikzmath{
\draw (-3, 1.7) --(-3, 7.7);
\draw (-2, .8) --(-2, 1.2);
\draw (0, .8) --(0, 1.7);
\draw (1.5, .8) --(1.5, 1.2);
\draw (3, .8) -- (3, 7.7);
\draw (4, .8) -- (4, 7.7);
\node at (-1.5, .3) {$(sA)^{\otimes p}$};
\node at (-.2, .3) {$sA$};
\node at (1.5, .3) {$(sA)^{\otimes m}$};
\node at (3, .3) {$N$};
\node at (4, .3) {$(sA)^{\otimes q}$};

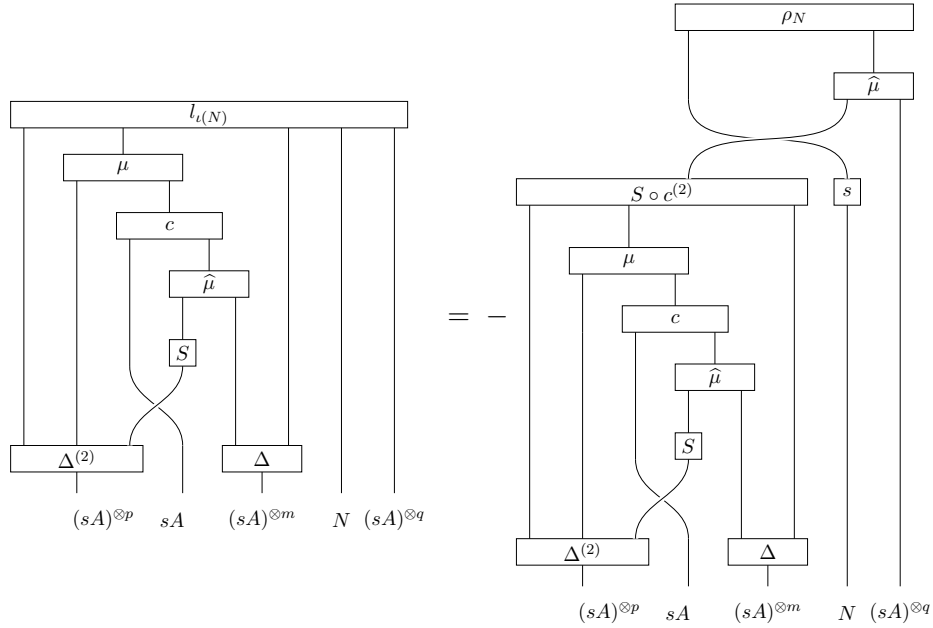
\draw (-3.25, 1.2) rectangle (-.75, 1.7);
\node at (-2, 1.45) {$\Delta^{(2)}$};
\draw (2.25, 1.2) rectangle (.75, 1.7);
\node at (1.5, 1.45) {$\Delta$};
\draw (1, 4.5) -- (1, 1.7);
\draw (-2, 5.6) -- (-2, 1.7);
\draw [knot] (0, 1.7) to[out=90,in=-90](-1, 3.2);
\draw [knot] (-1, 1.7) to[out=90,in=-90](0, 3.2);
\draw (-1, 5.6) -- (-1, 3.2);
\draw (0, 4.5) -- (0, 3.7);
\draw (-.25, 3.2) rectangle (.25, 3.7);
\node at (0, 3.45) {$S$};
\draw (1.25, 4.5) rectangle (-.25, 5);
\node at (.5, 4.75) {$\widehat{\mu}$};

\draw (-1.25, 5.6) rectangle (.75, 6.1);
\node at (-.25, 5.85) {$c$};
\draw (.5, 5.6) -- (.5, 5);
\draw (-.25, 6.1) -- (-.25, 6.7);
\draw (-2, 5.6) -- (-2, 6.7);
\draw (-2.25, 6.7) rectangle (0, 7.2);
\node at (-1.125, 6.95) {$\mu$};
\draw (2, 1.7) -- (2, 7.7);
\draw (-1.125, 7.2) -- (-1.125, 7.7);
\draw (-3.25, 7.7) rectangle (4.25, 8.2);
\node at (.5, 7.95) {$l_{\iota(N)}$};
}\hspace{1 mm} = \hspace{1 mm} - \hspace{1 mm}\tikzmath{
\draw (-3, 1.7) --(-3, 8);
\draw (-2, .8) --(-2, 1.2);
\draw (0, .8) --(0, 1.7);
\draw (1.5, .8) --(1.5, 1.2);
\draw (3, .8) -- (3, 8);
\draw (4, .8) -- (4, 10);
\node at (-1.5, .3) {$(sA)^{\otimes p}$};
\node at (-.2, .3) {$sA$};
\node at (1.5, .3) {$(sA)^{\otimes m}$};
\node at (3, .3) {$N$};
\node at (4, .3) {$(sA)^{\otimes q}$};

\draw (-3.25, 1.2) rectangle (-.75, 1.7);
\node at (-2, 1.45) {$\Delta^{(2)}$};
\draw (2.25, 1.2) rectangle (.75, 1.7);
\node at (1.5, 1.45) {$\Delta$};
\draw (1, 4.5) -- (1, 1.7);
\draw (-2, 5.6) -- (-2, 1.7);
\draw [knot] (0, 1.7) to[out=90,in=-90](-1, 3.2);
\draw [knot] (-1, 1.7) to[out=90,in=-90](0, 3.2);
\draw (-1, 5.6) -- (-1, 3.2);
\draw (0, 4.5) -- (0, 3.7);
\draw (-.25, 3.2) rectangle (.25, 3.7);
\node at (0, 3.45) {$S$};
\draw (1.25, 4.5) rectangle (-.25, 5);
\node at (.5, 4.75) {$\widehat{\mu}$};

\draw (-1.25, 5.6) rectangle (.75, 6.1);
\node at (-.25, 5.85) {$c$};
\draw (.5, 5.6) -- (.5, 5);
\draw (-.25, 6.1) -- (-.25, 6.7);
\draw (-2, 5.6) -- (-2, 6.7);
\draw (-2.25, 6.7) rectangle (0, 7.2);
\node at (-1.125, 6.95) {$\mu$};
\draw (2, 1.7) -- (2, 8);
\draw (-1.125, 7.2) -- (-1.125, 8);
\draw (-3.25, 8) rectangle (2.25, 8.5);
\node at (-.5, 8.25) {$S\circ c^{(2)}$};
\draw (3.25, 8.5) rectangle (2.75, 8);
\node at (3, 8.25) {$s$};
\draw [knot] (3, 8.5) to[out=90,in=-90](0, 10);
\draw [knot] (0, 8.5) to[out=90,in=-90](3, 10);
\draw (0, 10) -- (0, 11.3);
\draw (4.25, 10) rectangle (2.75, 10.5);
\node at (3.5, 10.25) {$\widehat{\mu}$};
\draw (3.5, 10.5) -- (3.5, 11.3);
\draw (-.25, 11.3) rectangle (4.25, 11.8);
\node at (2, 11.55) {$\rho_N$};
}$$
where the equality uses the explicit formula of $l_{\iota(N)}$ from \cite[Proposition 2]{Fer12}, and
$c\colon T^c(sA) \otimes T^c(sA) \to T^c(sA)$ denotes the concatenation map; see the notation before Proposition \ref{prop:concatenationidentity}. For $\iota(l_{\iota(N)})\overset{\infty}{\circ} \psi= \iota(l_{\iota(N)})\circ\widehat{\psi}$
$$\tikzmath{
\draw (-1.5, 0) --(-1.5, 1.2);
\draw (0, 0) --(0, 1.7);
\draw (1, 0) --(1, 4.5);
\draw (2, 0) --(2, 5);
\draw (3, 0) --(3, 6.5);

\node at (-1.5, -.5) {$(sA)^{\otimes p}$};
\node at (-.2, -.5) {$sA$};
\node at (1, -.5) {$(sA)^{\otimes m}$};
\node at (2, -.5) {$N$};
\node at (3, -.5) {$(sA)^{\otimes q}$};

\draw (-2.25, 1.2) rectangle (-.75, 1.7);
\node at (-1.5, 1.45) {$\Delta$};
\draw (-2, 4.5) -- (-2, 1.7);
\draw [knot] (0, 1.7) to[out=90,in=-90](-1, 3.2);
\draw [knot] (-1, 1.7) to[out=90,in=-90](0, 3.2);
\draw (-1, 5) -- (-1, 3.2);
\draw (0, 4.5) -- (0, 3.7);
\draw (-.25, 3.2) rectangle (.25, 3.7);
\node at (0, 3.45) {$S$};
\draw (1.25, 4.5) rectangle (-.25, 5);
\node at (.5, 4.75) {$\widehat{\mu}$};
\draw (-1.75, 4.5) rectangle (-2.25, 5);
\node at (-2, 4.75) {$S$};

\draw [knot] (2, 5) to[out=90,in=-90](1, 6.5);
\draw [knot] (.5, 5) to[out=90,in=-90](-.5, 6.5);
\draw [knot] (-1, 5) to[out=90,in=-90](-2, 6.5);
\draw [knot] (-2, 5) to[out=90,in=-90](2, 6.5);
\draw (1.75, 6.5) rectangle (3.25, 7);
\node at (2.5, 6.75) {$\widehat{\mu}$};
\draw (-2, 6.5) -- (-2, 7.6);
\draw (-.5, 6.5) -- (-.5, 7.6);
\draw (1, 6.5) -- (1, 7.6);
\draw (2.5, 7.6) -- (2.5, 7);
\draw (-2.25, 7.6) rectangle (3.25, 8.2);
\node at (.5, 7.9) {$l_{\iota(N)}$};
}\hspace{1 mm} = \hspace{1 mm} - \hspace{1 mm}\tikzmath{
\draw (-1.5, 0) --(-1.5, 1.2);
\draw (0, 0) --(0, 1.7);
\draw (1, 0) --(1, 4.5);
\draw (2, 0) --(2, 5);
\draw (3, 0) --(3, 6.5);

\node at (-1.5, -.5) {$(sA)^{\otimes p}$};
\node at (-.2, -.5) {$sA$};
\node at (1, -.5) {$(sA)^{\otimes m}$};
\node at (2, -.5) {$N$};
\node at (3, -.5) {$(sA)^{\otimes q}$};

\draw (-2.25, 1.2) rectangle (-.75, 1.7);
\node at (-1.5, 1.45) {$\Delta$};
\draw (-2, 4.5) -- (-2, 1.7);
\draw [knot] (0, 1.7) to[out=90,in=-90](-1, 3.2);
\draw [knot] (-1, 1.7) to[out=90,in=-90](0, 3.2);
\draw (-1, 5) -- (-1, 3.2);
\draw (0, 4.5) -- (0, 3.7);
\draw (-.25, 3.2) rectangle (.25, 3.7);
\node at (0, 3.45) {$S$};
\draw (1.25, 4.5) rectangle (-.25, 5);
\node at (.5, 4.75) {$\widehat{\mu}$};
\draw (-1.75, 4.5) rectangle (-2.25, 5);
\node at (-2, 4.75) {$S$};

\draw [knot] (2, 5) to[out=90,in=-90](1, 6.5);
\draw [knot] (.5, 5) to[out=90,in=-90](-.5, 6.5);
\draw [knot] (-1, 5) to[out=90,in=-90](-2, 6.5);
\draw [knot] (-2, 5) to[out=90,in=-90](2, 6.5);
\draw (1.75, 6.5) rectangle (3.25, 7);
\node at (2.5, 6.75) {$\widehat{\mu}$};
\draw (-2.25, 6.5) rectangle (-.25, 7);
\node at (-1.25, 6.75) {$S\circ c$};
\draw (.75, 6.5) rectangle (1.25, 7);
\node at (1, 6.75) {$s$};
\draw [knot] (3, 7) to[out=90,in=-90](3, 8.5);
\draw [knot] (1, 7) to[out=90,in=-90](-1.25, 8.5);
\draw [knot] (-1.25, 7) to[out=90,in=-90](1, 8.5);

\draw (-1.25, 8.5) -- (-1.25, 9.6);
\draw (.75, 8.5) rectangle (3.25, 9);
\node at (2, 8.75) {$\widehat{\mu}$};
\draw (2, 9) -- (2, 9.6);
\draw (-1.5, 10.1) rectangle (3.25, 9.6);
\node at (.875, 9.85) {$\rho_N$};
}
$$

$$\hspace{1 mm} \overset{\eqref{associativity}}{=} \hspace{1 mm}- \hspace{1 mm}\tikzmath{
\draw (-1.5, 0) --(-1.5, 1.2);
\draw (0, 0) --(0, 1.7);
\draw (1, 0) --(1, 4.5);
\draw (2, 0) --(2, 7.6);
\draw (3, 0) --(3, 9.6);

\node at (-1.5, -.5) {$(sA)^{\otimes p}$};
\node at (-.2, -.5) {$sA$};
\node at (1, -.5) {$(sA)^{\otimes m}$};
\node at (2, -.5) {$N$};
\node at (3, -.5) {$(sA)^{\otimes q}$};

\draw (-2.25, 1.2) rectangle (-.75, 1.7);
\node at (-1.5, 1.45) {$\Delta$};
\draw (-2, 5.6) -- (-2, 1.7);
\draw [knot] (0, 1.7) to[out=90,in=-90](-1, 3.2);
\draw [knot] (-1, 1.7) to[out=90,in=-90](0, 3.2);
\draw (-1, 5.6) -- (-1, 3.2);
\draw (0, 4.5) -- (0, 3.7);
\draw (-.25, 3.2) rectangle (.25, 3.7);
\node at (0, 3.45) {$S$};
\draw (1.25, 4.5) rectangle (-.25, 5);
\node at (.5, 4.75) {$\widehat{\mu}$};

\draw (-1.75, 5.6) rectangle (-2.25, 6.1);
\node at (-2, 5.85) {$S$};
\draw (-1.25, 5.6) rectangle (.75, 6.1);
\node at (-.25, 5.85) {$S\circ c$};
\draw (.5, 5.6) -- (.5, 5);
\draw [knot] (-.25, 6.1) to[out=90,in=-90](-2, 7.6);
\draw [knot] (-2, 6.1) to[out=90,in=-90](-.25, 7.6);
\draw (-2.25, 7.6) rectangle (0, 8.1);
\node at (-1.125, 7.85) {$\widehat{\mu}$};
\draw (1.75, 7.6) rectangle (2.25, 8.1);
\node at (2, 7.85) {$s$};

\draw [knot] (2, 8.1) to[out=90,in=-90](-.25, 9.6);
\draw [knot] (-.25, 8.1) to[out=90,in=-90](2, 9.6);
\draw (1.75, 9.6) rectangle (3.25, 10.1);
\node at (2.5, 9.85) {$\widehat{\mu}$};
\draw (-.25, 9.6) -- (-.25, 10.6);
\draw (2.5, 10.1) -- (2.5, 10.6);
\draw (-.5, 10.6) rectangle (3.25, 11.1);
\node at (1.375, 10.85) {$\rho_N$};
}\hspace{1 mm} \overset{\eqref{S:anti-alg}}{=} \hspace{1 mm} - \hspace{1 mm}\tikzmath{
\draw (-1.5, 0) --(-1.5, 1.2);
\draw (0, 0) --(0, 1.7);
\draw (1, 0) --(1, 4.5);
\draw (2, 0) --(2, 7.8);
\draw (3, 0) --(3, 9.6);
\node at (3.8, -.6) {,};

\node at (-1.5, -.5) {$(sA)^{\otimes p}$};
\node at (-.2, -.5) {$sA$};
\node at (1, -.5) {$(sA)^{\otimes m}$};
\node at (2, -.5) {$N$};
\node at (3, -.5) {$(sA)^{\otimes q}$};

\draw (-2.25, 1.2) rectangle (-.75, 1.7);
\node at (-1.5, 1.45) {$\Delta$};
\draw (-2, 5.6) -- (-2, 1.7);
\draw [knot] (0, 1.7) to[out=90,in=-90](-1, 3.2);
\draw [knot] (-1, 1.7) to[out=90,in=-90](0, 3.2);
\draw (-1, 5.6) -- (-1, 3.2);
\draw (0, 4.5) -- (0, 3.7);
\draw (-.25, 3.2) rectangle (.25, 3.7);
\node at (0, 3.45) {$S$};
\draw (1.25, 4.5) rectangle (-.25, 5);
\node at (.5, 4.75) {$\widehat{\mu}$};

\draw (-1.25, 5.6) rectangle (.75, 6.1);
\node at (-.25, 5.85) {$c $};
\draw (.5, 5.6) -- (.5, 5);
\draw (-.25, 6.1) -- (-.25, 6.7);
\draw (-2, 5.6) -- (-2, 6.7);
\draw (-2.25, 6.7) rectangle (0, 7.2);
\node at (-1.125, 6.95) {$\widehat{\mu}$};
\draw (-1.125, 7.2) -- (-1.125, 7.8);
\draw (-1.375, 7.8) rectangle (-.875, 8.3);
\node at (-1.125, 8.05) {$S$};
\draw (1.75, 7.8) rectangle (2.25, 8.3);
\node at (2, 8.05) {$s$};

\draw [knot] (2, 8.3) to[out=90,in=-90](-1.125, 9.6);
\draw [knot] (-1.125, 8.3) to[out=90,in=-90](2, 9.6);
\draw (1.75, 9.6) rectangle (3.25, 10.1);
\node at (2.5, 9.85) {$\widehat{\mu}$};
\draw (-1.125, 9.6) -- (-1.125, 10.6);
\draw (2.5, 10.1) -- (2.5, 10.6);
\draw (-1.375, 10.6) rectangle (3.25, 11.1);
\node at (.9375, 10.85) {$\rho_N$};
}$$
After applying \eqref{equ:mpq}, this expression is equal to $\iota(l_{\iota(N)})\circ (\widehat{I_A\overset{\infty}{\otimes}_A \id})$. Hence the diagram commutes.
\end{proof}

Define the left and right unit isomorphisms by
\[
\ell_M= l_{\iota(M)} \colon A \boxtimes_A M \to M
\]
and
\[
\gamma_M= r_M\overset{\infty}{\circ} (\id\overset{\infty}{\otimes}_A I_A)=r_M\circ(\widehat{\id\overset{\infty}{\otimes}_A I_A}) \colon M \boxtimes_A A \to M.
\]
Then the triangle axiom holds in $\ED^{\RR}_{\infty}(A)$.

\begin{prop}\label{prop:triangle-rightmodule}
Let $(A,m^n;\mu^{p,q})$ be a $B_{\infty}$-algebra, and let $(M,\rho^n_M)$ and $(N,\rho^n_N)$ be right $A_{\infty}$-modules. Then the following diagram commutes:
$$
\xymatrix{
(M\boxtimes_A A)\boxtimes_A N \ar[rr]^{a_{M,A,N}}\ar[rd]_{\gamma_M\boxtimes_A \id}& & M\boxtimes_A (A\boxtimes_A N) \ar[ld]^{\id\boxtimes_A\ell_N}\\
 & M\boxtimes_A N &
 }
$$
\end{prop}
\begin{proof}
Expanding the diagram using the definition of $\boxtimes_A$, it remains to prove that the outer pentagon in the following diagram commutes:
$$\xymatrix{
(M\overset{\infty}{\otimes}_A \iota(A))\overset{\infty}{\otimes}_A \iota(N) \ar[r]^{\widetilde a}\ar[d]_{(\id\overset{\infty}{\otimes}_A I_A)\overset{\infty}{\otimes}_A\id}& M\overset{\infty}{\otimes}_A (\iota(A)\overset{\infty}{\otimes}_A \iota(N))\ar[r]^{\id\overset{\infty}{\otimes}_A\psi}\ar[d]_{\id\overset{\infty}{\otimes}_A (I_A\overset{\infty}{\otimes}_A\id)} & M\overset{\infty}{\otimes}_A \iota(A\overset{\infty}{\otimes}_A \iota(N))\ar[d]^{\id\overset{\infty}{\otimes}_A\iota(l)} \\
(M\overset{\infty}{\otimes}_A A)\overset{\infty}{\otimes}_A\iota(N)\ar[r]^{\widetilde a} \ar@/_2pc/
[rr]_{r\overset{\infty}{\otimes}_A \id}& M\overset{\infty}{\otimes}_A (A\overset{\infty}{\otimes}_A \iota(N))\ar[r]^{\id\overset{\infty}{\otimes}_A l}& M\overset{\infty}{\otimes}_A \iota(N)
 }$$
where composition is as in \eqref{align:compositionmorphisms}.
The functoriality of $\widetilde a$ implies the commutativity of the left square. Lemma \ref{lemm:Hwithunit} implies the commutativity of the right square. The triangle axiom in $\ED^{\BB}_{\infty}(A)$ implies the commutativity of the lower triangle.
Thus the outer pentagon commutes.
\end{proof}

We now prove Theorem \ref{thm:monoidalcatofrightmod}.
\begin{proof}[Proof of Theorem \ref{thm:monoidalcatofrightmod}]
The associativity and triangle axioms are given by Propositions \ref{prop:asso-rightmodule} and \ref{prop:triangle-rightmodule}. Exactness follows from the triangulated functoriality of the $A_\infty$-tensor product together with the induction functor $\iota$ from Proposition \ref{prop:dualitybinfinity}. Moreover, by Remark \ref{rem:tensormorphism}, the tensor product $\boxtimes_A$ admits an internal Hom and hence commutes with coproducts.
\end{proof}

\subsection{Brace $B_\infty$-algebras}
We now record a useful special class of $B_\infty$-algebras: the brace $B_\infty$-algebras.
\begin{defn}\label{defn:brace}
A $B_\infty$-algebra $(A,m^n;\mu^{p,q})$ is called a brace $B_\infty$-algebra if $m^n=0$ for $n>2$ and $\mu^{p,q}=0$ for $p>1$. Here $\mu^{0,1}=\id_{sA}=\mu^{1,0}$ and $\mu^{0,0}=0$.

In particular, the underlying $A_\infty$-algebra is a dg algebra. We shall write a brace $B_\infty$-algebra simply as $(A,m^n;\mu^{1,q})$.
\end{defn}

\begin{rem}\label{rem:commtativebrace}
Let $(A,m^n;\mu^{1,q})$ be a brace $B_\infty$-algebra. Then its cohomology algebra $\mathrm H^*(A)$ is necessarily graded commutative. Indeed, we have the identity (see Definition \ref{defn:bracebmodule})
 \begin{align*}
 & m^1(\mu^{1,1} (sa \otimes sb)) + (-1)^{|sa||sb|} m^2(sb \otimes sa) + m^2(sa \otimes sb) \\
 ={} & \mu^{1,1}(m^1(sa) \otimes sb) + (-1)^{|sa|} \mu^{1,1}(sa \otimes m^1(sb)).
 \end{align*}
In particular, if $m^1=0$, then $(A,m^2)$ is a graded commutative algebra; compare Subsection \ref{subsection:commutative}. More generally, $(\mathrm H^*(A),m^2,[-,-])$ is a Gerstenhaber algebra, where the Lie bracket is defined by $$[a, b]=(-1)^{|a|} s^{-1} \mu^{1,1}(sa \otimes sb) -(-1)^{|b|+|sa||sb|} s^{-1} \mu^{1,1}(sb \otimes sa).$$ 
See \cite[Lemma 5.18]{CLW} or \cite[Subsection 5.2]{GJ}.
\end{rem}

\begin{rem}
Despite this restriction, brace $B_\infty$-algebras include many fundamental examples, such as the (singular) Hochschild cochain complex of an associative algebra (see, for example, \cite{GeVo, LoVa, Vor, Wa}) and the Yoneda dg algebra of a Hopf algebra (see Proposition \ref{prop:bracestructure} below).
\end{rem}

\subsubsection{\textbf{\textup{Brace dg modules}}}
We now introduce brace dg modules. They provide a convenient criterion for comparing an ordinary dg bimodule with its induced $A_\infty$-bimodule, and this comparison is essential for the unit constraint in Proposition \ref{prop:triangle-rightmodule}.

\begin{defn}\label{defn:bracebmodule}
Let $(A, m^n; \mu_A^{1, q})$ be a brace $B_{\infty}$-algebra.
A {\it brace $B_{\infty}$-module} over $A$ is a dg bimodule $(M,\beta_M^{0,0},\beta_M^{1,0},\beta_M^{0,1})$ endowed with degree-zero maps
$$\mu_M^{1,q}: sM\otimes (sA)^{\otimes q}\rightarrow sM$$
with $\mu_M^{1,0}=\id_M$, satisfying the following identities:
\begin{enumerate}
\item \label{higher-pre-Jacobi} Higher pre-Jacobi identity:
\begin{align*}
\mu_M^{1,q}\circ (\mu_M^{1,p} \smallotimes \id_{sA}^{\otimes q})=\sum_{r=1}^{p+q}\sum_{\substack{l_1+\cdots+l_r=p\\ n_1+\dotsb + n_r=q} } \mu_M^{1,r} \circ \left (\id_{sM} \smallotimes (\mu_A^{l_1,n_1}\smallotimes \cdots \smallotimes \mu_A^{l_r, n_r}) \circ \tau_{(l_1,\dotsc,l_r; n_1,\dotsc,n_r)}\right),
\end{align*}
where $\tau_{(l_1,\dotsc,l_r; n_1,\dotsc,n_r)}$ is the block permutation, see Remark \ref{rem:widehatmu}.

\item \label{distributivity} Distributivity:
\begin{align*}
\sum_{n_1+n_2=q}\beta_M^{0,1}\circ (\mu_M^{1, n_1}\smallotimes \mu_A^{1,n_2} \circ \tau_{(1,1; n_1,n_2)}) = \mu_M^{1,q}\circ (\beta_M^{0,1}\smallotimes \id^{\otimes q})
\end{align*}
where the equality is evaluated on elements of $(sM\otimes sA) \bigotimes (sA)^{\otimes q}$.
\item \label{higher-homotopy} Higher homotopy:
\begin{align*}
& \beta_M^{0,0}\circ \mu_M^{1,q}+ \beta_M^{0,1}\circ (\mu_M^{1,q-1}\smallotimes \id )+ \beta_M^{1, 0} \circ (\id\smallotimes \mu_M^{1,q-1})\circ\tau_{(0,1; 1,q-1)}\\ ={} & \mu_M^{1,q}\circ ( \beta_M^{0,0}\smallotimes \id^{\smallotimes q})+ \sum_{i=0}^{q-1} \mu_M^{1,q}\circ (\id^{\smallotimes i+1} \smallotimes m^1\smallotimes \id^{\smallotimes q-i-1}) + \sum_{i=0}^{q-2} \mu_M^{1,q-1}\circ ( \id^{\smallotimes i+1}\smallotimes m^2\smallotimes \id^{\smallotimes q-i-2})
\end{align*}
where the equality is evaluated on elements of $sM\bigotimes (sA)^{\otimes q}$, and on the right-hand side the term $\id^{\smallotimes i+1}$ means $\id_{sM} \smallotimes \id^{\smallotimes i}$.
\end{enumerate}
\end{defn}

\begin{rem}
The brace $B_\infty$-algebra $(A, m^n; \mu^{1,q})$ can be regarded as a brace $B_{\infty}$-module over itself via $\beta^{0,0}=m^1$, $\beta^{1,0}=\beta^{0,1}=m^2$, and the higher operations $\mu^{1,q}$.
\end{rem}

\begin{rem}\label{rem:dgainfinitymodules}
 Since a brace $B_\infty$-algebra $A$ is a dg algebra, Remark \ref{rem:Ainfinitydgmodules} gives a triangulated equivalence $\ED^{\RR}_\infty(A)\simeq \ED^{\RR}(A)$. Thus, in this subsection, we may work with right dg modules rather than general $A_\infty$-modules. Moreover, when restricted to dg bimodules, the tensor functor $-\overset{\infty}{\otimes}_A-$ agrees naturally with the derived tensor product $-\otimes_A^{\mathbb{L}}-$.
\end{rem}

A brace module is, in particular, a dg bimodule over the dg algebra $A$. We now consider the opposite $B_\infty$-algebra $A^{\rm opp}$; see Definition \ref{defnopposite}. Recall from \eqref{align:iotabimodule} that the right dg module structure also gives rise to an $A_\infty$-bimodule $\iota(M)$ over the Hopf algebra $(T^c(sA),\Delta,\widehat{m},\widehat{\mu}^{\opp})$. The next theorem shows that the additional brace-module operations provide an explicit $A_\infty$-quasi-isomorphism from $\iota(M)$ back to the original dg bimodule $M$.

\begin{thm}\label{thm:bracemodulegeneral}
Let $(A, m^n; \mu^{1, q})$ be a brace $B_{\infty}$-algebra and $(M, \beta_M; \mu_M^{1,q})$ a brace $B_{\infty}$-module over $A$.
Then the dg bimodule $(M, \beta_M)$ is $A_\infty$-quasi-isomorphic to $\iota(M)$ as $A_\infty$-bimodules.
In particular, $\iota(A)$ is $A_\infty$-quasi-isomorphic to $A$.
\end{thm}

\begin{proof}
The construction of the $A_\infty$-bimodule morphism $I_M\colon \iota(M)\to M$ is analogous to the proof of Theorem \ref{thm:iotabracemodule}. More explicitly, the component
$$I^{p,q}_M: (sA)^{\otimes p}\otimes s(\iota(M))\otimes (sA)^{\otimes q}\rightarrow sM$$
is depicted by the following graph:
$$\tikzmath{
\draw (-1.5, .9) --(-1.5, 1.2);
\draw (0, .9) --(0, 1.7);
\draw (1, .9) --(1, 4.5);

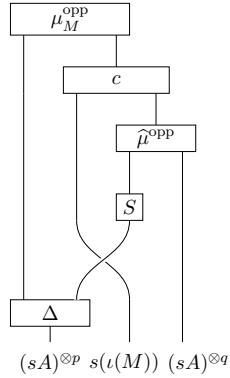
\node at (-1.5, .5) {$(sA)^{\otimes p}$};
\node at (-.1, .5) {$s(\iota(M))$};
\node at (1.3, .5) {$(sA)^{\otimes q}$};

\draw (-2.25, 1.2) rectangle (-.75, 1.7);
\node at (-1.5, 1.45) {$\Delta$};
\draw (-2, 5.6) -- (-2, 1.7);
\draw [knot] (0, 1.7) to[out=90,in=-90](-1, 3.2);
\draw [knot] (-1, 1.7) to[out=90,in=-90](0, 3.2);
\draw (-1, 5.6) -- (-1, 3.2);
\draw (0, 4.5) -- (0, 3.7);
\draw (-.25, 3.2) rectangle (.25, 3.7);
\node at (0, 3.45) {$S$};
\draw (1.25, 4.5) rectangle (-.25, 5);
\node at (.5, 4.75) {$\widehat{\mu}^{\opp}$};

\draw (-1.25, 5.6) rectangle (.75, 6.1);
\node at (-.25, 5.85) {$c$};
\draw (.5, 5.6) -- (.5, 5);
\draw (-.25, 6.1) -- (-.25, 6.7);
\draw (-2, 5.6) -- (-2, 6.7);
\draw (-2.25, 6.7) rectangle (0, 7.3);
\node at (-1.125, 7) {$\mu^{\opp}_M$};
}$$
Here $\widehat \mu^{\opp}$ appears because we are working over the opposite $B_\infty$-algebra. Since $\mu_M^{p,q}=0$ for $p>1$, the map
$$I^{p,0}_M: (sA)^{\otimes p}\otimes s(\iota(M))\rightarrow sM$$
reduces to
$$\tikzmath{
\node at (-1.2, 1.3) {$(sA)^{\otimes p}$};
\node at (0.6, 1.3) {$s(\iota(M))$};
\node at (1.8, 1.2) {,};
\draw [knot] (0, 1.7) to[out=90,in=-90](0, 3.2);
\draw [knot] (-1, 1.7) to[out=90,in=-90](-1, 3.2);
\draw (-1.25, 3.2) rectangle (.25, 3.8);
\node at (-.5, 3.5) {$\mu_M^{\opp}$};
}$$
and $I^{p,q}_M=0$ for $q \neq 0$.
The verification that $I_M$ is an $A_\infty$-bimodule morphism is the same graphical computation as in the proof of Theorem \ref{thm:iotabracemodule}, with the brace-module identities replacing the corresponding brace-algebra identities.
\end{proof}

\begin{rem}\label{rem:iotaMhigherleft}
Although $M$ is a dg bimodule, $\iota(M)$ is generally only an $A_\infty$-bimodule. It is a dg bimodule only in special cases, for instance when the higher brace operations $\mu^{1,q}$ vanish for $q\geq1$; see Proposition \ref{prop:commuativeiota}.
\end{rem}

\section{Koszul duality and tensor products for Hopf algebras}\label{section:koszulduality}
Throughout this section, $H$ is a finite-dimensional Hopf algebra and all $H$-modules are right modules. The coproduct gives the tensor product of two right $H$-modules its diagonal action,
\[
(x\otimes y)h=\sum xh^{(1)}\otimes yh^{(2)}.
\]
For finite-dimensional Hopf algebras, the tensor product of injective modules is again injective. Consequently, the homotopy category $\EK(\mathrm{Inj}\text{-}H)$ of complexes of injective $H$-modules is a monoidal triangulated category; its tensor unit is represented by an injective resolution of the trivial module $\Bbbk$; compare \cite{BeKr}.

Our aim is to describe this tensor product after Koszul duality. The Koszul dual algebra will be a dg model for $\operatorname{RHom}_H(\Bbbk,\Bbbk)$, namely the Yoneda dg algebra $\Y(\Bbbk,\Bbbk)$. The Hopf structure on $H$ supplies this dg algebra with brace operations, and hence with the monoidal structure constructed in Section \ref{section:monoidaltri}. We then compare that monoidal structure with the tensor product on $\EK(\mathrm{Inj}\text{-}H)$. The comparison is lax monoidal in general and becomes a monoidal equivalence on the localizing subcategory generated by the trivial module.

\subsection{The Yoneda dg category}
The construction in this subsection is available for every associative algebra; the Hopf structure will enter only in the next subsection. Recall from \cite{ChWa} the \emph{Yoneda dg category} $\mathcal{Y}=\mathcal{Y}_H$. Its objects are complexes of right $H$-modules. For two such complexes $X$ and $Y$, set
\[
\mathcal{Y}(X,Y)=\prod_{n\geq0}\mathcal{Y}_n(X,Y),
\qquad
\mathcal{Y}_n(X,Y)=\operatorname{Hom}\bigl(X\otimes(sH)^{\otimes n},Y\bigr).
\]
Thus $\mathcal{Y}_n(X,Y)$ consists of cochains with $n$ bar variables, and $\mathcal{Y}_0(X,Y)=\operatorname{Hom}(X,Y)$. The differential has an internal part, coming from the differentials of $X$ and $Y$, and a bar part, coming from the multiplication and module structures. More precisely, it is determined by
\[
\begin{pmatrix}\delta_{\rm in}\\
\delta_{\rm ex}\end{pmatrix}
\colon \mathcal{Y}_n(X,Y)
\longrightarrow
\mathcal{Y}_n(X,Y)\oplus\mathcal{Y}_{n+1}(X,Y),
\]
where
\begin{equation}\label{deltayoneda1}
\begin{aligned}
\delta_{\rm in}(f)(x \otimes sa_{1,n}) ={} & d_Y\bigl(f(x \otimes sa_{1,n})\bigr)-(-1)^{|f|}f(d_X(x) \otimes sa_{1,n}),\\
\delta_{\rm ex}(f)(x\otimes sa_{1,n+1}) ={} & (-1)^{|f|}f(x a_1\otimes sa_{2,n+1})+(-1)^{|f|+n+1}f(x\otimes sa_{1,n})\circ a_{n+1}\\
&+\sum_{i=1}^{n}(-1)^{|f|+i} f(x\otimes sa_{1,i-1}\otimes sa_ia_{i+1}\otimes sa_{i+2,n+1}).
\end{aligned}
\end{equation}
If $X$ and $Y$ are ordinary modules, viewed as stalk complexes in degree $0$, then $\delta_{\rm in}=0$ and this is the usual bar-type differential.

Composition is concatenation of bar variables. For $f\in\mathcal{Y}_n(X,Y)$ and $g\in\mathcal{Y}_m(Y,Z)$, define $g\odot f\in\mathcal{Y}_{m+n}(X,Z)$ by
\begin{equation}\label{equation:yonedaproduct}
(g\odot f)(x\otimes sa_{1,m+n})
 =g\bigl(f(x\otimes sa_{1,n})\otimes sa_{n+1,m+n}\bigr).
\end{equation}
The identity of $X$ is the usual map $\id_X\in\mathcal{Y}_0(X,X)$. In particular,
\[
\id_Y\odot f=f=f\odot\id_X.
\]
The dg algebra $\Y(\Bbbk,\Bbbk)$ is therefore a concrete model for the derived endomorphism algebra of the trivial module, and
\[
H^*\Y(\Bbbk,\Bbbk)\cong\Ext_H^*(\Bbbk,\Bbbk).
\]

\subsection{The brace structure supplied by the coproduct}\label{subsection:binfinitybialgebra}
The counit $\varepsilon_H\colon H\to\Bbbk$ makes $\Bbbk$ the trivial right $H$-module. In this subsection, we use the coproduct of $H$ to construct a brace $B_\infty$-structure on
\[
E:=\Y(\Bbbk,\Bbbk)
\]
and compatible brace-module structures on $\Y(X,Y)$. Only the bialgebra structure of $H$ is needed for these constructions.

Under the identification
\[
\Y_m(\Bbbk,\Bbbk)=\operatorname{Hom}\bigl((sH)^{\otimes m},\Bbbk\bigr),
\]
the product $\odot$ is the usual concatenation, or cup, product: for $f\in\Y_n(\Bbbk,\Bbbk)$ and $g\in\Y_m(\Bbbk,\Bbbk)$,
\[
(g\odot f)(sa_{1,m+n})=f(sa_{1,n})g(sa_{n+1,m+n}).
\]
Its unit is $\id_\Bbbk\in\Y_0(\Bbbk,\Bbbk)$.

Let $X$ and $Y$ be complexes of right $H$-modules. The coproduct gives $\Y(X,Y)$ a natural dg $E$-bimodule structure. For $f\in\Y_n(X,Y)$ and $g\in\Y_m(\Bbbk,\Bbbk)$, define
\begin{align}\label{align:cupproducts}
\begin{aligned}
(g\cup f)(x\otimes sa_{1,m+n})
 & = f(x\otimes sa_{1,n})a^{(2)}_{n+1}\cdots a^{(2)}_{m+n}\,
 g(sa^{(1)}_{n+1}\otimes\cdots\otimes sa^{(1)}_{m+n}),\\
(f\cup g)(x\otimes sa_{1,m+n})
 & = f(xa^{(2)}_1\cdots a^{(2)}_m\otimes sa_{m+1,m+n})
 g(sa^{(1)}_1\otimes\cdots\otimes sa^{(1)}_m).
\end{aligned}
\end{align}
These formulas simply split each relevant element by the coproduct: one tensor factor is evaluated by $g$, while the other acts on the module variable or on the value of $f$. When $m=0$, the two actions reduce to scalar multiplication:
\[
(f\cup k)(x\otimes sa_{1,n})=k f(x\otimes sa_{1,n})=(k\cup f)(x\otimes sa_{1,n})
\qquad(k\in\Bbbk).
\]

We next define the brace operations. For $h\in\Y_n(\Bbbk,\Bbbk)$, let
\[
\widetilde h\colon(sH)^{\otimes n}\longrightarrow H
\]
be the $H$-valued cochain
\[
\widetilde h(sa_1\otimes\cdots\otimes sa_n)
 =\sum h(sa_1^{(1)}\otimes\cdots\otimes sa_n^{(1)})a_1^{(2)}a_2^{(2)}\cdots a_n^{(2)},
\]
where $\Delta(a_i)=\sum a_i^{(1)}\otimes a_i^{(2)}$. For $n=0$, set $\widetilde h=0$. Thus $\widetilde h$ is obtained by applying the coproduct to every input, evaluating $h$ on the first tensor factors, and multiplying the remaining factors in $H$.

For $q\geq0$, define
\[
\mu^{1,q}_{X,Y}\colon
\Y(X,Y)\otimes\Y(\Bbbk,\Bbbk)^{\otimes q}
\longrightarrow\Y(X,Y).
\]
If $f\in\Y_n(X,Y)$ and $g_i\in\Y_{m_i}(\Bbbk,\Bbbk)$, then
\begin{align}\label{align:mupqformular}
\mu^{1,q}_{X,Y}(sf\smallotimes sg_1\smallotimes\dotsb\smallotimes sg_q)
 ={}&\sum(-1)^{\epsilon_q}sf\circ
 \bigl(\id\smallotimes\id^{\smallotimes i_{q+1}}\smallotimes s\circ\g_q
 \smallotimes\cdots \notag\\
 &\hspace{7em}\smallotimes\id^{\smallotimes i_2}\smallotimes s\circ\g_1
 \smallotimes\id^{\smallotimes i_1}\bigr).
\end{align}
where
\[
\epsilon_q=\sum_{j=1}^{q-1}(|g_j|-1)(|g_{j+1}|+\cdots+|g_q|-q+j),
\]
and the sum ranges over
\[
\bigl\{(i_1,\dotsc,i_{q+1})\in\mathbb Z_{\geq0}^{q+1}
 \mid i_1+\cdots+i_{q+1}=n-q\bigr\}.
\]
In other words, the operation inserts the $H$-valued cochains $\widetilde g_1,\dotsc,\widetilde g_q$ into the bar variables of $f$. If there are not enough slots, the result is zero. We set $\mu^{1,0}_{X,Y}=\id$.

\begin{prop}\label{prop:bracestructure}
Let $H$ be a bialgebra. Then $E=\Y(\Bbbk,\Bbbk)$, equipped with the operations above, is a brace $B_\infty$-algebra. For every pair of complexes of right $H$-modules $X$ and $Y$, the dg $E$-bimodule $\Y(X,Y)$ is a brace $B_\infty$-module over $E$.
\end{prop}
\begin{proof}
Consider the map
\[
\Omega\colon\Y(\Bbbk,\Bbbk)\longrightarrow C^*(H,H)
\]
defined by
\[
\Omega^n(f)(sa_1\otimes\cdots\otimes sa_n)
 =f(sa_1^{(1)}\otimes\cdots\otimes sa_n^{(1)})a_1^{(2)}\cdots a_n^{(2)};
\]
compare \cite[Proposition 2.3]{LMSS}. This is an embedding of complexes and is compatible with the brace operations. Since the Hochschild cochain complex $C^*(H,H)$ carries its standard brace $B_\infty$-structure \cite{GeVo}, the image of $\Omega$ inherits a brace $B_\infty$-structure, which is precisely the one defined above on $\Y(\Bbbk,\Bbbk)$.

The brace-module identities for $\Y(X,Y)$ are the corresponding insertion identities with one distinguished module-valued cochain. They follow by the same computation, using coassociativity and multiplicativity of the coproduct. We omit the routine sign verification.
\end{proof}

\begin{rem}
On cohomology, the brace structure recovers the Gerstenhaber algebra structure on $\Ext_H^*(\Bbbk,\Bbbk)$ constructed in \cite{FaSo}. Thus the brace $B_\infty$-algebra $\Y(\Bbbk,\Bbbk)$ is a chain-level enhancement of the familiar Gerstenhaber structure on Hopf-algebra cohomology. It may also be viewed as a dual analogue of the brace construction in \cite[Section 5]{Kad}; see also \cite[Section 2]{Yo}.
\end{rem}

Proposition \ref{prop:bracestructure} allows us to apply the general construction of Section \ref{section:monoidaltri}.
\begin{thm}\label{thm:bialgebrainducedbimodules}
Let $H$ be a bialgebra and set $E=\Y(\Bbbk,\Bbbk)$. Then $\ED(E)$ is a closed monoidal triangulated category with tensor product $\boxtimes_E$. Moreover, for every pair $X,Y$ of complexes of right $H$-modules, the naturally defined dg $E$-bimodule $\Y(X,Y)$ is $A_\infty$-isomorphic to its induced bimodule $\iota(\Y(X,Y))$.
\end{thm}
\begin{proof}
The first assertion follows from Proposition \ref{prop:bracestructure} and Theorem \ref{thm:monoidalcatofrightmod}. The second follows from Proposition \ref{prop:bracestructure} and Theorem \ref{thm:iotabracemodule}.
\end{proof}

\begin{rem}
No antipode of $H$ is used in Proposition \ref{prop:bracestructure} or Theorem \ref{thm:bialgebrainducedbimodules}; the multiplication, unit, coproduct, and counit are sufficient. The Hopf hypothesis becomes important when we compare with the tensor category of $H$-modules and use injective resolutions below.
\end{rem}

\begin{rem}\label{rem:relativeyoneda}
The construction has a normalized relative version that is often more convenient for computations. Let $E_0\subset H$ be a semisimple subalgebra and write $\overline H=H/(E_0)$ for the quotient $E_0$-$E_0$-bimodule. Recall from \cite[Section 7]{ChWa} that the normalized $E_0$-relative Yoneda complex is
\[
\overline\Y_m(X,Y)
 =\operatorname{Hom}_{E_0\text{-}E_0}
 \bigl(X\otimes_{E_0}(s\overline H)^{\otimes_{E_0}m},Y\bigr),
\]
with the induced bar differential. The natural inclusion
\[
\overline\Y(X,Y)\longrightarrow\Y(X,Y)
\]
is a quasi-isomorphism.

If $E_0$ is also a sub-bialgebra, then the coproduct preserves the relative normalized construction. Hence the brace operations restrict to $\overline\Y(\Bbbk,\Bbbk)$, and
\[
\overline\Y(\Bbbk,\Bbbk)\hookrightarrow\Y(\Bbbk,\Bbbk)
\]
is a quasi-isomorphism of brace $B_\infty$-algebras. This smaller model is the Hopf-algebra analogue of normalized relative bar cochains and will be used in Proposition \ref{prop:braceradical}.
\end{rem}

\subsection{A multiplicative injective resolution of the trivial module}
We now return to the assumption that $H$ is finite-dimensional and set
\[
\mathbb 1:=\Y(H,\Bbbk).
\]
This is the bar-type injective resolution that will represent the tensor unit in $\EK(\mathrm{Inj}\text{-}H)$. Indeed,
\[
\Y_m(H,\Bbbk)=\operatorname{Hom}\bigl(H\otimes(sH)^{\otimes m},\Bbbk\bigr)
\]
is an injective right $H$-module for every $m\geq0$, and the counit defines a coaugmentation
\[
\epsilon\colon\Bbbk\longrightarrow\Y(H,\Bbbk),
\qquad
1_\Bbbk\longmapsto\varepsilon.
\]
Here $\varepsilon\in\Y_0(H,\Bbbk)$ is the counit of $H$. See \cite[Subsection 3.2]{ChWa1} for a general discussion.

The coproduct of $H$ makes this resolution multiplicative. Define
\[
\mu\colon\Y(H,\Bbbk)\otimes\Y(H,\Bbbk)
\longrightarrow\Y(H,\Bbbk)
\]
by
\begin{align*}
&\mu(f\smallotimes g)(a_0\smallotimes sa_1\smallotimes\dotsb\smallotimes sa_{m+n})\\
={}&(-1)^{nm}
 f(a_0^{(1)}\smallotimes sa_1^{(1)}\smallotimes\dotsb\smallotimes sa_m^{(1)})
 g(a_0^{(2)}a_1^{(2)}\dotsb a_m^{(2)}\smallotimes sa_{m+1}\smallotimes\dotsb\smallotimes sa_{m+n}),
\end{align*}
for $f\in\Y_m(H,\Bbbk)$ and $g\in\Y_n(H,\Bbbk)$. The map $\mu$ is compatible with the differential and with the right $H$-action. Thus $\mathbb 1$ is a dg algebra object in $\EK(\mathrm{Inj}\text{-}H)$.

\begin{prop}\label{prop:yhkasso}
The product $\mu$ is associative. Moreover,
\[
\mu\colon\mathbb 1\otimes\mathbb 1\longrightarrow\mathbb 1
\]
is a homotopy equivalence of right dg $H$-modules.
\end{prop}
\begin{proof}
Let $f\in\Y_m(H,\Bbbk)$, $g\in\Y_n(H,\Bbbk)$, and $h\in\Y_p(H,\Bbbk)$, and put $k=m+n+p$. Expanding the two parenthesizations gives
\begin{align*}
&(-1)^{(m+n)p}\mu(\mu(f\smallotimes g)\smallotimes h)
 (a_0\smallotimes sa_1\smallotimes\dotsb\smallotimes sa_k)\\
={}&\mu(f\smallotimes g)(a_0^{(1)}\smallotimes sa_1^{(1)}\smallotimes\dotsb\smallotimes sa_{m+n}^{(1)})
 h(a_0^{(2)}\dotsb a_{m+n}^{(2)}\smallotimes sa_{m+n+1}\smallotimes\dotsb\smallotimes sa_k)\\
={}&(-1)^{mn}f(a_0^{(11)}\smallotimes sa_{1,m}^{(11)})
 g(a_0^{(12)}\dotsc a_m^{(12)}\smallotimes sa_{m+1,m+n}^{(1)})
 h(a_0^{(2)}\dotsc a_{m+n}^{(2)}\smallotimes sa_{m+n+1,k}),
\end{align*}
and
\begin{align*}
&(-1)^{m(n+p)}\mu(f\smallotimes\mu(g\smallotimes h))
 (a_0\smallotimes sa_1\smallotimes\dotsb\smallotimes sa_k)\\
={}&f(a_0^{(1)}\smallotimes sa_1^{(1)}\smallotimes\dotsb\smallotimes sa_m^{(1)})
 \mu(g\smallotimes h)(a_0^{(2)}a_1^{(2)}\dotsb a_m^{(2)}\smallotimes sa_{m+1}\smallotimes\dotsb\smallotimes sa_k)\\
={}&(-1)^{np}f(a_0^{(1)}\smallotimes sa_{1,m}^{(1)})
 g(a_0^{(21)}\dotsb a_m^{(21)}\smallotimes sa_{m+1,m+n}^{(1)})\\
&\qquad\cdot
 h(a_0^{(22)}\dotsb a_m^{(22)}a_{m+1}^{(2)}\dotsb a_{m+n}^{(2)}
 \smallotimes sa_{m+n+1,k}).
\end{align*}
Coassociativity identifies these two expressions, proving associativity. Compatibility with the differential follows directly from the bar differential and the fact that $\Delta$ is an algebra map.

The coaugmentation fits into the commutative diagram
\[
\xymatrix{
\mathbb 1\otimes\mathbb 1\ar[r]^-{\mu} & \mathbb 1\\
\Bbbk\otimes\Bbbk\ar[r]^-{\cong}\ar[u]^-{\epsilon\otimes\epsilon} & \Bbbk\ar[u]^-{\epsilon}.
}
\]
Thus $\mu$ is a lift of the canonical isomorphism $\Bbbk\otimes\Bbbk\cong\Bbbk$ between injective resolutions. It is therefore a homotopy equivalence.
\end{proof}

The left cup action from \eqref{align:cupproducts} induces a dg algebra morphism
\[
\Y(\Bbbk,\Bbbk)
\longrightarrow
\operatorname{Hom}_H(\Y(H,\Bbbk),\Y(H,\Bbbk)),
\qquad
f\longmapsto(g\longmapsto f\cup g),
\]
which is a quasi-isomorphism by \cite[Proposition 3.11]{ChWa1}. Thus the Yoneda dg algebra $E=\Y(\Bbbk,\Bbbk)$ is identified, up to quasi-isomorphism, with the dg endomorphism algebra of the chosen injective resolution $\mathbb 1$.

The multiplicative resolution $\mathbb 1$ now yields the comparison map needed for monoidality.
\begin{prop}\label{prop:naturaltrans}
For $X,Y\in\EK(\mathrm{Inj}\text{-}H)$, there is a natural morphism
\[
\widetilde\eta_{X,Y}\colon
\operatorname{Hom}_H(\mathbb 1,X)
\otimes^{\mathbb L}_{\Y(\Bbbk,\Bbbk)}
\operatorname{Hom}_H(\mathbb 1,Y)
\longrightarrow
\operatorname{Hom}_H(\mathbb 1,X\otimes Y).
\]
\end{prop}
\begin{proof}
In $\EK(\mathrm{Inj}\text{-}H)$, write $\mu^{-1}\colon\mathbb 1\to\mathbb 1\otimes\mathbb 1$ for the inverse of the isomorphism represented by $\mu$. The external tensor product of morphisms, followed by precomposition with $\mu^{-1}$, gives
\[
\operatorname{Hom}_H(\mathbb 1,X)\otimes\operatorname{Hom}_H(\mathbb 1,Y)
\longrightarrow
\operatorname{Hom}_H(\mathbb 1\otimes\mathbb 1,X\otimes Y)
\xrightarrow{\operatorname{Hom}_H(\mu^{-1},X\otimes Y)}
\operatorname{Hom}_H(\mathbb 1,X\otimes Y).
\]
Associativity of $\mu$ implies that this map is balanced over
$\operatorname{Hom}_H(\mathbb 1,\mathbb 1)$. Transporting the action along the quasi-isomorphism
\[
\Y(\Bbbk,\Bbbk)\xrightarrow{\simeq}\operatorname{Hom}_H(\mathbb 1,\mathbb 1)
\]
and passing to the derived balanced tensor product gives $\widetilde\eta_{X,Y}$.
\end{proof}

\subsection{The lax monoidal Koszul duality functor}
Keep the notation
\[
E=\Y(\Bbbk,\Bbbk),
\qquad
\mathbb 1=\Y(H,\Bbbk).
\]
The Koszul duality functor associated with the injective resolution $\mathbb 1$ is
\[
F\colon\EK(\mathrm{Inj}\text{-}H)\longrightarrow\ED(E),
\qquad
F(X)=\operatorname{Hom}_H(\mathbb 1,X).
\]
The preceding subsection gives the ordinary derived tensor comparison. To compare it with the product $\boxtimes_E$, we use the brace-module structure.

Since $\mathbb 1=\Y(H,\Bbbk)$ is a dg $E$-bimodule by \eqref{align:cupproducts}, the complex $F(X)$ inherits a dg $E$-bimodule structure.
\begin{lem}\label{lem:bracemoduleyhk}
For every $X\in\EK(\mathrm{Inj}\text{-}H)$, the dg $E$-bimodule $F(X)$ is a brace $B_\infty$-module over $E$. Consequently,
\[
\iota(F(X))\simeq F(X)
\]
as $A_\infty$-$E$-bimodules.
\end{lem}
\begin{proof}
For $\psi\in\operatorname{Hom}_H(\Y(H,\Bbbk),X)$ and $g_i\in E$, define
\[
\mu^{1,q}(s\psi\otimes sg_1\otimes\dotsb\otimes sg_q)(f)
 =(-1)^\epsilon
 \psi\bigl(s^{-1}\mu_{H,\Bbbk}^{1,q}(sf\otimes sg_1\otimes\dotsb\otimes sg_q)\bigr),
\]
where $\mu_{H,\Bbbk}^{1,q}$ is the brace action on $\Y(H,\Bbbk)$ from \eqref{align:mupqformular} and
\[
\epsilon=|f|(|g_1|+\cdots+|g_q|-q).
\]
Thus the braces on $F(X)$ are obtained by precomposing $\psi$ with the braces on the resolution $\mathbb 1$. The brace-module identities follow from those for $\Y(H,\Bbbk)$. The final assertion is Theorem \ref{thm:iotabracemodule}.
\end{proof}

Let $\operatorname{Loc}(\mathbb 1)$ denote the smallest localizing subcategory of $\EK(\mathrm{Inj}\text{-}H)$ containing $\mathbb 1$; that is, the smallest full triangulated subcategory containing $\mathbb 1$ and closed under coproducts.

\begin{thm}\label{thm:monoidalfunctor}
Let $H$ be a finite-dimensional Hopf algebra. The Koszul duality functor
\[
F\colon\EK(\mathrm{Inj}\text{-}H)\longrightarrow\ED(\Y(\Bbbk,\Bbbk))
\]
is triangulated lax monoidal. Its restriction
\[
F\big|_{\operatorname{Loc}(\Y(H,\Bbbk))}
\colon
\operatorname{Loc}(\Y(H,\Bbbk))
\longrightarrow
\ED(\Y(\Bbbk,\Bbbk))
\]
is a monoidal triangulated equivalence.
\end{thm}
\begin{proof}
For $X,Y\in\EK(\mathrm{Inj}\text{-}H)$, Lemma \ref{lem:bracemoduleyhk} and Remark \ref{rem:dgainfinitymodules} give natural isomorphisms
\begin{align*}
F(X)\boxtimes_E F(Y)
 & =\operatorname{Hom}_H(\mathbb 1,X)\boxtimes_E\operatorname{Hom}_H(\mathbb 1,Y)\\
 &\simeq\operatorname{Hom}_H(\mathbb 1,X)
 \overset{\infty}{\otimes}_E
 \iota\bigl(\operatorname{Hom}_H(\mathbb 1,Y)\bigr)\\
 &\simeq\operatorname{Hom}_H(\mathbb 1,X)
 \otimes_E^{\mathbb L}
 \operatorname{Hom}_H(\mathbb 1,Y).
\end{align*}
Composing with Proposition \ref{prop:naturaltrans} yields the structure morphism
\[
\eta_{X,Y}\colon F(X)\boxtimes_E F(Y)\longrightarrow F(X\otimes Y).
\]
The quasi-isomorphism
\[
E\xrightarrow{\simeq}\operatorname{Hom}_H(\mathbb 1,\mathbb 1)=F(\mathbb 1)
\]
provides the unit morphism. Associativity of $\mu$ and the unit property of the coaugmentation imply the lax monoidal coherence relations.

The object $\mathbb 1$ is compact, $F(\mathbb 1)$ is quasi-isomorphic to $E$, and $E$ generates $\ED(E)$. The standard derived Morita argument therefore identifies the restriction of $F$ to $\operatorname{Loc}(\mathbb 1)$ with a triangulated equivalence
\[
\operatorname{Loc}(\mathbb 1)\xrightarrow{\sim}\ED(E).
\]
It remains to show that the lax structure maps are isomorphisms on this subcategory.

Fix $X\in\EK(\mathrm{Inj}\text{-}H)$ and let
\[
\mathcal C_X=\{Y\in\EK(\mathrm{Inj}\text{-}H)\mid
\eta_{X,Y}\text{ is an isomorphism}\}.
\]
Both sides of $\eta_{X,-}$ are exact and preserve coproducts, so $\mathcal C_X$ is a localizing subcategory. Moreover, $\mathbb 1\in\mathcal C_X$, because $F(\mathbb 1)\simeq E$ and $\eta_{X,\mathbb 1}$ is the right unit isomorphism. Hence
\[
\operatorname{Loc}(\mathbb 1)\subseteq\mathcal C_X.
\]
In particular, $\eta_{X,Y}$ is an isomorphism whenever $Y\in\operatorname{Loc}(\mathbb 1)$, and therefore whenever both $X$ and $Y$ belong to $\operatorname{Loc}(\mathbb 1)$. Thus the restricted equivalence is monoidal.
\end{proof}

\begin{rem}
The word ``lax'' is essential. Outside $\operatorname{Loc}(\Y(H,\Bbbk))$, the map $\eta_{X,Y}$ need not be invertible. Subsection \ref{subsec:radicalsquarezero} gives an explicit non-local Hopf algebra for which the source of one structure map is zero and its target is nonzero.
\end{rem}

For a local Hopf algebra, the distinction disappears. Recall that $H$ is \emph{local} if its underlying algebra has a unique maximal ideal. In the finite-dimensional case, the trivial module $\Bbbk$ is then the unique simple module, so its injective resolution generates the entire homotopy category of injectives.
\begin{cor}\label{cor:tensor-triangle}
Let $H$ be a finite-dimensional local Hopf algebra. Then
\[
F\colon\EK(\mathrm{Inj}\text{-}H)\longrightarrow\ED(\Y(\Bbbk,\Bbbk))
\]
is a monoidal triangulated equivalence.
\end{cor}
\begin{proof}
For a local finite-dimensional algebra, the injective resolution of the unique simple module generates $\EK(\mathrm{Inj}\text{-}H)$. Thus
\[
\operatorname{Loc}(\Y(H,\Bbbk))=\EK(\mathrm{Inj}\text{-}H),
\]
and the assertion follows from Theorem \ref{thm:monoidalfunctor}.
\end{proof}

\begin{rem}
 Corollary \ref{cor:tensor-triangle} recovers, by a purely
$A_\infty$-algebraic argument, the monoidal equivalence conjectured
by Krause and previously proved by Benson--Krause using the
$E_\infty$-algebra $C^*(BG;\Bbbk)$. The following results compares the two constructions at
the level of monoidal triangulated categories. We do not construct a direct quasi-isomorphism comparing the chosen brace
$B_\infty$-structure on $E$ with an $E_\infty$-structure on
$C^*(BG;\Bbbk)$. 
\end{rem}

\begin{cor}[Comparison with the Benson--Krause model]
\label{cor:BKcomparison}
Let $G$ be a finite $p$-group and suppose that
$\operatorname{char}(\Bbbk)=p$. Set
\[
E=\Y(\Bbbk,\Bbbk)
\]
for the Hopf algebra $H=\Bbbk G$. Then there is a monoidal
triangulated equivalence
\[
\bigl(\ED(E),\boxtimes_E\bigr)
\simeq^\otimes
\left(
\ED\bigl(C^*(BG;\Bbbk)\bigr),
\otimes^{E_\infty}_{C^*(BG;\Bbbk)}
\right).
\]
\end{cor}

\begin{proof}
By Corollary \ref{cor:tensor-triangle}, the first category is
monoidally equivalent to
\[
\bigl(\EK(\mathrm{Inj}\text{-}\Bbbk G),\otimes_\Bbbk\bigr).
\]
By \cite[Theorems 4.2 and 7.8]{BeKr}, the same is true for the second
category. Composing one equivalence with a quasi-inverse of the other
gives the assertion.
\end{proof}

\begin{rem}
The compact objects recover the bounded representation-theoretic form of the construction. Since $\mathbb 1=\Y(H,\Bbbk)$ is compact, $F$ preserves coproducts and restricts to
\[
F\colon\ED^b(\mathrm{mod}\text{-}H)
\longrightarrow
\operatorname{per}(\Y(\Bbbk,\Bbbk)),
\qquad
X\longmapsto\Y(\Bbbk,X).
\]
More precisely, for $X\in\ED^b(\mathrm{mod}\text{-}H)$, \cite[Proposition 3.11]{ChWa1} gives a quasi-isomorphism
\[
F(X)=\operatorname{Hom}_H(\Y(H,\Bbbk),\Y(H,X))
\xleftarrow{\simeq}
\Y(\Bbbk,X).
\]
Under this identification, the lax monoidal structure is represented by
\[
\eta_{X,Y}\colon
\Y(\Bbbk,X)\otimes^{\mathbb L}_{\Y(\Bbbk,\Bbbk)}\Y(\Bbbk,Y)
\longrightarrow
\Y(\Bbbk,X\otimes Y),
\]
for $X,Y\in\ED^b(\mathrm{mod}\text{-}H)$.
\end{rem}

\section{Examples and Hopf-algebra phenomena}\label{section:example}
We conclude with three examples that isolate the main features of the construction. The first is a consistency check: for a graded-commutative algebra with the trivial brace structure, the product $\boxtimes$ is the usual derived tensor product. The second shows that, for a non-local Hopf algebra, the Koszul duality functor can be genuinely lax monoidal rather than strong monoidal. The third computes the braces for an elementary $2$-group and shows that the answer depends on the coproduct, even when the underlying algebra and the dg Yoneda algebra are unchanged.

\subsection{The trivial brace structure and graded-commutative algebras}\label{subsection:commutative}
Let $R$ be a graded-commutative $\Bbbk$-algebra. It carries a ``trivial'' brace $B_\infty$-structure $(R,m^2;\mu^{p,q})$: the only nonzero $A_\infty$-operation is the multiplication $m^2$, and
\[
\mu^{0,0}=0,
\qquad
\mu^{0,1}=\id_{sR}=\mu^{1,0},
\qquad
\mu^{p,q}=0
\quad\text{for }(p,q)\notin\{(0,1),(1,0)\}.
\]
See \cite{LoRo} for related constructions.

\begin{rem}
Conversely, suppose that $(A,\{m^n\};\{\mu^{p,q}\})$ is a trivial $B_\infty$-algebra in the above sense: $m^n=0$ for $n\neq2$, and $\mu^{p,q}=0$ unless $(p,q)=(0,1)$ or $(1,0)$. Then Remark \ref{rem:commtativebrace} implies that $(A,m^2)$ is a graded-commutative associative algebra; see \cite[Subsection 3.2]{Kad}.
\end{rem}

For a graded-commutative dg algebra, every right dg module $M$ has its familiar symmetric bimodule structure
\[
a\cdot x\cdot b=(-1)^{|x||a|}xab
\qquad
(a,b\in R,\ x\in M).
\]
Accordingly, $\ED(R)$ has the usual monoidal structure given by the derived tensor product $\otimes_R^{\mathbb L}$. We now verify that the product $\boxtimes_R$ constructed from the trivial brace structure recovers exactly this familiar tensor product.

Because $\mu^{1,0}=\mu^{0,1}=\id_{sR}$, the extended product $\widehat\mu$ on the tensor coalgebra $T^c(sR)$ is the shuffle product; see \cite[Section 1.3]{LoRo}. Explicitly,
\begin{align}\label{align:shfulle}
\widehat\mu
(sa_1\otimes\dotsb\otimes sa_p\bigotimes sa_{p+1}\otimes\dotsb\otimes sa_{p+q})
 =\sum(-1)^\eta sa_{i_1}\otimes\dotsb\otimes sa_{i_{p+q}},
\end{align}
where the sum runs over all $(p,q)$-shuffles $(i_1,\dotsc,i_{p+q})$ and $(-1)^\eta$ is the Koszul sign. The antipode is tensor reversal:
\[
S(sa_1\otimes\dotsb\otimes sa_n)
 =(-1)^\epsilon sa_n\otimes\dotsb\otimes sa_1,
\]
where
\[
\epsilon=\sum_{i=1}^{n}|sa_i|(|sa_1|+\dotsb+|sa_{i-1}|).
\]

Let $(M,\rho_M)$ be a right dg $R$-module, and recall the induced $A_\infty$-bimodule
\[
\iota(M)=(M,\beta_M^{p,q})
\]
from Proposition \ref{prop:dualitybinfinity}.

\begin{prop}\label{prop:commuativeiota}
For every right dg $R$-module $M$, the induced $A_\infty$-bimodule $\iota(M)$ is the symmetric dg $R$-bimodule associated with $M$. Consequently, there is a canonical identification of monoidal triangulated categories
\[
(\ED(R),\boxtimes_R)\simeq(\ED(R),\otimes_R^{\mathbb L}).
\]
\end{prop}
\begin{proof}
The structure maps of $\iota(M)$ are
\begin{align}\label{align:betapqcommutative}
\begin{aligned}
&\beta_M^{p,q}(sa_1\otimes\cdots\otimes sa_p\otimes sx\otimes sb_1\otimes\cdots\otimes sb_q)\\
={}&\rho_M\left(
sx\otimes
\widehat\mu^{\opp}
\bigl(S(sa_1\otimes\cdots\otimes sa_p)\bigotimes
sb_1\otimes\cdots\otimes sb_q\bigr)
\right)\\
={}&(-1)^\epsilon
\rho_M^{p+q+1}\left(
sx\otimes
\widehat\mu
\bigl(sb_1\otimes\cdots\otimes sb_q\bigotimes
sa_p\otimes\cdots\otimes sa_1\bigr)
\right),
\end{aligned}
\end{align}
where $\widehat\mu$ is the shuffle product in \eqref{align:shfulle}. Since $M$ is a dg module, $\rho_M^i=0$ for $i>2$. Hence
\[
\beta_M^{p,q}=0
\qquad\text{whenever }p+q>1,
\]
so $\iota(M)$ is an ordinary dg bimodule. Its remaining structure maps are
\[
\beta_M^{0,0}=\rho_M^1,
\qquad
\beta_M^{0,1}=\rho_M^2,
\qquad
\beta_M^{1,0}(sa\otimes sm)
 =(-1)^{|sm||sa|+1}\rho_M^2(sm\otimes sa).
\]
Passing to the unshifted convention of Example \ref{exm:shiftedandunshifted}, the induced left action is
\[
a\circ m
 =(-1)^{|a|}\beta_M^{1,0}(sa\otimes sm)
 =(-1)^{|a||m|}m\circ a.
\]
This is precisely the symmetric left action.

For two right dg $R$-modules $M$ and $N$, we therefore have
\[
M\boxtimes_R N
 =M\overset{\infty}{\otimes}_R\iota(N)
 \simeq M\otimes_R^{\mathbb L}\iota(N)
 \simeq M\otimes_R^{\mathbb L}N.
\]
The first equality is the definition of $\boxtimes_R$, the second uses the comparison between the $A_\infty$-tensor product and the derived tensor product for dg modules, and the third uses $\iota(N)=N$ as symmetric bimodules.
\end{proof}

\begin{rem}\label{rem:iotanotful}
Even in this elementary case, the induction functor need not be full. Here
\[
\iota\colon\ED^{\RR}(R)\longrightarrow\ED^{\BB}(R)
\]
is the functor that equips a right dg module with its symmetric bimodule structure. If it were full, then for every $i\geq0$ the natural map
\[
\Ext_R^i(R,R)\longrightarrow\Ext_{R\text{-}R}^i(R,R)
\]
would be an isomorphism. The left-hand side vanishes for $i>0$, whereas the right-hand side is Hochschild cohomology and is generally nonzero; for example, take $R=\Bbbk[y]$ with $|y|=0$. Thus the failure of fullness is already visible in the graded-commutative case.
\end{rem}

\subsection{A non-local Hopf algebra with radical square zero}\label{subsec:radicalsquarezero}
We next give a four-dimensional example showing that the lax monoidal structure of Theorem \ref{thm:monoidalfunctor} need not be strong on all of $\EK(\mathrm{Inj}\text{-}H)$.

Let $H=\Bbbk Q/I$, where
\[
\xymatrix@C=5em@R=5em{
0 \ar@/^1pc/[r]^{\alpha}& 1 \ar@/^1pc/[l]^{\beta}
}
\]
and $I=(\alpha\beta,\beta\alpha)$. Thus
\[
\dim_\Bbbk H=4,
\qquad
1_H=e_0+e_1,
\]
and the Jacobson radical, spanned by $\alpha$ and $\beta$, has square zero.

This algebra admits the following Hopf structure; see \cite[Section 4]{YL}:
\begin{align}\label{algin:coprdocutradical}
\begin{aligned}
\Delta(e_0)&=e_0\otimes e_0+e_1\otimes e_1,\\
\Delta(e_1)&=e_0\otimes e_1+e_1\otimes e_0,\\
\Delta(\alpha)&=e_0\otimes\alpha+e_1\otimes\beta+\alpha\otimes e_0-\beta\otimes e_1,\\
\Delta(\beta)&=e_0\otimes\beta+e_1\otimes\alpha+\beta\otimes e_0-\alpha\otimes e_1,
\end{aligned}
\end{align}
with counit
\[
\varepsilon(e_0)=1,
\qquad
\varepsilon(e_1)=\varepsilon(\alpha)=\varepsilon(\beta)=0,
\]
and antipode
\[
S(e_0)=e_0,\qquad
S(e_1)=e_1,\qquad
S(\alpha)=\beta,\qquad
S(\beta)=-\alpha.
\]

The trivial module is the simple module
\[
S_0=\Bbbk e_0.
\]
Write $P_i=e_iH$ for the two indecomposable projective modules. The lax monoidal structure gives a morphism
\[
\eta_{P_0,P_0}\colon
\Y(S_0,P_0)\boxtimes_{\Y(S_0,S_0)}\Y(S_0,P_0)
\longrightarrow
\Y(S_0,P_0\otimes P_0)
\]
in $\operatorname{per}(\Y(S_0,S_0))$.

\begin{prop}\label{prop:laxmonoidal}
The morphism $\eta_{P_0,P_0}$ is not an isomorphism.
\end{prop}
\begin{proof}
A finite-dimensional Hopf algebra is self-injective, so $P_0$ and $P_1$ are injective. Directly from the module structure,
\[
\Ext_H^*(S_0,P_0)=\Hom_H(S_0,P_0)=0,
\qquad
\Ext_H^*(S_0,P_1)=\Hom_H(S_0,P_1)=\Bbbk.
\]
Consequently,
\[
\Y(S_0,P_0)\simeq0,
\qquad
\Y(S_0,P_1)\not\simeq0
\]
in $\operatorname{per}(\Y(S_0,S_0))$. On the other hand, the tensor-product rule for $H$ gives
\[
P_0\otimes P_0\simeq P_0\oplus P_1
\]
as right $H$-modules; see \cite[Section 4]{YL}. Therefore
\[
\Y(S_0,P_0\otimes P_0)
 \simeq\Y(S_0,P_0\oplus P_1)
 \simeq\Y(S_0,P_1)
 \not\simeq0.
\]
The source of $\eta_{P_0,P_0}$ is zero, whereas its target is nonzero. Hence $\eta_{P_0,P_0}$ cannot be an isomorphism.
\end{proof}

It is useful to compare this failure with the Koszul dual algebra itself. In this example the Yoneda dg algebra is formal, and even its brace structure is quasi-isomorphic to the trivial one.

\begin{prop}\label{prop:braceradical}
The brace $B_\infty$-algebra $\Y(S_0,S_0)$ is quasi-isomorphic to the trivial brace $B_\infty$-algebra on
\[
\Bbbk[y],
\qquad
|y|=2.
\]
Equivalently, the underlying dg algebra is $\Bbbk[y]$ with zero differential, and
\[
\mu^{p,q}=0
\qquad
\text{for }(p,q)\notin\{(1,0),(0,1)\}.
\]
\end{prop}
\begin{proof}
Let
\[
E_0=\Bbbk Q_0=\Bbbk e_0\oplus\Bbbk e_1\subset H.
\]
This is a semisimple Hopf subalgebra. By Remark \ref{rem:relativeyoneda}, the normalized $E_0$-relative Yoneda complex $\overline\Y(S_0,S_0)$ is a brace $B_\infty$-algebra and the inclusion
\[
\overline\Y(S_0,S_0)\hookrightarrow\Y(S_0,S_0)
\]
is a brace $B_\infty$-quasi-isomorphism.

Since
\[
\overline H\simeq\Bbbk\alpha\oplus\Bbbk\beta=\Bbbk Q_1
\qquad\text{and}\qquad
(s\overline H)^{\otimes_{E_0}n}\simeq s^n\Bbbk Q_n,
\]
we obtain
\[
\overline\Y_n(S_0,S_0)
 =\operatorname{Hom}_{E_0\text{-}E_0}
 \bigl(\Bbbk e_0\otimes_{E_0}\Bbbk Q_n,\Bbbk e_0\bigr)
 \simeq
 \begin{cases}
 \operatorname{Span}_\Bbbk\{(\alpha\beta\cdots\alpha\beta,e_0)\},
   & n\text{ even},\\
 0,& n\text{ odd}.
 \end{cases}
\]
For even $n$, the displayed basis vector is the cochain dual to the unique alternating path of length $n$ from vertex $0$ back to vertex $0$.

Let $y$ be the class represented by $(\alpha\beta,e_0)$. Concatenation of alternating paths gives an isomorphism of dg algebras
\[
\overline\Y(S_0,S_0)\cong\Bbbk[y],
\qquad
|y|=2,
\]
with zero differential. Finally, the coproduct formulas in \eqref{algin:coprdocutradical}, together with \eqref{align:mupqformular}, show that
\[
\mu^{1,q}=0
\qquad(q>0)
\]
on the normalized complex. Hence its brace structure is trivial.
\end{proof}

Thus the obstruction in Proposition \ref{prop:laxmonoidal} is not caused by a complicated Koszul dual algebra: the target is equivalent to the ordinary monoidal category $\ED(\Bbbk[y])$. Rather, it records that the trivial module does not generate all of $\EK(\mathrm{Inj}\text{-}H)$ when $H$ is non-local.

\subsection{Elementary $2$-groups and dependence on the coproduct}
We finish with an example in which the underlying algebra is fixed but two different Hopf structures produce different brace operations.

Let $G=C_2=\{e,\sigma\}$ and assume that $\operatorname{char}(\Bbbk)=2$. Setting
\[
x=\sigma-e
\]
gives an algebra isomorphism
\[
H:=\Bbbk[x]/(x^2)\xrightarrow{\simeq}\Bbbk G.
\]
Transporting the group-algebra Hopf structure to $H$, we obtain
\begin{align}\label{align:coproductkx}
\Delta(x)=1\otimes x+x\otimes1+x\otimes x,
\end{align}
together with
\[
\varepsilon(x)=0,
\qquad
S(x)=x.
\]
The extra term $x\otimes x$ reflects the fact that $1+x=\sigma$ is group-like.

Let
\[
\overline H=H/(\Bbbk\cdot1).
\]
It is one-dimensional, with basis the class of $x$. The normalized Yoneda complex of Remark \ref{rem:relativeyoneda} has
\[
\overline\Y_n(\Bbbk,\Bbbk)
 =\operatorname{Hom}\bigl((s\overline H)^{\otimes n},\Bbbk\bigr),
\]
which is one-dimensional in every degree. Concatenation identifies the underlying dg algebra with
\begin{align}\label{align:isokx}
\overline\Y(\Bbbk,\Bbbk)\cong\Bbbk[y],
\qquad
|y|=1,
\end{align}
and the differential is zero. The coproduct in \eqref{align:coproductkx}, however, gives nontrivial higher braces.

\begin{prop}\label{prop:dualnumberbrace}
Under the identification \eqref{align:isokx}, the brace operations are
\[
\mu^{1,q}
\bigl(sy^{n_0}\bigotimes sy^{n_1}\otimes\dotsb\otimes sy^{n_q}\bigr)
 =
\binom{n_0}{q}n_1n_2\cdots n_q\,
sy^{n_0+n_1+\dotsb+n_q-q},
\]
for $n_0,\dotsc,n_q\geq0$, with the coefficient interpreted in $\Bbbk$. In particular, the operation vanishes if $n_0<q$ or if $n_i=0$ for some $1\leq i\leq q$.
\end{prop}
\begin{proof}
Let
\[
f_n\in\operatorname{Hom}\bigl((s\overline H)^{\otimes n},\Bbbk\bigr)
\]
be determined by
\[
f_n((sx)^{\otimes n})=1_\Bbbk.
\]
Under \eqref{align:isokx}, the cochain $f_n$ corresponds to $y^n$. We have the following identities 
\begin{align}\label{align:bracekx}
\begin{aligned}
&\mu^{1,q}
(sf_{n_0}\bigotimes sf_{n_1}\smallotimes\dotsb\smallotimes sf_{n_q})
\bigl((sx)^{\smallotimes n_0+\dotsb+n_q-q}\bigr)\\
={}&
\sum f_{n_0}
\bigl(
\id^{\smallotimes i_1}\smallotimes s\widetilde f_{n_1}
\smallotimes\id^{\smallotimes i_2}\smallotimes\dotsb
\smallotimes\id^{\smallotimes i_q}\smallotimes s\widetilde f_{n_q}
\smallotimes\id^{\smallotimes i_{q+1}}
\bigr)
\bigl((sx)^{\smallotimes n_0+\dotsb+n_q-q}\bigr)\\
={}&
\sum n_1n_2\cdots n_q\,
f_{n_0}
\bigl(
(sx)^{\smallotimes i_1}\smallotimes sx
\smallotimes(sx)^{\smallotimes i_2}\smallotimes\dotsb
\smallotimes(sx)^{\smallotimes i_q}\smallotimes sx
\smallotimes(sx)^{\smallotimes i_{q+1}}
\bigr)\\
={}&
\binom{n_0}{q}n_1n_2\cdots n_q,
\end{aligned}
\end{align}
There are no signs because $\operatorname{char}(\Bbbk)=2$. Since $f_{n_0+\dotsb+n_q-q}$ evaluates to $1$ on the displayed tensor, this scalar computation is exactly the identity asserted in the proposition.

The first identity in \eqref{align:bracekx} follows from Formula \eqref{align:mupqformular}. Let us explain the the second and third identities in  \eqref{align:bracekx}. For each $i$, we have 
\begin{align}\label{align:fwidetilde}
\begin{aligned}
\widetilde f_{n_i}((sx)^{\otimes n_i})
 &=
 f_{n_i}(sx^{(1)}\otimes\dotsb\otimes sx^{(1)})
 \,\overline{x^{(2)}\cdots x^{(2)}}\\
 &=
 \sum_{j=1}^{n_i}
 f_{n_i}(sx\otimes\dotsb\otimes sx)
 \underbrace{1\cdots1}_{j-1}x
 \underbrace{1\cdots1}_{n_i-j}\\
 &=n_i x.
\end{aligned}
\end{align}
Indeed, by \eqref{align:coproductkx}, a nonzero contribution requires all first tensor factors to be $x$ and exactly one second tensor factor to be $x$; terms with two such factors vanish because $x^2=0$. Thus the $i$th insertion contributes the factor $n_i$.

Finally, since the sum in \eqref{align:bracekx} is  taken over all sequences of nonnegative integers $\{i_1,i_2,\dotsc,i_{q+1}\}$  such that $i_1+\dotsb+i_{q+1}=n_0-q$.
The cardinality is $\binom{n_0}{q}$. Multiplying the contributions proves \eqref{align:bracekx}.
\end{proof}

\begin{rem}
For an elementary abelian $2$-group
\[
G=(\mathbb Z/2\mathbb Z)^r,
\]
the group algebra is the tensor product of $r$ copies of the rank-one algebra above. Accordingly, the brace $B_\infty$-algebra $\Y(\Bbbk,\Bbbk)$ is obtained from the corresponding tensor product of the rank-one brace algebras.

The same associative algebra $\Bbbk[x]/(x^2)$ also occurs as a restricted universal enveloping algebra. Its Hopf structure has primitive generator
\[
\Delta(x)=1\otimes x+x\otimes1;
\]
see \cite[Section 8]{BIK}. For this coproduct, the $H$-valued cochain in \eqref{align:fwidetilde} vanishes, so all higher operations $\mu^{1,q}$ with $q\geq1$ are zero. Hence the normalized Yoneda algebra is the trivial brace $B_\infty$-algebra of Subsection \ref{subsection:commutative}.

We therefore obtain two brace $B_\infty$-structures on the same graded algebra $\Bbbk[y]$: a nontrivial one from the group-algebra coproduct and a trivial one from the primitive coproduct. They induce different monoidal structures on the same underlying derived category $\ED(\Bbbk[y])$; compare \cite[Remark 14.3]{BeKr} and \cite[Theorem 3.6]{Hov11}. This example makes explicit that the monoidal structure remembers the Hopf coproduct, not merely the associative algebra structure of $H$ or the graded algebra $\Ext_H^*(\Bbbk,\Bbbk)$.
\end{rem}

\vskip 5pt

\noindent {\bf Acknowledgements.}   The authors are very grateful to Xiao-Wu Chen, Hua-Lin Huang, Bernhard Keller, Henning Krause, Jing Yu and Guodong Zhou for many useful comments and discussions. This work was supported by the National Key R\&D Program of China (2024YFA1013802, 2024YFA1013803) and the NSFC (Nos.~12271243, 12371043).

\vskip 5pt 


\end{document}